\documentclass[10pt]{amsart}
\usepackage[all]{xy}\usepackage{stmaryrd,mathrsfs,amsmath,amssymb,shuffle,pgf,tikz}
\usepackage[unicode=true, pdfborder={0 0 0}, bookmarksdepth=-1]{hyperref}

\date{\today}

\newtheorem{thm}{Theorem}[section]
\newtheorem*{thmA}{Theorem A}
\newtheorem*{thmB}{Theorem B}
\newtheorem*{thmC}{Theorem C}
\newtheorem{lem}[thm]{Lemma}
\newtheorem{lemdef}[thm]{Lemma-Definition}
\newtheorem{cor}[thm]{Corollary}
\newtheorem{prop}[thm]{Proposition}

{\theoremstyle{definition} \newtheorem{rem}[thm]{Remark}}
{\theoremstyle{definition} \newtheorem{defn}[thm]{Definition}}

\newcommand{\sha}{\shuffle}

\numberwithin{equation}{subsection}
\numberwithin{figure}{section}
\begin{document}

\baselineskip 16pt 

\title[Analogues of hyperlogarithm functions on affine complex curves]{Analogues of hyperlogarithm functions \\ on affine complex curves}

\author{Benjamin Enriquez}
\address{IRMA (UMR 7501) et D\'epartement de Math\'ematiques, Universit\'e de Strasbourg, 7 rue Ren\'e-Descartes, 67084 Strasbourg (France)}
\email{b.enriquez@math.unistra.fr}
\author{Federico Zerbini}
\address{University of Oxford, Mathematical Institute, Andrew Wiles Building, Radcliffe Observatory Quarter (550), Woodstock Road, Oxford, OX2 6GG (UK)}
\email{federico.zerbini@maths.ox.ac.uk}

\begin{abstract}
For $C$ a smooth affine complex curve, there is a unique minimal 
unital subalgebra $A_C$ of the algebra $\mathcal O_{hol}(\tilde C)$
of holomorphic functions on its universal cover $\tilde C$, which is 
stable under all the operations $f\mapsto \int f\omega$, for $\omega$ in the space $\Omega(C)$ of regular differentials on $C$. 
We identify $A_C$ with the image of the iterated integration map $I_{x_0} : \mathrm{Sh}(\Omega(C))\to\mathcal O_{hol}(\tilde C)$ 
based at any point $x_0$ of $\tilde C$ (here $\mathrm{Sh}(-)$ denotes the shuffle algebra of a vector space), as well as with  
the unipotent part, with respect to the action of $\mathrm{Aut}(\tilde C/C)$, of a subalgebra of $\mathcal O_{hol}(\tilde C)$ of 
moderate growth functions. 
We show that any regular Maurer-Cartan (MC) element $J$ on $C$ with values in the topologically free Lie algebra over 
$\mathrm H^1_{\mathrm{dR}}(C)^*$ gives rise to an isomorphism of $A_C$ with $\mathcal O(C) \otimes\mathrm{Sh}(\mathrm H^1_{\mathrm{dR}}
(C))$, where $\mathcal O(C)$ is the algebra of regular functions on $C$, leading to the assignment of a subalgebra 
$\mathcal H_C(J)$ of~$A_C$ (isomorphic to $\mathrm{Sh}(\mathrm H^1_{\mathrm{dR}}(C))$)
to any MC element. We also associate a MC element $J_\sigma$ to each section 
$\sigma$ of the projection $\Omega(C)\to \mathrm H^1_{\mathrm{dR}}(C)$;
when $C$ has genus $0$, we exhibit a particular section $\sigma_0$ for which 
$\mathcal H_C(J_{\sigma_0})$ is the algebra of hyperlogarithm functions (Poincaré, Lappo-Danilevsky).
\end{abstract}
\maketitle

{\small \setcounter{tocdepth}{2}
\tableofcontents}

\section{Introduction}

\subsection{The context}\label{sect:context}

To an inclusion $\mathcal O \subset \tilde{\mathcal O}$ of unital complex commutative algebras and a derivation $\partial$  
of $\tilde{\mathcal O}$ which is both surjective and with $\mathrm{ker}(\partial)=\mathbb C$, one may associate the smallest 
subalgebra of $\tilde{\mathcal O}$ which contains $\mathcal O$ and is stable under the antiderivation operation 
$\tilde{\mathcal O}\ni f\mapsto \partial^{-1}(f) \subset \tilde{\mathcal O}$. Two instances of this construction 
were studied in detail in the literature: 

$\bullet$ $\tilde{\mathcal O}=\mathcal O_{hol}(\mathfrak H)$ is the algebra of holomorphic functions on the 
complex upper half-plane $\mathfrak H=\{\tau \in \mathbb C\,|\,\mathrm{Im}(\tau)>0\}$, $\mathcal O=\mathrm{QM}_*$ is the algebra of 
quasi-modular forms for $\mathrm{SL}_2(\mathbb Z)$, and $\partial=d/d\tau$ (see \cite{Ma});

$\bullet$ $\mathcal O=\mathbb C[z,1/(z-s),s \in S_\infty]=:\mathcal O(\mathbb P^1_{\mathbb C}\smallsetminus S)$, 
where $S\subset\mathbb P^1_{\mathbb C}$ is a finite subset with $S\ni\infty$, $S_\infty$ denotes $S\smallsetminus\{\infty\}$,  $\tilde{\mathcal O}$ is the algebra of holomorphic functions on a 
universal cover of $\mathbb P^1_{\mathbb C}\smallsetminus S$, and $\partial=d/dz$ (see \cite{Br:these}).

In both cases, precise results were obtained on the structure of the said smallest subalgebra. 
Let us describe the results of the second case in more detail.  In that case, the conditions on the looked-for algebra 
are equivalent to requiring it being both unital and stable
%under the operation
%$f\mapsto (z\mapsto \int_{z_0}^z f(u)du)=:\int_{z_0}f$ (here $z_0$ is a fixed point in $\mathbb C\smallsetminus S_\infty$). It is 
%equivalent
%to require that this algebra be both unital and stable 
under all the operations $f\mapsto\int_{z_0} f\omega
:=(z\mapsto \int_{z_0}^z f\omega)$, where~$\omega$ runs over all the regular differentials on 
$\mathbb P^1_{\mathbb C}\smallsetminus S$; indeed, the latter condition implies that the algebra contains the functions 
$\int_{z_0} df$, where $f$ runs over $\mathcal O(\mathbb P^1_{\mathbb C}\smallsetminus S)$. 

Such an algebra necessarily contains all the iterated integrals of the differentials $d\log(z-s)$, $s\in S_\infty$, which 
are the hyperlogarithm (HL) functions $L_w$ indexed by $w\in \hat S_\infty^*$ (where 
$\hat S_\infty^*:=\sqcup_{n \geq 0} \hat S_\infty^n$ is the set of words in $\hat S_\infty$, which is $S_\infty$ viewed as an 
abstract set). The generating series $\mathbf L:=\sum_w L_w\cdot w$ is a 
multivalued holomorphic function on $\mathbb P^1_{\mathbb C}\smallsetminus S$ 
with values in the group of group-like elements of the algebra of noncommutative formal series
$\mathbb C\langle\langle \hat S_\infty\rangle\rangle$, 
such that $d\mathbf L(z)=\mathbf L(z)\cdot\sum_s (\hat s\cdot d\log(z-s))$. 
It was proven in~\cite{Br:these}, Cor. 5.6, that: 
\begin{itemize}
    \item the algebra $A_{\mathbb P^1_{\mathbb C}\smallsetminus S}:=\mathcal O(\mathbb P^1_{\mathbb C}\smallsetminus S)
    [L_w,w\in  \hat S_\infty^*]$ is stable under antiderivation, so that $A_{\mathbb P^1_{\mathbb C}\smallsetminus S}$ is the smallest 
    (for the inclusion) extension of $\mathcal O(\mathbb P^1_{\mathbb C}\smallsetminus S)$ with this property; 
    \item the map $\mathcal O(\mathbb P^1_{\mathbb C}\smallsetminus S)\otimes\mathrm{Sh}(\mathbb C\hat S_\infty)\to 
    A_{\mathbb P^1_{\mathbb C}\smallsetminus S}$, 
    $f\otimes w\mapsto f\cdot L_w$ is an algebra isomorphism, where $\mathrm{Sh}(V)$ is the shuffle algebra associated with 
    a vector space $V$; in particular, the family $(L_w)_w$ is linearly independent over
    $\mathcal O(\mathbb P^1_{\mathbb C}\smallsetminus S)$ (this was also proved in \cite{DDMS}).
\end{itemize}

The HL functions, and hence all the functions of $A_{\mathbb P^1_{\mathbb C}\smallsetminus S}$, have unipotent monodromies along the paths 
encircling 
the points of $S$, and one can show that $A_{\mathbb P^1_{\mathbb C}\smallsetminus S}$ is a union of unipotent modules (i.e. iterated extensions 
of
the trivial module) over $\pi_1(\mathbb P^1_{\mathbb C}\smallsetminus S)$.

The HL functions were introduced in \cite{Po}, motivated by monodromy computations. They were later applied in 
\cite{LD} to the Riemann-Hilbert problem, and subsequently in \cite{Br:these} to the identification of a 
set of periods arising from
the moduli space of marked stable genus-zero curves with the set of multiple zeta values (Goncharov-Manin conjecture). The HL techniques 
of \cite{Br:these} led 
in \cite{Pa:thesis} to an algorithm which can be used to express, in physics, a large class of Feynman integrals\footnote{Feynman integrals are a useful tool to obtain approximations 
of scattering amplitudes, which predict in quantum field theory the probability of interactions of elementary particles.} in terms of HLs; this was implemented in the software program \texttt{HyperInt}. 

Similar questions were studied replacing $\mathbb P^1_{\mathbb C}$ by a curve of genus one. 
To an elliptic curve~$\mathcal E$, one attaches an algebra~$\mathcal A_3$ containing the function field 
of $\mathcal E$ (see \cite{BDDT1}, three lines before (3.35)) using iterated integration. In {\it loc. cit,} §6, it is 
proved that $\mathcal A_3$ is stable under $f\mapsto\int_{z_0} f\omega_0$, where $\omega_0$ is a fixed non-zero regular differential over 
$\mathcal E$ and $z_0$ is any point in $\mathcal E$.  
One can derive from this the construction, for any finite subset $S$ of $\mathcal E$, of an 
algebra containing the algebra of regular 
functions on $\mathcal E\smallsetminus S$, which 
is stable under $f\mapsto\int_{z_0} f\omega_0$; this algebra is therefore stable under the operations 
$f\mapsto \int_{z_0} f \omega$, where 
$\omega$ runs over all the regular differentials on $\mathcal E\smallsetminus S$. 
Similarly to the genus-zero case, the functions from $\mathcal A_3$ arise naturally in the computation of Feynman 
integrals (see \cite{BDDT2}). 

It is a natural question to construct analogues of the HL functions associated to an arbitrary affine curve $C$. Such functions 
are likely to find an application in physics also when the genus of the curve is higher than one, such as for instance to compute hyperelliptic Feynman integrals (see \cite{MMPPW}), or the genus-two contribution to string 
theory amplitudes (see \cite{DGP}). 

In order to treat this problem, we fix a universal cover $p : \tilde C\to C$ and introduce the notion of {\it minimal stable 
subalgebra} (MSSA) of the algebra of holomorphic functions $\mathcal O_{hol}(\tilde C)$ of $\tilde C$ as follows: we call {\it stable 
subalgebra (SSA) of $\mathcal O_{hol}(\tilde C)$} 
a unital subalgebra $A$ of the algebra $\mathcal O_{hol}(\tilde C)$ such that 
for any $f\in A$, regular differential $\omega$ on $C$ and $z_0\in\tilde C$, the function $\int_{z_0} 
f\omega$ belongs to~$A$. The intersection $A_C$ of all SSA of $\mathcal O_{hol}(\tilde C)$ is again a SSA 
which is minimal for the inclusion, and which we call the MSSA of $\mathcal O_{hol}(\tilde C)$. 

The present paper is devoted to the study of $A_C$. We introduce the notion of a Maurer-Cartan element associated 
with the curve $C$, and show each such element gives rise via iterated integration to an algebra isomorphism 
$A_C \simeq\mathcal O(C) \otimes \mathrm{Sh}(\mathrm H^1_{\mathrm{dR}}(C))$. We also show that $A_C$ is a union of 
unipotent $\pi_1(C)$-modules, contained in an algebra of moderate growth functions over~$\tilde C$, and is maximal with respect to 
this property. All this shows that the properties of~$A_C$ are generalizations of those of 
$A_{\mathbb P^1_{\mathbb C}\smallsetminus S}$; the isomorphism $A_C \simeq\mathcal O(C) \otimes 
\mathrm{Sh}(\mathrm H^1_{\mathrm{dR}}(C))$ is also an analogue of the main result of \cite{Ma}. 
In the companion paper \cite{EZ}, we make $A_C$ explicit when 
$C=\mathcal E\smallsetminus S$, with $S$ a finite subset of an elliptic curve~$\mathcal E$, 
and we explicitly relate~$A_C$ with the algebra $\mathcal A_3$ from \cite{BDDT1}.  
The recent work \cite{DHS}, which introduces non-holomorphic variants of HL functions over one-punctured curves $C$ of arbitrary genus, could hopefully be related to the present work. 

\subsection{The main results}\label{pres:results:03:11}

\subsubsection{Conventions} 

The following conventions will be adopted throughout the paper. 
The base field of all the algebraic structures (vector spaces, Lie, Hopf or associative algebras, etc.) 
is $\mathbb C$. We denote\footnote{Except in Rem. \ref{rem:1:4:1611}(a), in \S\ref{sect:92:0711} and in the second half of 
\S\ref{sect:5:5:1611}, where $C$ takes a particular value.}
 by $C$ a smooth complex affine algebraic curve, as well as the underlying Riemann surface, by   
$p:\tilde C\to C$ a universal cover, and by $\mathcal O(C)$ the algebra of regular functions on $C$. 
Then $p^*: \mathcal O(C)\to \mathcal O_{hol}(\tilde C)$ is an injective algebra morphism. We denote by  $\Omega(C)$ the space of regular differentials on $C$, and we set 
$\mathrm H_C:=\Omega(C)/d\mathcal O(C)$ ($=\mathrm H^1_{\mathrm{dR}}(C)$ as $C$ is affine). 

\subsubsection{Maurer-Cartan (MC) elements and the associated isomorphisms}\label{sect:121:2710}

Denote by $\mathfrak g:=\mathbb L(\mathrm H_C^*)$ the free Lie algebra generated by $\mathrm H_C^*$; it is graded by 
the condition that $\mathrm H_C^*$ has degree 1, and we denote by $\hat{\mathfrak g}$ its degree completion. 

\begin{defn}\label{def:MC:0812}
(a) A \emph{Maurer-Cartan (MC) element} for $C$ is an element $J\in\Omega(C)\hat\otimes\,\hat{\mathfrak g}$. 

(b) $J$ is \emph{non-degenerate} iff $\mathrm{im}(J\in \Omega(C)\,\hat\otimes\,\hat{\mathfrak g}\to\mathrm{H}_C\otimes \mathrm{H}_C^*)=id$, the
map being given by the tensor product of the canonical projections. 

(c) $\mathrm{MC}(C)$ is the set of all MC elements for $C$, and 
$\mathrm{MC}_{nd}(C)$ is the subset of all non-degenerate elements. 
\end{defn}

Let $(J,x_0)\in\mathrm{MC}_{nd}(C)\times\tilde C$. One proves that there is a unique smooth function 
$\mathbf L_{J,x_0} : \tilde C\to \mathrm{exp}(\hat{\mathfrak g}):=\mathcal G((U\mathfrak g)^\wedge)$ 
(where $\mathcal G$ stands for the group of group-like elements of a topological Hopf algebra, and $(U\mathfrak g)^\wedge$
is the degree completion of the universal enveloping algebra of $\mathfrak g$) such that 
$d\mathbf L_{J,x_0}=\mathbf L_{J,x_0}\cdot J$ and $\mathbf L_{J,x_0}(x_0)=1$, which turns out to be holomorphic
(see Prop. \ref{prop:constr:L:0304}).  

Define then  $\tilde f_{J,x_0} : \mathrm{Sh}(\mathrm H_C)\to \mathcal O_{hol}(\tilde C)$ to be the map taking 
$a$ to the function 
\begin{equation}\label{def:tildef:0304}
\tilde f_{J,x_0}(a):=(\tilde C\ni x\mapsto \langle a,\mathbf L_{J,x_0}(x)\rangle\in\mathbb C),
\end{equation}
where $\langle,\rangle$ is the pairing 
$\mathrm{Sh}(\mathrm H_C)\times (U\mathfrak g)^\wedge\to\mathbb C$ induced by the composition 
$\mathrm{Sh}(\mathrm H_C)=\oplus_{n\geq 0}(U\mathfrak g)[n]^*\to 
(\prod_{n\geq 0}U\mathfrak g[n])^*=((U\mathfrak g)^\wedge)^*$. 

Similarly to the case of classical hyperlogarithms, $\mathbf L_{J,x_0}$ may be viewed as an element of 
$\mathcal O_{hol}(\tilde C) \hat\otimes (U\mathfrak g)^\wedge$. 
Hence $\mathbf L_{J,x_0}$ is a generating series of the image by $\tilde f_{J,x_0}$ of a basis of 
$\mathrm{Sh}(\mathrm H_C)$, which are multivalued functions on $C$ defined by \emph{iterated integrals}. 
By Lem.-Def. \ref{def:2:20:0908}, if $x\in \tilde C$ and $\omega_1,\ldots,\omega_k \in \Omega(C)$, 
then the iterated integral 
$\int_\gamma \omega_1\cdots\omega_k:=\int_{0\leq t_1\leq\ldots\leq t_k\leq 1}
\gamma^*\omega_1(t_1)  \wedge\cdots \wedge \gamma^*\omega_k(t_k)$ 
is independent of a path $\gamma$ from~$x_0$ to~$x$, and denoted by
$\int_{x_0}^x \omega_1\cdots\omega_k$. 

\begin{defn}\label{def:1:2:0311}
(a) $\Sigma_C$ denotes the set of sections $\sigma : \mathrm H_C\to\Omega(C)$ of the canonical projection. 

(b) $\sigma\mapsto J_\sigma$ is the map $\Sigma_C\to \mathrm{MC}_{nd}(C)$ such that 
$\sigma\mapsto J_\sigma:=\sum_i\sigma(h_i)\otimes h^i$, where $(h_i)_i$ is a basis of $\mathrm{H}_C$ and $(h^i)_i$ 
is the dual basis of $\mathrm{H}_C^*$.  
\end{defn}

\begin{lem}[see Lem. \ref{GR:ULJ}]\label{lem:1:3:2211} 
Let $(J,x_0)\in\mathrm{MC}_{nd}(C)\times\tilde C$.

(a) The map $\tilde f_{J,x_0} : \mathrm{Sh}(\mathrm H_C)\to \mathcal O_{hol}(\tilde C)$ is a morphism of algebras. 

(b) If $\sigma\in\Sigma_C$, $\tilde f_{J_\sigma,x_0}([h_1|\ldots|h_k])=(x\mapsto \int_{x_0}^x \sigma(h_1)\cdots \sigma(h_k))$,  
where $[h_1|\ldots|h_k]\in\mathrm{Sh}(\mathrm H_C)$ is the element corresponding to 
$h_1 \otimes \cdots\otimes h_k\in\mathrm H_C^{\otimes k}$. 
\end{lem}

\begin{thmA}[see \S\ref{sect:proof:of:thms}]\label{thm:A} 
(a) For any $J\in\mathrm{MC}_{nd}(C)$, the image 
$\mathrm{im}(\tilde f_{J,x_0} : \mathrm{Sh}(\mathrm H_C)\to\mathcal O_{hol}(\tilde C))$ is independent of $x_0\in\tilde C$; it will be 
denoted $\mathcal H_C(J)$. 

(b) Let $(J,x_0)\in\mathrm{MC}_{nd}(C)\times\tilde C$. The map $f_{J,x_0} : \mathrm{Sh}(\mathrm H_C)\otimes\mathcal O(C)
\to\mathcal O_{hol}(\tilde C)$, 
$a\otimes f\mapsto p^*(f)\cdot \tilde f_{J,x_0}(a)$ induces an algebra isomorphism 
$f_{J,x_0} : \mathrm{Sh}(\mathrm H_C)\otimes\mathcal O(C)\to A_C$. 
\end{thmA}

In particular, any $\mathbb C$-basis of $\mathcal H_C(J)$ (for example, the family $(\tilde f_{J,x_0}(w))_{w}$, 
where $w$ runs over a basis of $\mathrm{Sh}(\mathrm H_C)$) is linearly independent over $\mathcal O(C)$ and forms a basis of $A_C$ as a 
$\mathcal O(C)$-module.  

\begin{rem}\label{rem:1:4:1611}
(a) If $C=\mathbb P^1_{\mathbb C}\smallsetminus S$, then $\mathrm H_C\simeq\mathbb C\hat S_\infty$. 
A particular element of $\Sigma_C$ is $\sigma_0$ given by $\mathbb C\hat S_\infty\ni \hat s\mapsto d\log(z-s)
\in\Omega(\mathbb P^1_{\mathbb C}\smallsetminus S)$. Then 
$\mathcal H_{\mathbb P^1_{\mathbb C}\smallsetminus S}(J_{\sigma_0})$ is equal to $\mathbb C[L_w,w \in \hat S^*_\infty]$ 
(see \S\ref{sect:92:0711}). It follows that the algebras 
$\mathcal H_C(J)$, where $J\in\mathrm{MC}_{nd}(C)$, are generalizations of the algebra  of HL functions. 

(b)  While $\mathcal H_C(J)$ varies with $J$, Thm. A(b) says that 
the product $\mathcal O(C)\cdot \mathcal H_C(J)$ does not and is equal to $A_C$. 
\end{rem}

\subsubsection{Group aspects of the isomorphisms associated with the MC elements}

Let $\Gamma_C:=\mathrm{Aut}(\tilde C/C)$ be the automorphism group of the cover $p : \tilde C\to C$. 

%a) HA isos 

\begin{defn}\label{defn:1:5}
For $\Gamma$ a group, define $(\mathbb C\Gamma)'$ to be the subset of $(\mathbb C\Gamma)^*$ of all linear forms 
 which vanish on the union $\cup_{n\geq0}(\mathbb C\Gamma)_+^{n+1}$, where $(\mathbb C\Gamma)_+$
is the augmentation ideal of the group algebra~$\mathbb C\Gamma$.  
\end{defn}

If $\Gamma$ is finitely generated, it follows from Lem. \ref{lem:4:8:jeu} that 
$(\mathbb C\Gamma)'$ is a commutative Hopf algebra, 
which may be identified with the function algebra of the prounipotent completion $\Gamma^{\mathrm{un}}$ (see \S\ref{sect:D:2:0711}). 
This is in particular the case if $\Gamma=\Gamma_C$. 
The algebra $\mathrm{Sh}(\mathrm H_C)$ is also equipped with a commutative Hopf algebra structure, 
the coproduct $\Delta_{\mathrm{Sh}(\mathrm H_C)}$ being given by deconcatenation.

\begin{lem}[see Lem. \ref{HA:pairing} and Prop. \ref{prop:1508}]\label{lem:1:6:2211} 
For any $(J,x_0)\in\mathrm{MC}_{nd}(C)\times\tilde C$, the map
$$
p_{J,x_0} : \mathbb C\Gamma_C\times\mathrm{Sh}(\mathrm{H}_C)\to\mathbb C, \quad
\gamma\otimes a\mapsto \langle a,\mathbf L_{J,x_0}(\gamma x_0)\rangle
%=I_{x_0}( J_*(a))(\gamma\cdot x_0)
$$
is a Hopf algebra pairing. 
It gives rise to an isomorphism of commutative Hopf algebras 
$$
\nu(p_{J,x_0}) : \mathrm{Sh}(\mathrm{H}_C)\to (\mathbb C\Gamma_C)'. 
$$
\end{lem}
The relation of this result with Chen's ``$\pi_1$ de Rham theorem'' is discussed in Rem. \ref{rem:4:12:0712}. 

The group of $\mathbb C$-points of the spectrum of 
a commutative Hopf algebra $O$ is $\mathrm{Spec}(O)(\mathbb C):=\mathrm{Hom}(O,\mathbb C)$. 
Then\footnote{One checks that the bijection 
$\mathrm{Hom}_{\mathbb C-vec}(\mathrm{Sh}(\mathrm H_C),\mathbb C)
\simeq \hat T(\mathrm H_C^*)=(U\mathfrak g)^\wedge$ induces a bijection 
$\mathrm{Hom}_{\mathbb C-alg}(\mathrm{Sh}(\mathrm H_C),\mathbb C)
\simeq \mathcal G((U\mathfrak g)^\wedge)$.}
$\mathrm{Spec}(\mathrm{Sh}(\mathrm{H}_C))(\mathbb C)=\mathcal G((U\mathfrak g)^\wedge)
=\mathrm{exp}(\hat{\mathfrak g})$, 
and $\mathrm{Spec}((\mathbb C\Gamma_C)')(\mathbb C)
=\mathcal G((\mathbb C\Gamma_C)^\wedge)=\Gamma_C^{\mathrm{un}}(\mathbb C)$ (see \cite{BGF}, Thm. 3.224 and Appendix 
\ref{sect:D:2:0711}). The group isomorphism 
corresponding to $\nu(p_{J,x_0})$ is 
\begin{equation}\label{gp:iso:2209}
\Gamma_C^{\mathrm{un}}(\mathbb C)\to \mathrm{exp}(\hat{\mathfrak g}), \quad \gamma\to \mathbf L_{J,x_0}(\gamma x_0).    
\end{equation}
 
%b) HACA isos 

\begin{defn}[see Def. \ref{def:B1:0711}] 
A  {\it Hopf algebra with comodule-algebra} (HACA) is a pair $(O,A)$, with $O$ a Hopf algebra and 
$A$ an associative algebra, equipped with an algebra morphism $\Delta_A : A\to O\otimes A$, which turns $A$ into a comodule
over the coalgebra $O$. 
\end{defn}

The action of an algebraic group $G$ on a variety $V$ gives rise to a HACA $(\mathcal O(G),\mathcal O(V))$.
For any pair $(O,\mathbf a)$ of a commutative Hopf algebra $O$ and a commutative  
algebra $\mathbf a$, the pair $(O,O\otimes\mathbf a)$ is a HACA, with $\Delta_{O\otimes\mathbf a}:=\Delta_O\otimes id_{\mathbf a}$; 
it corresponds to the action of $G:=\mathrm{Spec}(O)$ on $V:=\mathrm{Spec}(O)\times\mathrm{Spec}(\mathbf a)$. 
In particular, the pair $(\mathrm{Sh}(\mathrm H_C),\mathrm{Sh}(\mathrm H_C)\otimes\mathcal O(C))$ is a HACA. 

\begin{lem}[see Lem. \ref{NEwlem}, Prop. \ref{lem:5:10:2811}, and Lem. \ref{lem:5:8:0307}]\footnote{
The combination of Lem. \ref{NEwlem} and Prop. \ref{lem:5:10:2811} implies that $A_C$ 
is a $\Gamma_C$-stable subalgebra of $F_\infty\mathcal O_{hol}(\tilde C)$ (the notation $F_\infty$ being as in 
Lem. \ref{lem:5:8:0307}), and therefore that $(\mathbb C\Gamma_C,A_C)$ is an object in $\mathbf{HAMA}_{fd}$ (see Def. 
\ref{def:HAMA:fd});  Lem. \ref{lem:1:8:2211} then follows from Lem. \ref{lem:5:8:0307}.
}\label{lem:1:8:2211} 
 The right action  $(f,\gamma) \mapsto f_{|\gamma}:=(x \mapsto f(\gamma x))$  
 of $\Gamma_C$ on $\mathcal O_{hol}(\tilde C)$ induces a HACA structure on $((\mathbb C\Gamma_C)',A_C)$. 
\end{lem}

% démo économique : A_C subset F_infty O_hol(tilde C) est Gamma_C invt // 

\begin{thmB}[see \S\ref{sect:proof:of:thms}]\label{thm:0:10:0510} 
Let $(J,x_0)\in\mathrm{MC}(C)\times\tilde C$. The pair 
$$
(\nu(p_{J,x_0}),f_{J,x_0}) : 
(\mathrm{Sh}(\mathrm H_C),\mathrm{Sh}(\mathrm H_C)\otimes\mathcal O(C)) \to ((\mathbb C\Gamma_C)',A_C)
$$
is a HACA isomorphism. 
\end{thmB}

To the HACA structure $(\mathrm{Sh}(\mathrm H_C),\mathrm{Sh}(\mathrm H_C)\otimes\mathcal O(C))$ (resp. $((\mathbb C\Gamma_C)',A_C)$), 
one associates an action of the group $\mathrm{exp}(\hat{\mathfrak g})$ (resp. $\Gamma^{\mathrm{un}}(\mathbb C)$) on 
$\mathrm{Sh}(\mathrm H_C)\otimes\mathcal O(C)$ (resp. on $A_C$).  Thm. B can then be translated into an equivariance statement: 
%d) rem : interpr groupique 
the algebra isomorphism $f_{J,x_0}$ is compatible with the group isomorphism 
\eqref{gp:iso:2209} and with the action of its source and target on the target and source of $f_{J,x_0}$. 

One can also introduce the notion of a connection over a HACA, which generalizes the notion of connection over a principal bundle  
in the case of the HACA $(\mathcal O(G),\mathcal O(V))$, with $V$ a principal $G$-bundle (see \S\ref{sect:6:0212}). We construct a 
connection $\nabla$ on the HACA $((\mathbb C\Gamma_C)',A_C)$ and 
compute its pull-back by $(\nu(p_{J,x_0}),f_{J,x_0})$, which is independent of $x_0$, for any $(J,x_0)\in \mathrm{MC}_{nd}(C) \times \tilde C$. % and denoted $\nabla_J$. 

\subsubsection{Isomorphisms of filtrations}

Let $\overline C$ be the smooth compactification of $C$, and  
$S:=\overline C\smallsetminus C$ be the complement of $C$ in $\overline C$, 
which is a finite set; then $\tilde C$ is a universal cover of $\tilde{\overline C}\smallsetminus p^{-1}(S)$. A function of 
$\mathcal O_{hol}(\tilde C)$ is called {\it moderate growth} if its 
growth at the neighborhood of $p^{-1}(S)$ is moderate in the sense of \cite{Ph}, §IX.1 (see Def. \ref{def:1:8:0208}). 

\begin{lem}[see Prop. \ref{lem:1:7:toto}]\label{lem:1:9:2211} 
 The subset $\mathcal O_{mod}(\tilde C)\subset \mathcal O_{hol}(\tilde C)$ of moderate growth functions is a subalgebra, 
 which is stable under the action of $\Gamma_C$. 
\end{lem}

One attaches a filtration of $\mathcal O_{mod}(\tilde C)$ to the action of 
$\Gamma_C$ on $\mathcal O_{mod}(\tilde C)$ as follows. 

\begin{defn}\label{def:f:infty:o:mod:0311}
For $n\geq 0$, we set $F_n\mathcal O_{mod}(\tilde C):=\{f\in \mathcal O_{mod}(\tilde C)|f_{|(\mathbb C\Gamma_C)_+^{n+1}}=0\}$. 
\end{defn}

\begin{lem}[see Prop. \ref{lem:2:12:1508} and Lem. \ref{LEM:HACA}]\label{lem:1:11:2211} 
$F_{\bullet}\mathcal O_{mod}(\tilde C)$ is an increasing algebra filtration of
$\mathcal O_{mod}(\tilde C)$ with $F_0\mathcal O_{mod}(\tilde C)=\mathcal O(C)$, stable under the
action of $\Gamma_C$. 
\end{lem}

Inspired by \cite{Chen}, we also define two ``differential'' filtrations of $\mathcal O_{hol}(\tilde C)$.

\begin{defn}\label{def:filtr:0911}
(a) $F_0^\delta\mathcal O_{hol}(\tilde C):=\mathbb C$.

(b) For $n\geq 0$, 
$F_{n+1}^\delta\mathcal O_{hol}(\tilde C):=\{f\in 
\mathcal O_{hol}(\tilde C)|d(f)\in \Omega(C)\cdot F_n^\delta\mathcal O_{hol}(\tilde C)\}$.

(c) For $n\geq 0$, $F_n^\mu\mathcal O_{hol}(\tilde C):=\mathcal O(C)\cdot F_n^\delta\mathcal O_{hol}(\tilde C)$.
\end{defn}

\begin{lem}[see Prop. \ref{new:cor}]\label{lem:intro} 
Set $F_\bullet^{\delta/\mu}:=F_\bullet^{\delta/\mu}\mathcal O_{hol}(\tilde C)$. 

(a) $F_\bullet^\delta$ and $F_\bullet^\mu$ are increasing algebra filtrations of $\mathcal O_{hol}(\tilde C)$. 

(b) $F_0^\delta\subset F_0^\mu\subset F_1^\delta\subset F_1^\mu\subset\cdots$
\end{lem}

Moreover, let us set $F_n\mathrm{Sh}(V):=\oplus_{k\leq n}\mathrm{Sh}_n(V)$ for $n\geq 0$ and $V$ any vector space, 
and let us remark that, for $x_0\in \tilde C$, the map $I_{x_0} : \mathrm{Sh}(\Omega(C))\to \mathcal O_{hol}(\tilde C)$
given by $[\omega_1|\cdots|\omega_n]\mapsto (x\mapsto \int_{x_0}^x \omega_1\cdots\omega_n)$
is an algebra morphism (see Lem. \ref{lem:221:3108}). Then one has the following.

\begin{lem}[see Prop. \ref{(a1)}]\label{lem:1:14:2211} 
$F_\bullet\mathrm{Sh}(\Omega(C))$ is an algebra filtration of $\mathrm{Sh}(\Omega(C))$ and  
$I_{x_0}(F_\bullet\mathrm{Sh}(\Omega(C)))$ is an algebra filtration of $\mathcal O_{hol}(\tilde C)$, which is independent 
of $x_0$.    
\end{lem}

These various algebra filtrations can be compared as follows. 

\begin{thmC}[see \S\ref{sect:proof:of:thms}] 
Let $(J,x_0)\in\mathrm{MC}_{nd}(C)\times\tilde C$. 

(a) One has the following equalities of algebra filtrations of $\mathcal O_{hol}(\tilde C)$: 
\begin{equation}\label{thm:c:1st:line}
F_\bullet\mathcal O_{mod}(\tilde C)=f_{J,x_0}(F_\bullet\mathrm{Sh}(\mathrm H_C)\otimes\mathcal O(C))=F_\bullet^\mu
\mathcal O_{hol}(\tilde C);
\end{equation}
\begin{equation}\label{thm:c:2nd:line}
I_{x_0}(F_\bullet\mathrm{Sh}(\Omega(C)))
=F_\bullet^\delta\mathcal O_{hol}(\tilde C)=f_{J,x_0}( 
F_\bullet\mathrm{Sh}(\mathrm H_C)\otimes\mathbb C+F_{\bullet-1}\mathrm{Sh}(\mathrm H_C)\otimes\mathcal O(C)). 
\end{equation}

(b) One has the equalities 
\begin{equation}\label{thm:c:3rd:line}
A_C=I_{x_0}(\mathrm{Sh}(\Omega(C)))=F_\infty^\delta\mathcal O_{hol}(\tilde C)
=F_\infty^\mu\mathcal O_{hol}(\tilde C)=f_{J,x_0}(\mathrm{Sh}(\mathrm H_C)\otimes\mathcal O(C))
=F_\infty\mathcal O_{mod}(\tilde C)
\end{equation}
of subalgebras of $\mathcal O_{hol}(\tilde C)$, where $F_\infty X:=\cup_{n\geq 0}F_nX$ for $F_\bullet X$ a filtration on a vector space~$X$. 
\end{thmC}

\subsubsection{Filtered formality for HACA structures}

Thm. B leads to the following extension of the notion of filtered formality (\cite{SW}) to the setting of 
HACAs. Any Hopf algebra~$O$ is equipped with a Hopf algebra filtration $F_\bullet O$, given when  
$O=(\mathbb C\Gamma)'$ by $F_nO=((\mathbb C\Gamma)_+^{n+1})^\perp$ for $n\geq 0$ 
(see Lem. \ref{sect:A2:0711}, \cite{BGF}, §3.3.2 and \cite{Fr} §7.2), which gives rise to a graded Hopf algebra $\mathrm{gr}(O)$. We 
say that $O$ is {\it filtered formal} if there exists an isomorphism of filtered Hopf algebras $O\to\mathrm{gr}(O)$, compatible 
with the filtrations and whose associated graded is the identity (see Def. \ref{def:FF:HA:3110}); we show that if a finitely generated group 
$\Gamma$ is filtered formal in the sense of~\cite{SW}, then the Hopf algebra $(\mathbb C\Gamma)'$ is filtered formal (see 
Prop. \ref{prop:C3:3110}). These notions extend to HACAs as follows. Any HACA $(O,A)$ is equipped with a HACA filtration 
$(F_\bullet O,F_\bullet A)$, whose first term is the above filtration of $O$ (see Lem. \ref{lem:5:3:0307}), and therefore 
gives rise to a graded HACA $(\mathrm{gr}(O),\mathrm{gr}(A))$ (see Lem. \ref{lem:10:12:2410}). We say that the HACA $(O,A)$ 
is {\it filtered formal} if there is an isomorphism of filtered HACAs $(O,A)\to(\mathrm{gr}(O),\mathrm{gr}(A))$, whose associated 
graded is the identity (see Def.~\ref{def:C4:0711}). 

\begin{prop}[see Prop. \ref{3:17:0711}] \label{repetition:prop:317} 
The HACA $((\mathbb C\Gamma)',F_\infty\mathcal O_{mod}(\tilde C))$ is filtered formal;  
the associated graded HACA is $(\mathrm{gr}((\mathbb C\Gamma)'),\mathrm{gr}(\mathcal O_{mod}(\tilde C)))$, where 
the components are the graded spaces associated to the filtrations of 
$(\mathbb C\Gamma)'$ and $F_\infty\mathcal O_{mod}(\tilde C)$. 
\end{prop}

\subsection{Organization of the paper}

In Part 1, we prove the results announced in \S\ref{pres:results:03:11}. 
In §\ref{sect;7;0212} (resp. \S\ref{sect:2:1309}), we introduce the necessary material on iterated integrals, Maurer-Cartan elements, 
and hyperlogarithm functions (resp. on moderate growth functions). In \S\ref{sect:3:2908}, we prove that $f_{J,x_0}$ 
is an isomorphism of filtered algebras; the argument relies on techniques of HACAs, which are explained in the appendices. 
In §\ref{sect:5:1412}, we prove results on the filtrations on $\mathcal O_{hol}(\tilde C)$ and their relation with the minimal 
stable subalgebra $A_C$, which  lead to the proof of Theorems A, B and~C. 

Part 2 is devoted to complementary results. In §\ref{sect:6:0212}, we introduce an study the notion of connections on HACAs. 
In \S\ref{sect:4:0711}, we give precisions on the local behavior of elements of~$A_C$, which may be viewed, since 
$A_C\subset\mathcal O_{mod}(\tilde C)$, as functions on a universal cover of $\tilde{\overline C}\smallsetminus p^{-1}(S)$ with moderate 
growth near each element of $p^{-1}(S)$ (see Prop. \ref{toto:0409}); in particular, the germs of the functions of $A_C$ near each such 
element are \emph{Nilsson-class functions} (in the sense of \cite{Ph}, p.~154) of a particular kind.
In \S\ref{sect:54:2811}, we relate $A_C$ to the minimal acyclic extension of the differential graded algebra (dga) 
$\Omega^\bullet(C):=(\mathcal O(C) \oplus \Omega(C),d)$ (see Prop. \ref{prop:5:16:2811}). 
In \S\ref{sect:8:0212}, we identify $\mathrm{Ker}(I_{x_0})$ with the image of an explicit map (see Thm. \ref{thm:7:1:0409}(a)), 
in the spirit of bar-complex theory (see \cite{H}), and we associate to each 
section $\sigma : \mathrm{H}_C^{\mathrm{dR}}\to\Omega(C)$ a complement $\mathrm{Sub}_\sigma$ of $\mathrm{Ker}(I_{x_0})$ in 
$\mathrm{Sh}(\Omega(C))$ (Thm. \ref{thm:7:1:0409}(b)). 

Part 3 is divided into four appendices, dealing with Hopf algebras and module or comodule algebras (HAMAs/HACAs). 

\part{Theorems A,B,C and their proofs}

\section{Iterated integrals, Maurer-Cartan elements, %the morphisms $\tilde f_{J,x_0}$ 
and hyperlogarithm functions}\label{sect;7;0212}

In §\ref{subsect:IIs:1008}, we introduce the iterated integral morphism $I_{x_0}$ attached to a point 
$x_0\in\tilde C$ and review its standard properties. In §\ref{subsect:mu:J}, we attach to each $J\in\mathrm{MC}(C)$ an algebra 
morphism $J_* : \mathrm{Sh}(\mathrm H_C)\to\mathrm{Sh}(\Omega(C))$. In §\ref{sect:91:0711} (see Prop. \ref{prop:2:15:0812}), we show 
that it gives rise to the morphism $\tilde f_{J,x_0} : \mathrm{Sh}(\mathrm H_C)\to \mathcal O_{hol}(\tilde C)$ from 
§\ref{pres:results:03:11}. We show that the image of $\tilde f_{J,x_0}$ is independent of $x_0$, and denote by 
$\mathcal H_C(J)$ its image. When $\sigma\in\Sigma_C$ (see Def. \ref{def:1:2:0311}), we set 
$\mathcal H_C(\sigma):=\mathcal H_C(J_\sigma)$. In §\ref{sect:92:0711}, we consider the case where 
$C=\mathbb P^1_{\mathbb C}\smallsetminus S$, and we exhibit a section $\sigma_0$ such that 
$\mathcal H_{\mathbb P^1_{\mathbb C}\smallsetminus S}(\sigma_0)$ coincides with the algebra of hyperlogarithm functions 
(Prop.~\ref{prop:mor:0:moreover}). 

\subsection{Iterated integrals}\label{subsect:IIs:1008}

Recall that the shuffle algebra $(\mathrm{Sh}(V),\sha)$ associated with a vector space~$V$ is isomorphic to the tensor algebra 
$\oplus_{d\geq 0}V^{\otimes d}$ as a vector space and that $\sha$ is commutative. It has the 
decomposition $\mathrm{Sh}(V)=\oplus_{d\geq 0}\mathrm{Sh}_d(V)$, where $\mathrm{Sh}_d(V)$ is the degree $d$ component $V^{\otimes d}$, 
which is an algebra grading; an element $v_1\otimes\cdots\otimes v_d\in V^{\otimes d}$ will be denoted $[v_1|\cdots|v_d]\in \mathrm{Sh}_d(V)$. 

%Let $\pi_{\mathrm{Sh}(V)} : \mathrm{Sh}(V)\to V$ be the projection on the component of degree 1, and 
%define $\partial^V : \mathrm{Sh}(V)\to \mathrm{Sh}(V) \otimes V$ by 
%$\partial^V:=(id \otimes \pi_{\mathrm{Sh}(V)}) \circ \Delta_{\mathrm{Sh}(V)}$. Then $\partial^V$ is given by
%$[v_1|\ldots,|v_n]\mapsto [v_1|\ldots,|v_{n-1}]\otimes v_n$ if $n>0$, $1\mapsto 0$ if $n=0$ and 
%$v_1,\ldots,v_n\in V$.
%Then $\partial^V$ is a derivation of the commutative and associative algebra  $\mathrm{Sh}(V)$
%with values in the free module $\mathrm{Sh}(V) \otimes V$ (with action $a\sha(b\otimes v):=(a\sha b)\otimes v$), namely 
%\begin{equation}\label{DER:partial}
%    \partial^V(a\sha b)= a \sha \partial^V(b)+b \sha \partial^V(a)
%\end{equation}
%for any $a,b\in \mathrm{Sh}(V)$. 

\begin{lemdef}\label{def:2:20:0908}
Let $x_0\in\tilde C$. 

(a) 
 For $x\in \tilde C$ and $\omega_1,\ldots,\omega_k \in \Omega(C)$, the iterated integral 
$$
\textstyle\int_\gamma \omega_1\cdots\omega_k:=\int_{0\leq t_1\leq\ldots\leq t_k\leq 1}
\gamma^*\omega_1(t_1)  \wedge\cdots \wedge \gamma^*\omega_k(t_k)
$$ 
is independent of a path $\gamma$ from~$x_0$ to~$x$; it will be denoted $\int_{x_0}^x \omega_1\cdots\omega_k$. 

(b) For any $n\geq 0$, there exists a unique linear map 
$I_{x_0}^{(n)} : \mathrm{Sh}_n(\Omega(C))\to\mathcal O_{hol}(\tilde C)$ such that 
$I_{x_0}^{(0)}(1)=1$, and for any $\omega_1,\ldots,\omega_n\in\Omega(C)$, one has 
$I_{x_0}^{(n)}([\omega_1|\ldots|\omega_n])=(x\mapsto\int_{x_0}^x \omega_1\ldots\omega_n)$. 

(c) The linear map $I_{x_0} : \mathrm{Sh}(\Omega(C))\to \mathcal O_{hol}(\tilde C)$ is the direct sum $\oplus_{n\geq 0}I_{x_0}^{(n)}$.  
\end{lemdef}

\begin{proof}
Let $\Omega_{hol}(\tilde C)$ be the space of 
holomorphic differentials on $\tilde C$ and $\mathrm{int}_{x_0} : \Omega_{hol}(\tilde C)\to\mathcal O_{hol}(\tilde C)$
be the linear map $\omega\mapsto (x\mapsto \int_{x_0}^x\omega)$, which is well-defined since 
 $\tilde C$ is simply-connected and since the elements of $\Omega_{hol}(\tilde C)$ are closed. Statement (a) follows from the equality $\int_\gamma \omega_1\cdots\omega_k=\mathrm{int}_{x_0}(\omega_k\cdot\mathrm{int}_{x_0}(\omega_{k-1}\cdots \omega_2\cdot
\mathrm{int}_{x_0}(\omega_1)))(x)$. Statements (b) and (c) are obvious. 
\end{proof}

One has therefore for $n\geq 1$ and $\omega_1,\ldots,\omega_n \in \Omega(C)$ the equality (in $\mathcal O_{hol}(\tilde C)$)
\begin{equation}\label{IND:1008}
I_{x_0}([\omega_1|\ldots|\omega_n])
=\mathrm{int}_{x_0}(I_{x_0}([\omega_1|\ldots|\omega_{n-1}])\cdot p^*\omega_n).
\end{equation}

\begin{lem}\label{lem:221:3108}
For any $x_0\in\tilde C$, $I_{x_0} : \mathrm{Sh}(\Omega(C))\to\mathcal O_{hol}(\tilde C)$ is an algebra morphism. 
\end{lem}

\begin{proof} See for example \cite{BGF}, Thm. 3.19, (3.22). 
\end{proof}

\begin{lem}\label{lem:invce:2308}
For any $x_0,x_1\in\tilde C$, $\gamma\in\Gamma_C$ and $a\in \mathrm{Sh}(\Omega(C))$, $I_{x_0}(a)(x_1)=I_{\gamma x_0}(a)(\gamma x_1)$. 
\end{lem}

\begin{proof}
It suffices to prove this for $a$ homogeneous. One argues by induction on $\mathrm{deg}(a)$. The identity is obvious for $\mathrm{deg}(a)=0$, 
and the identity for degree $n$ follows from the identity for degree $n-1$, \eqref{IND:1008} and the invariance of $p^*\omega$ for 
$\omega \in\Omega(C)$. 
\end{proof}

\begin{defn}\label{DEFSHV}
If $V$ is a vector space, then $\Delta_{\mathrm{Sh}(V)}$ is the deconcatenation coproduct on $\mathrm{Sh}(V)$, defined by 
$[v_1|\cdots|v_n]\mapsto \sum_{k=0}^n[v_1|\cdots|v_k]\otimes[v_{k+1}|\cdots|v_n]$ for $v_1,\ldots,v_n\in V$; it equips
$\mathrm{Sh}(V)$ with a commutative Hopf algebra structure. 
\end{defn}

\begin{lem}\label{lem:chain:rule:2308}
For $x_0,x_1 \in \tilde C$ and $a \in \mathrm{Sh}(\Omega(C))$, one has 
\begin{equation}\label{EQN:1408}
I_{x_0}(a)=I_{x_0}(a^{(1)})(x_1)I_{x_1}(a^{(2)}), 
\end{equation}
where $a^{(1)} \otimes a^{(2)}$ is Sweedler's notation for $\Delta_{\mathrm{Sh}(\Omega(C))}(a)$. 
\end{lem}

\begin{proof}
%The fact that the two sides of this equality belong to $\mathcal O_{mod}(\tilde C)$ follows from Prop. \ref{lem:INOMOD:1408}. 
Let us prove \eqref{EQN:1408} by induction on the degree of a homogeneous element $a$ in $\mathrm{Sh}(\Omega(C))$. Eq. \eqref{EQN:1408}
is obvious if $a$ has degree $0$. Assume that \eqref{EQN:1408} is proved for any $a$ of degree $<n$ and let us prove it in degree $n$. 
Let $\omega_1,\ldots,\omega_n \in \Omega(C)$. Then 
\begin{align*}
    &  dI_{x_0}([\omega_1|\ldots|\omega_n])=I_{x_0}([\omega_1|\ldots|\omega_{n-1}])\cdot p^*\omega_n
    {\displaystyle =\sum_{k=0}^{n-1} I_{x_0}([\omega_1|\ldots|\omega_k])(x_1)I_{x_1}([\omega_{k+1}|\ldots|\omega_{n-1}])\cdot p^*\omega_n}
    \\ & =\sum_{k=0}^{n-1} I_{x_0}([\omega_1|\ldots|\omega_k])(x_1)dI_{x_1}([\omega_{k+1}|\ldots|\omega_n]) 
    =d(\sum_{k=0}^{n-1} I_{x_0}([\omega_1|\ldots|\omega_k])(x_1)I_{x_1}([\omega_{k+1}|\ldots|\omega_n])).
\end{align*}
where the first and third equalities follow from \eqref{IND:1008}, and the second equality from \eqref{EQN:1408} in degree $n-1$. 

It follows that $I_{x_0}([\omega_1|\ldots|\omega_n])-\sum_{k=0}^{n-1} I_{x_0}([\omega_1|\ldots|\omega_k])(x_1)
I_{x_1}([\omega_{k+1}|\ldots|\omega_n]) \in\mathcal O_{hol}(\tilde C)$ is a constant function. 
Its value at $x_1$ is $I_{x_0}([\omega_1|\ldots|\omega_n])(x_1)$, therefore $I_{x_0}([\omega_1|\ldots|\omega_n])=\sum_{k=0}^{n} I_{x_0}([\omega_1|\ldots|\omega_k])(x_1)I_{x_1}([\omega_{k+1}|\ldots|\omega_n])$, proving \eqref{EQN:1408} in degree $n$. 
\end{proof}

\begin{lem}\label{NEwlem}
If $x_0\in\tilde C$, then $I_{x_0}(\mathrm{Sh}(\Omega(C)))$ is a $\Gamma_C$-stable 
subalgebra of $\mathcal O_{hol}(\tilde C)$, and for any $f\in I_{x_0}(\mathrm{Sh}(\Omega(C)))$ there exists $n\geq 0$ 
such that $f_{|(\mathbb C\Gamma_C)_+^{n+1}}=0$.  
\end{lem}

\begin{proof}
It follows from Lem. \ref{lem:chain:rule:2308} that the map $a\otimes\gamma\mapsto a_{|\gamma}:=
I_{x_0}(a^{(1)})(\gamma\cdot x_0)
a^{(2)}$ defines a right action of $\Gamma_C$ on the algebra $\mathrm{Sh}(\Omega(C))$, and that the algebra 
morphism $I_{x_0} : \mathrm{Sh}(\Omega(C))\to\mathcal O_{hol}(\tilde C)$ is $\Gamma_C$-equivariant.  
This implies that $I_{x_0}(\mathrm{Sh}(\Omega(C)))$ is a $\Gamma_C$-stable subalgebra of $\mathcal O_{hol}(\tilde C)$. 
The second statement follows from the equivariance of $I_{x_0}$ and from 
$a_{|(\mathbb C\Gamma_C)_+^{n+1}}=0$ for any $a\in\mathrm{Sh}_n(\Omega(C))$. 
\end{proof}

\subsection{Maurer-Cartan elements and associated morphisms}\label{subsect:mu:J}

Let $O$ be a Hopf algebra with coproduct~$\Delta_O$. For $a\geq 1$, let 
$\Delta_O^{(a)} : O\to O^{\otimes a}$ be the morphism obtained by iteration of 
$\Delta_O$. Let $pr_O : O\to O/\mathbb C$ be the canonical projection. 

\begin{defn}[see \cite{Q} B3, \cite{BGF}, \S3.3.2  or \cite{Fr} \S7.2] \label{conilp:filtra}
For $n\geq 0$, we define $F_n  O:=\mathrm{Ker}(pr_O^{\otimes n+1}\circ\Delta_O^{(n+1)})$. 
\end{defn}
We also set $F_{-1}  O=\{0\}$. Note that $F_0  O=\mathbb C1$.

\begin{lem}\label{lem:12:0309}
Let $O$ be a Hopf algebra such that $O=F_\infty O$, with $F_\bullet O$ as in Def. \ref{conilp:filtra}, let~$V$ be a vector space 
and let $\mu : O\to V$ be a linear map which is a derivation with respect to the counit $\epsilon$ of $O$, i.e. satisfying the 
identity $\mu(fg)=\mu(f)\epsilon(g)+\epsilon(f)\mu(g)$. Then the map $\mu_* : O\to \mathrm{Sh}(V)$ given by 
$f\mapsto \sum_{r\geq 0}[\mu(f^{(1)})|\ldots|\mu(f^{(r)})]$ is well-defined (using Sweedler's notation for the iterated 
coproduct of $O$, which is the counit if $r=0$), and is a Hopf algebra morphism. 
\end{lem}

\begin{proof}
Since $\mu(1)=0$, the map $f\mapsto \sum_{r\geq 0}[\mu(f^{(1)})|\ldots|\mu(f^{(r)})]$ 
takes $F_n%^{\mathrm{hpf}} 
O$ to $F_n\mathrm{Sh}(V)$ for any $n\geq0$, which implies that $\mu_*$ is well-defined.
For $f\in O$, and denoting by $\Delta_X$ the coproduct of~$X$ for $X$ any of the Hopf algebras
$O$ and $\mathrm{Sh}(V)$, one has 
\begin{align*}
& \mu_*^{\otimes2}\circ\Delta_O(f)=\mu_*(f^{(1)})\otimes \mu_*(f^{(2)})
=\sum_{r,s\geq0} [\mu(f^{(1)(1)})|\ldots|\mu(f^{(1)(r)})]\otimes [\mu(f^{(2)(1)})|\ldots|\mu(f^{(2)(s)})]
\\ & =\sum_{r,s\geq0} [\mu(f^{(1)})|\ldots|\mu(f^{(r)})]\otimes [\mu(f^{(r+1)})|\ldots|\mu(f^{(r+s)})]
=\Delta_{\mathrm{Sh}(V)}\circ\mu_*(f). 
\end{align*}
For $f,g\in O$, one has 
\begin{align*}
&\mu_*(fg)=\sum_{n\geq 0}[\mu(f^{(1)}g^{(1)})|\ldots|\mu(f^{(n)}g^{(n)})]
\\ & =\sum_{n\geq 0}[\epsilon(f^{(1)})\mu(g^{(1)})+\mu(f^{(1)})\epsilon(g^{(1)})|\ldots|
\epsilon(f^{(n)})\mu(g^{(n)})+\mu(f^{(n)})\epsilon(g^{(n)})]
\\ & =\sum_{n\geq 0}\sum_{n=k+l}\sum_{\substack{K,L||K|=l,|L|=l\\ K\sqcup L=[\![1,n]\!]}}
\prod_{a=1}^k \mu(f^{(i_a)})^{i_a}\prod_{b=1}^l\mu(g^{(j_b)})^{j_b}\prod_{a=1}^k\epsilon(g^{(i_a)})\prod_{b=1}^l\epsilon(f^{(j_b)})
\\ &  =\sum_{n\geq 0}\sum_{n=k+l}\sum_{\substack{K,L||K|=l,|L|=l\\ K\sqcup L=[\![1,n]\!]}}
[\mu(f^{(1)})|\ldots|\mu(f^{(k)})]^K
\cdot 
[\mu(g^{(1)})|\ldots|\mu(g^{(l)})]^L
\\ &  =\sum_{n\geq 0}\sum_{n=k+l}
[\mu(f^{(1)})|\ldots|\mu(f^{(k)})]\sha
[\mu(g^{(1)})|\ldots|\mu(g^{(l)})]
=\mu_*(f)\mu_*(g). 
\end{align*}
where the transport to $\mathrm{Sh}(V)$ of the product in the tensor algebra $T(V)$ is denoted by  
the two first product signs in the third line and by $\cdot$ in the fourth line; 
in the third line, we denote by $v\mapsto v^i$ the map $V\to \mathrm{Sh}(V)$, $v\mapsto 
[\underbrace{1\ldots 1}_{i-1}|v|\underbrace{1\ldots 1}_{n-i}]$, and by $(i_1,\ldots,i_k)$ and 
$(j_1,\ldots,j_l)$ the increasing sequences such that $K=\{i_1,\ldots,i_k\}$ and $L=\{j_1,\ldots,j_l\}$; 
in the fourth line, $x\mapsto x^K$ is the map $\mathrm{Sh}_k(V)\to\mathrm{Sh}(V)$, 
$[v_1|\ldots,|v_k]\mapsto \prod_{a=1}^k v_a^{i_a}$ and $y\mapsto y^L$ has the similar meaning. 

This proves that $\mu_*$ is compatible with the products and coproducts; one checks that it is compatible with the other 
aspects of the Hopf algebra structure (unit, counit, antipode). 
\end{proof}

%Let $\mathfrak G$ be a $\mathbb Z_{>0}$-graded Lie algebra, so $\mathfrak g=\oplus_{n\geq 1}\mathfrak G_n$, such that 
%$\mathfrak G$ is generated by $\mathfrak G_1$ and $\mathfrak G_1$ is finite dimensional. 
%Let $\hat{\mathfrak G}:=\hat\oplus_{n\geq 1}\mathfrak G_n$ be the degree 
%completion of $\mathfrak G$. For $V$ a $\mathbb Z_{>0}$-graded vector space and $W$ a vector space, set
%$\hat V\hat\otimes W:=\hat\oplus_{n\geq 0}V_n\otimes W$, and $W\hat\otimes\hat V:=\hat\oplus_{n\geq 0}W\otimes V_n$. 

%\begin{defn}
%The set $\mathrm{MC}(C,\hat{\mathfrak G})$ is the completed tensor product $\Omega(C)\hat\otimes\hat{\mathfrak G}$; 
%we call the elements of this set Maurer-Cartan elements for the pair $(C,\hat{\mathfrak G})$. 
%\end{defn}

%One then has $\mathrm{MC}(C)=\mathrm{MC}(C,\hat{\mathbb L}(\mathrm H_C^*))$.  

\begin{lem}\label{lem:14:0309}\label{lem:comp:Sh:heart:1628}
For $V$ a vector space, $\mathrm{Sh}(V)=F_\infty(\mathrm{Sh}(V))$, where $F_\bullet \mathrm{Sh}(V)$ 
is\footnote{This ``Hopf algebra'' filtration of $\mathrm{Sh}(V)$ can be shown to coincide with the 
``degree'' filtration, which a posteriori justifies denoting them in the same way.} 
as in Def. \ref{conilp:filtra}.   
\end{lem}

\begin{proof}
Since $\mathrm{Sh}(V)$ is a connected graded Hopf algebra, this follows from Prop. \ref{prop:LA:filt}(d).
\end{proof}

Let $J\in \mathrm{MC}(C)$ (see Def. \ref{def:MC:0812}). Let $\mu_J : \mathrm{Sh}(\mathrm H_C)\to\Omega(C)$
be the composition $\mathrm{Sh}(\mathrm H_C)\simeq\oplus_{n\geq 0}T(\mathrm H_C^*)[n]^*\to 
\oplus_{n\geq 0}\mathbb L(\mathrm H_C^*)[n]^*\to\Omega(C)$, where the second map is dual to the inclusion 
$\mathbb L(\mathrm H_C^*)[n]\subset T(\mathrm H_C^*)[n]$ and the last map is induced by $J$. 

\begin{lem}\label{lem:02:0208}
Let $J\in \mathrm{MC}(C)$.
%$J\in \mathrm{MC}(C,\hat{\mathfrak G})$. 
Then the linear map $\mu_J : \mathrm{Sh}(\mathrm H_C)\to\Omega(C)$
%$\mu_J : U(\mathfrak G)'\to\Omega(C)$ 
is a derivation with respect to the 
counit of $\mathrm{Sh}(\mathrm H_C)$. 
%$U(\mathfrak G)'$. The restriction of $\mu_J$ to $U(\mathfrak g)[n]^*\cdot U(\mathfrak g)[m]^*$ is zero whenever $n,m>0$ or $n=m=0$. 
\end{lem}

\begin{proof}
For $f,g \in \mathrm{Sh}(\mathrm H_C)$, %(U\mathfrak G)'$, 
one has $\mu_J(fg)=(fg \otimes id)(J)
=(f \otimes g \otimes id)((\Delta_{\hat T(\mathrm H_C^*)} \otimes id)(J))
=(f \otimes g \otimes id)(J^{13}+J^{23})
=\mu_J(f)\epsilon(g)+\epsilon(f)\mu_J(g)$, 
where the third equality follows from the primitiveness of the first component of $J$, and $x\mapsto x^{13},x^{23}$
are the maps $\hat T(\mathrm H_C^*)\hat\otimes\Omega(C)\to \hat T(\mathrm H_C^*)^{\hat\otimes2}\hat\otimes\Omega(C)$
induced by $t \otimes \omega\mapsto t \otimes 1 \otimes \omega,1 \otimes t \otimes \omega$.
\end{proof}

\begin{cor}
    If $J\in\mathrm{MC}(C)$, %$J\in\mathrm{MC}(C,\hat{\mathfrak G})$, 
    then the map 
\begin{equation}\label{DEF:J*}
    %\underline 
J_* : \mathrm{Sh}(\mathrm H_C)%U(\mathfrak G)'
\to \mathrm{Sh}(\Omega(C)),\quad \xi\mapsto \sum_{r\geq 0}[\mu_J(\xi^{(1)})|\cdots|\mu_J(\xi^{(r)})]
\end{equation}
 is well-defined, and is an algebra morphism. 
\end{cor}

\begin{proof}
    This follows from Lems. \ref{lem:12:0309}, \ref{lem:14:0309} with $V=\mathrm H_C$, and \ref{lem:02:0208}. 
\end{proof}

\begin{lemdef}\label{lem:def:sigma:J:1012}
For $J\in\mathrm{MC}_{nd}(C)$, there is a unique $\sigma\in\Sigma_C$ such that 
the degree $1$ component of $J$ (for the degree of $\hat{\mathbb L}(\mathrm H_C)^*$) is equal to $J_\sigma$. 
This defines a map $\mathrm{MC}_{nd}(C)\to\Sigma_C$, $J\mapsto\sigma_J$. It satisfies the identity $\sigma_{J_\sigma}=\sigma$.  
\end{lemdef}

\begin{proof} Obvious.\end{proof}  

\begin{lem}\label{GR:ULJ}
(a) For $\sigma\in\Sigma_C$, the morphism $(J_\sigma)_* : \mathrm{Sh}(\mathrm{H}_C)\to\mathrm{Sh}(\Omega(C))$ 
(see Def. \ref{def:1:2:0311}) is graded, and coincides with the morphism $\sigma_*$ functorially induced by the 
linear map $\sigma$.  

(b) For $J\in\mathrm{MC}_{nd}(C)$, the morphism $J_* : \mathrm{Sh}(\mathrm{H}_C)\to\mathrm{Sh}(\Omega(C))$ is compatible with the 
filtrations $F_\bullet$ of both sides, and the associated graded morphism coincides with the 
morphism attached to $\sigma_J$ so $\mathrm{gr}(J_*)=(\sigma_J)_*$. 
\end{lem}

\begin{proof} (a) is obvious. For (b), 
let $n\geq 1$, $h_1,\ldots,h_n \in \mathrm H_C$ and $\xi:=[h_1|\ldots|h_n]\in
\mathrm{Sh}(\mathrm H_C)$. Then for $r \geq n$, 
$$
\xi^{(1)} \otimes\ldots\otimes \xi^{(r)} \in \delta_{r,n}h_1 \otimes\ldots \otimes h_n+
\oplus_{i=1}^r \mathrm{Sh}(\mathrm H_C)^{\otimes i-1} \otimes 1 \otimes 
\mathrm{Sh}(\mathrm H_C)^{\otimes r-i},
$$
which since $\mu_J(1)=0$ (see Lem. \ref{lem:02:0208}) and $\mu_J(h)=\sigma_J(h)$ for $h \in \mathrm H_C$ 
implies $[\mu_J(\xi^(1))|\ldots|\mu_J(\xi^{(r)}]=\delta_{r,n}[\sigma_J(h_1)|\ldots|\sigma_J(h_n)]$. The statement follows by combining this with \eqref{DEF:J*}.  
\end{proof}  

\subsection{The algebras $\mathcal H_C(J)$, $\mathcal H_C(\sigma)$}\label{sect:91:0711}

Let $J\in\mathrm{MC}_{nd}(C)$, $d:=\mathrm{dim}(\mathrm H_C)$ and $(h_i)_{i\in[\![1,d]\!]}$ 
be\footnote{We set $[\![1,d]\!]:=\{1,2,\ldots,d\}$.} a basis of $\mathrm H_C$. Then  
$$
(I_{x_0}\circ J_*([h_{i_1}|\ldots|h_{i_k}]))_{\substack{k\geq 0,i_1,\ldots,i_k \in [\![1,d]\!]}}. 
$$
is a family of elements of $\mathcal H_C(J)$. When $J=J_\sigma$ with $\sigma\in\Sigma_C$, this family is equal to
$$
(I_{x_0}([\sigma(h_{i_1})|\ldots|\sigma(h_{i_k})]))_{\substack{k\geq 0,i_1,\ldots,i_k \in [\![1,d]\!]}}. 
$$
The following statement is a generalization of Cor. 5.6 in~\cite{Br:these}. 

\begin{prop}\label{prop:constr:L:0304}
For $(J,x_0)\in\mathrm{MC}_{nd}(C)\times\tilde C$, the element 
$$
\mathbf L_{J,x_0}:=\sum_{k\geq 0,i_1,\ldots,i_k\in[\![1,d]\!]}
I_{x_0} \circ J_*([h_{i_1}|\ldots|h_{i_k}])
\otimes (h^{i_1}\otimes\cdots\otimes h^{i_k})\in \mathcal O_{hol}(\tilde C)\hat\otimes\hat T((\mathrm H_C)^*), 
$$
where $(h^i)_{i\in[\![1,d]\!]}$ is the basis of $(\mathrm H_C)^*$ dual to $(h_i)_{i\in[1,d]}$,
satisfies the equality 
\begin{equation}\label{diff:eq:J}
(d\otimes \mathrm{id})(\mathbf L_{J,x_0})=\mathbf L_{J,x_0}\cdot J%\Big(\sum_{i\in[\![1,d]\!]}p^*\sigma(h_i)\otimes h^i\Big)
\end{equation}
(equality in $\Omega_{hol}(\tilde C)\hat\otimes\hat T((\mathrm H_C)^*)$) as well as $\mathbf L_{J,x_0}(x_0)=1$; it is the only element of the algebra
$\mathcal O_{hol}(\tilde C)\hat\otimes\hat T((\mathrm H_C)^*)$ satisfying these conditions. 
\end{prop}
 
\begin{proof} 
It follows from the definition of $\mu_J$ that 
$J=\sum_{r\geq 0}\sum_{(i_1,\ldots,i_r)\in[\![1,d]\!]^r}\mu_J([h_{i_1}|\ldots|h_{i_r}]) 
\otimes h^{i_1} \otimes\cdots \otimes h^{i_r}$ 
(equality in $\Omega(C) \hat\otimes \hat T(\mathrm H_C^*)$, and therefore, expanding $J$ as a sum  
$\sum_{\alpha} \omega_\alpha \otimes x^\alpha$  (convergent for the topology of $\Omega(C) \hat\otimes \hat T(\mathrm{H}_C^*)$), 
that $\mathbf L_{J,x_0}=\sum_{r\geq 0} \sum_{\alpha_1,\ldots,\alpha_r}
I_{x_0}([\omega_{\alpha_1}|\ldots|\omega_{\alpha_r}]) \otimes x^{\alpha_1}\cdots x^{\alpha_r}$. 
Then 
$(d\otimes \mathrm{id})\mathbf L_{J,x_0}=\sum_{r\geq 1} \sum_{\alpha_1,\ldots,\alpha_r}
I_{x_0}([\omega_{\alpha_1}|\ldots|\omega_{\alpha_{r-1}}])\omega_{\alpha_r} \otimes x^{\alpha_1}\cdots x^{\alpha_r}
=\mathbf L_{J,x_0}\cdot \sum_\alpha \omega_\alpha \otimes x^\alpha=\mathbf L_{J,x_0}\cdot J$.  
%\eqref{diff:eq:J} follows from $d(I_{x_0}([t_1|\ldots|t_n]))=I_{x_0}([t_1|\ldots|t_{n-1}])\cdot p^*t_n$ for $n\geq 1$
%and $t_1,\ldots,t_n\in \Omega(C)$. 

The second statement follows from the fact that $t\mapsto I_{x_0}(t)(x_0)$ is the augmentation map 
$\mathrm{Sh}(\Omega(C))\to\mathbb C$. The uniqueness follows from the fact the ratio of two solutions must be equal to 1 at $x_0$ and  
be killed by $d$, hence be 1. 
\end{proof}

Eq. \eqref{diff:eq:J} is a generalization of eq. (5.1) in \cite{Br:these}, and $\mathbf L_{J,x_0}$ is a generalization of the 
solution $L(z)$ of this equation constructed in {\it loc. cit.,} Prop. 5.1.

\begin{prop}\label{prop:2:15:0812}
For any $(J,x_0)\in\mathrm{MC}_{nd}(C)\times\tilde C$, the algebra morphism $I_{x_0} \circ J_* : \mathrm{Sh}(H_C)\to
\mathcal O_{hol}(\tilde C)$ is such that for any $a \in \mathrm{Sh}(\mathrm H_C)$,   
$$
I_{x_0} \circ J_*(a)=\langle\mathrm{id}\otimes a,\mathbf L_{J,x_0}\rangle. 
$$
Therefore
$I_{x_0} \circ J_*=\tilde f_{J,x_0}$ (see \eqref{def:tildef:0304}); in particular
$\tilde f_{J,x_0} : \mathrm{Sh}(\mathrm H_C)\to\mathcal O_{hol}(\tilde C)$ 
is an algebra morphism. 
\end{prop}

\begin{proof}
Follows from the fact that $(k,i_1,\ldots,i_k)\mapsto [h_{i_1}|\ldots|h_{i_k}]$ and 
$(k,i_1,\ldots,i_k)\mapsto h^{i_1}\otimes\cdots\otimes h^{i_k}$ are dual bases of 
$\mathrm{Sh}(\mathrm H_C)$ and $T(\mathrm H_C^*)$. 
\end{proof}

\begin{lem}\label{lem:bibli:3005}
Let $J\in\mathrm{MC}_{nd}(C)$. The image 
$\tilde f_{J,x_0}(\mathrm{Sh}(\mathrm H_C))$ is independent of $x_0\in\tilde C$. 
We denote by $\mathcal H_C(J)\subset \mathcal O_{hol}(\tilde C)$ this subalgebra, and set 
$\mathcal H_C(\sigma):=\mathcal H_C(J_\sigma)$ for any $\sigma\in\Sigma_C$.   
\end{lem}

\begin{proof} Let $x_0,x_1$ in $\tilde C$. Let $\mu_{x_0}^{x_1} : \mathrm{Sh}(\mathrm H_C)\to\mathbb C$ be the map 
$t\mapsto I_{x_0}(J_*(t))(x_1)$. This is an algebra morphism, as it is the composition of the algebra morphism
$\tilde f_{J,x_0}$ and of the morphism $\mathcal O_{mod}(\tilde C)
\stackrel{\mathrm{ev}_{x_1}}{\to}\mathbb C$ of evaluation at $x_1$. Set 
$a_{x_0}^{x_1}:=(\mu_{x_0}^{x_1}\otimes\mathrm{id}) \circ \Delta_{\mathrm{Sh}(\mathrm H_C)}$; 
this is an algebra endomorphism of $\mathrm{Sh}(\mathrm H_C)$ as it is a composition of algebra morphisms. It is also a vector 
space automorphism of $\mathrm{Sh}(\mathrm H_C)$ since it is compatible with the filtration and the associated graded
endomorphism of $\mathrm{gr}(\mathrm{Sh}(\mathrm H_C))$ is the identity. Therefore $a_{x_0}^{x_1}$ is an algebra automorphism of 
$\mathrm{Sh}(\mathrm H_C)$. 

Eq. \eqref{EQN:1408} implies the 
identity $I_{x_1}(J_*(t))=I_{x_1}(J_*(t^{(1)}))(x_0)I_{x_0}(J_*(t^{(2)}))$ 
(equality in $\mathcal O_{hol}(\tilde C)$) for any $t\in \mathrm{Sh}(\mathrm H_C)$, therefore the equality 
$I_{x_1}\circ J_*=I_{x_0}\circ J_*\circ a_{x_0}^{x_1}$ (equality of maps 
$\mathrm{Sh}(\mathrm H_C)\to \mathcal O_{hol}(\tilde C)$), i.e. $\tilde f_{J,x_1}=\tilde f_{J,x_0}\circ a_{x_0}^{x_1}$,
which together with the bijectivity of $a_{x_0}^{x_1}$
implies the statement.  
\end{proof}

\subsection{The algebra $\mathcal H_C(\sigma)$ in the genus $0$ case}\label{sect:92:0711}

Let $S\subset\mathbb P^1_{\mathbb C}$ be a finite set containing $0$ and~$\infty$, let $C:=\mathbb P^1_{\mathbb C}\smallsetminus  S$. 
Recall the linear isomorphism $\mathrm H_C\simeq\mathbb C \hat S_\infty$, where 
$S_\infty=S\smallsetminus \{\infty\}$ and let $\sigma_0\in\Sigma_C$ be such that for any 
$s\in S_\infty$, $\sigma_0(\hat s)=\mathrm{dlog}(z-s)$. By Lem.  \ref{lem:bibli:3005}, one attaches to it the subalgebra 
$\mathcal H_{\mathbb P^1_{\mathbb C}\smallsetminus S}(\sigma_0)\subset \mathcal O_{hol}(\tilde C)$. 

Set $\mathrm H_C^{(0)}:=\oplus_{s\in S_\infty\smallsetminus \{0\}}\mathbb C\hat s\subset \mathrm H_C$ and 
$\mathrm{Sh}^*(\mathrm H_C):=\mathbb C\oplus[\mathrm H_C^{(0)}|\mathrm{Sh}(\mathrm H_C)]$ (recall that 
$[-|-]$ denotes the concatenation in $\mathrm{Sh}(\mathrm H_C)$). Then $\mathrm{Sh}^*(\mathrm H_C)$ is a subalgebra of $\mathrm{Sh}(\mathrm H_C)$.

\begin{lem}[see \cite{Pa:thesis}, \S3.3] \label{lem:2124:3005} 
(a) Denote by $\mathrm{Sh}^*(\mathrm H_C)[X]$ the polynomial algebra in one variable over the algebra $\mathrm{Sh}^*(\mathrm H_C)$. 
The combination of the canonical injection $\mathrm{Sh}^*(\mathrm H_C)\hookrightarrow\mathrm{Sh}(\mathrm H_C)$ and of the assignment 
$X\mapsto[\hat 0]$ gives rise to an algebra isomorphism $\mathrm{Sh}^*(\mathrm H_C)[X]\to \mathrm{Sh}(\mathrm H_C)$. 

(b) Let $\delta>0$ be such that $]0,\delta[ \subset C$. Fix a connected component $K$ of $p^{-1}(]0,\delta[)$.  The restriction of $p$ is 
a bijection $K\to]0,\delta[$; denote by $q_K\! :\, ]0,\delta[\to K$ the inverse bijection. For 
$t\in \mathrm{Sh}^*(\mathrm H_C)$ and $z\in \tilde C$, the limit 
$\lim_{\epsilon\to 0}I_{q_K(\epsilon)}((\sigma_0)_*(t))(z)$ exists; the function  
$$
\tilde f^*_{\sigma_0,0}(t):=(z\mapsto \lim_{\epsilon\to 0}I_{q_K(\epsilon)}((\sigma_0)_*(t))(z))
$$
belongs to $\mathcal O_{hol}(\tilde C)$; the map 
$\tilde f^*_{\sigma_0,0} : \mathrm{Sh}^*(\mathrm H_C)\to\mathcal O_{hol}(\tilde C)$, $t\mapsto \tilde f^*_{\sigma_0,0}(t)$ is an algebra morphism. 
\end{lem}

\begin{defn}[see \cite{Pa:thesis}, \S3.3] \label{def:2011} 
We denote by $\tilde f_{\sigma_0,0} : \mathrm{Sh}(\mathrm H_C)\to\mathcal O_{hol}(\tilde C)$ the composition with the inverse of the isomorphism of Lem. \ref{lem:2124:3005}(a)
of the morphism $\mathrm{Sh}^*(\mathrm H_C)[X]\to \mathcal O_{hol}(\tilde C)$ extending $\tilde f^*_{\sigma_0,0}$ by $X\mapsto \mathrm{log}$. 
\end{defn}

The map $\tilde f_{\sigma_0,0}$ is an algebra morphism. For $w\in \hat{S}_\infty^*\subset \mathrm{Sh}(\mathrm H_C)$, the element $\tilde f_{\sigma_0,0}(w)\in\mathcal O_{hol}(\tilde C)$ is denoted 
$L_w$ in \cite{Pa:thesis}, \S3.3, and called the hyperlogarithm function associated to $w$.

\begin{prop}\label{prop:mor:0:moreover}
One has $\tilde f_{\sigma_0,0}(\mathrm{Sh}(\mathrm H_C))=\mathbb C[L_w|w \in \hat S^*_\infty]=
\mathcal H_{\mathbb P^1_{\mathbb C}\smallsetminus S}(\sigma_0)$. 
\end{prop}

\begin{proof} Fix $z_0\in\tilde C$. 
The subalgebra $\mathrm{Sh}^*(\mathrm H_C)$ of $\mathrm{Sh}(\mathrm H_C)$ is a right coideal for the coalgebra structure, so that 
$\Delta_{\mathrm{Sh}(\mathrm H_C)}$ induces an algebra morphism $\Delta_{\mathrm{Sh}^*(\mathrm H_C)} : \mathrm{Sh}^*(\mathrm H_C)
\to\mathrm{Sh}^*(\mathrm H_C)\otimes \mathrm{Sh}(\mathrm H_C)$. Composing with the tensor product of the composition 
$\mathrm{Sh}^*(\mathrm H_C)\stackrel{\tilde f^*_{\sigma_0,0}}{\to}\mathcal O_{hol}(\tilde C)
\stackrel{\mathrm{ev}_{z_0}}{\to}\mathbb C$ with the identity, one gets an algebra morphism 
$((\mathrm{ev}_{z_0}\circ \tilde f^*_{\sigma_0,0})\otimes\mathrm{id})\circ \Delta_{\mathrm{Sh}^*(\mathrm H_C)} : \mathrm{Sh}^*(\mathrm H_C)
\to\mathrm{Sh}(\mathrm H_C)$. By Lem. \ref{lem:2124:3005}, there exists a unique algebra endomorphism $a_0^{z_0}$ of $\mathrm{Sh}(\mathrm H_C)$, 
whose restriction to $\mathrm{Sh}^*(\mathrm H_C)$ coincides with $((\mathrm{ev}_{z_0}\circ \tilde f^*_{\sigma_0,0})\otimes\mathrm{id})\circ 
\Delta_{\mathrm{Sh}^*(\mathrm H_C)}$ and such that $[\hat 0]\mapsto[\hat 0]+\mathrm{log}(z_0)$. One checks that $a_0^{z_0}$ is 
compatible with the filtration of 
$\mathrm{Sh}(\mathrm H_C)$ and that the 
associated graded endomorphism of $\mathrm{gr}(\mathrm{Sh}(\mathrm H_C))$ is the identity, so that $a_0^{z_0}$ is an algebra automorphism of 
$\mathrm{Sh}(\mathrm H_C)$. 

Specializing \eqref{EQN:1408} for $t\in (\sigma_0)_*(\mathrm{Sh}^*(\mathrm H_C))$ and taking its limit for $z\to 0$, 
one obtains 
the equality $\tilde f^*_{\sigma_0,0}(t)(z'')=\tilde f^*_{\sigma_0,0}(t^{(1)})(z')I_{z'}(\sigma_0(t^{(2)}))(z'')$ for $t\in \mathrm{Sh}^*(\mathrm H_C)$ and any 
$z',z''\in\tilde C$. Setting in this identity $z':=z_0$ and viewing both sides as a function of $z''$, one
obtains the identity 
$$
\tilde f^*_{\sigma_0,0}(t)=\tilde f^*_{\sigma_0,0}(t^{(1)})(z_0)I_{z_0}((\sigma_0)_*(t^{(2)}))
$$
(in $\mathcal O_{hol}(\tilde C)$) which is equivalent to the statement that the restrictions to $\mathrm{Sh}^*(\mathrm H_C)$
of $\tilde f_{\sigma_0,0}$ and $I_{z_0}\circ(\sigma_0)_* \circ a_0^{z_0}$, which are algebra morphisms 
$\mathrm{Sh}(\mathrm H_C)\to\mathcal O_{hol}(\tilde C)$ are equal.
The images by these morphisms of $[\hat 0]\in \mathrm{Sh}(\mathrm H_C)$ are also equal, since $\tilde f_{\sigma_0,0}([\hat 0])
=(z\mapsto \mathrm{log}(z))$
while $a_0^{z_0}([\hat 0])=[\hat 0]+\mathrm{log}(z_0)$ and $I_{z_0}\circ(\sigma_0)_*([\hat 0])=(z\mapsto \mathrm{log}(z)-\mathrm{log}(z_0))$. 
As $\mathrm{Sh}(\mathrm H_C)$ is generated by $\mathrm{Sh}^*(\mathrm H_C)$ and $[\hat 0]$, the algebra morphism status of both
$\tilde f_{\sigma_0,0}$ and $I_{z_0}\circ(\sigma_0)_* \circ a_0^{z_0}$ implies that
\begin{equation}\label{equality:endos:I}
\tilde f_{\sigma_0,0}=I_{z_0}\circ(\sigma_0)_* \circ a_0^{z_0}=\tilde f_{\sigma_0,0}\circ a_0^{z_0}    
\end{equation}
(equality of algebra morphisms $\mathrm{Sh}(\mathrm H_C)\to\mathcal O_{hol}(\tilde C)$). 

One then has $\tilde f_{\sigma_0,0}(\mathrm{Sh}(\mathrm H_C))
=\tilde f_{\sigma_0,z_0} \circ a_0^{z_0}(\mathrm{Sh}(\mathrm H_C))
=\tilde f_{\sigma_0,z_0}(\mathrm{Sh}(\mathrm H_C))=\mathcal H_{\mathbb P^1_{\mathbb C}\smallsetminus S}(\sigma_0)$, where the first 
equality follows from \eqref{equality:endos:I}, the second follows from the automorphism status of $a_0^{z_0}$, and the last from 
Lem. \ref{lem:bibli:3005}. 

One also has $\mathbb C[L_w|w \in \hat S^*_\infty]=\tilde f_{\sigma_0,0}(\mathbb C[w|w \in \hat S_\infty^*])
=\tilde f_{\sigma_0,0}(\mathrm{Sh}(\mathrm H_C))$. 
\end{proof}

\begin{rem}\label{rem:2:20}
The element $J_{\sigma_0}=\sum_{s\in S_\infty}d\log(z-s)\otimes h^s$
%This connection 
is related to the Knizhnik-Zamolodchikov (KZ) connection as follows. 
Recall that this connection, denoted $\nabla_{\mathrm{KZ}}=d+A_{\mathrm{KZ}}$, 
is a $\mathrm{exp}(\mathfrak t_n)$-connection over the configuration space $C_n(\mathbb C)$ of $n$ points in $\mathbb C$, 
where $\mathfrak t_n$ is the topological Lie algebra with generators $t_{ij}$, $i\neq j\in [\![1,n]\!]$ and relations $t_{ji}=t_{ij}$ for 
$|\{i,j\}|=2$, $[t_{ik}+t_{jk},t_{ij}]=0$ for $|\{i,j,k\}|=3$ and $[t_{ij},t_{kl}]=0$ for $|\{i,j,k,l\}|=4$, 
and that $A_{\mathrm{KZ}}=\sum_{i \neq j}d\mathrm{log}(z_i-z_j)\otimes t_{ij}$. 
Let $(s_1,\ldots,s_n)\in C_n(\mathbb C)$ and $ S_\infty:=\{s_1,\ldots,s_n\}$; then $\mathbb C\smallsetminus S_\infty$ is the preimage of 
$(s_1,\ldots,s_n)$ by the projection $C_{n+1}(\mathbb C)\to C_n(\mathbb C)$. Let $\mathrm{inj}_{ S_\infty} : \mathbb C\smallsetminus  S_\infty 
\hookrightarrow C_{n+1}(\mathbb C)$ be the canonical injection, and let $\iota : \mathbb L((\mathrm H_C)^*)
\to \mathfrak t_{n+1}$ be the Lie algebra morphism induced by $h^i\mapsto t_{i,n+1}$ for $i=1,\ldots,n$. Then 
$$
\iota_*(J_{\sigma_0})=(\mathrm{inj}_{ S_\infty})^*(A_{\mathrm{KZ}}). 
$$    
\end{rem}

\section{Moderate growth functions}\label{sect:2:1309}

We introduce in §\ref{subset:2:1:3108} the algebra $\mathcal O_{mod}(\tilde C)$ of moderate growth functions on $\tilde C$, and in 
§\ref{subsect:Omega:mod:2811} the space of moderate growth differentials $\Omega_{mod}(\tilde C)$, which is a module over it. 
In \S\ref{sect:iis:mod:0411}, we study the relations of the iterated integral morphism $I_{x_0}$ with moderate growth functions.

\subsection{Moderate growth functions on $\tilde C$}\label{subset:2:1:3108}

\subsubsection{Moderate growth functions on a disc}\label{mgfoad:0709}

Set $D:=\{z\in\mathbb C\,\,|\,\,|z|<1\}$, $D^\times:=D\smallsetminus\{0\}$ and $\tilde D^\times:=\{u\in\mathbb C\,|\,\Im(u)>0\}$. Let 
$e : \tilde D^\times\to D^\times$ be the map defined by $e(u):=\mathrm{exp}(2\pi \mathrm i u)$ (we set $\mathrm i:=\sqrt{-1}$), and let $\theta$ 
be the automorphism of $\tilde D^\times$ given by $\theta(u)=u+1$. For $f : M\to N$ a morphism of complex manifolds, we 
denote by $f^* : \mathcal O_{hol}(N)\to \mathcal O_{hol}(M)$ the induced morphism between algebras of holomorphic functions. 
For $F=D,D^\times,\tilde D^\times$, we denote by $\mathcal O_{hol}(F)$ the algebra of holomorphic functions on $F$.  

\begin{defn}[see \cite{Ph}, Ch. VIII, Def. 1.2] \label{1529:0908} 
%(a) For $n\geq0$, 
Define $\mathcal O_{mod}%^{(n)}
(\tilde D^\times)\subset \mathcal O_{hol}(\tilde D^\times)$ as the set of functions 
$f$ such that there exists an integer $n>0$ and a function $\{(a,b)\in\mathbb R^2\,|\,a\leq b\}\ni (a,b)\mapsto C_{a,b} 
\in \mathbb R_+$ such that 
$|f(x+\mathrm iy)|\leq C_{a,b} e^{2\pi ny}$ for $(x,y)\in [a,b]\times\mathbb R_+$. 
\end{defn}

\begin{rem}
If $f$ satisfies these conditions for the pair $(n,(a,b)\mapsto C_{a,b})$, then it satisfies it also for 
$(n+1,(a,b)\mapsto C_{a,b})$. 
\end{rem}

\begin{lem}\label{lem:debut:0924}
$\mathcal O_{mod}%^{(n)}
(\tilde D^\times)$ is a subalgebra of %vector subspace of 
$\mathcal O_{hol}(\tilde D^\times)$, equipped with an action of $\mathbb Z$ where $1$ acts by $\theta^*$. 
\end{lem}

\begin{proof}
The constant function $c$ satisfies the conditions from Def. \ref{1529:0908} with $n=0$ and $(a,b)\mapsto C_{a,b}=|c|$; 
if the function $f$ (resp. $g$) satisfies them with $(n,(a,b)\mapsto C_{a,b})$ (resp. $(m,(a,b)\mapsto D_{a,b})$), then the 
function $f+g$ satisfies them with $(\mathrm{max}(n,m),(a,b)\mapsto C_{a,b}+D_{a,b})$, the function $fg$ satisfies them with 
$(n+m,(a,b)\mapsto C_{a,b}D_{a,b})$, and the function $\theta^*f$ satisfies them with $(n,(a,b)\mapsto C_{a+1,b+1})$.
\end{proof}

\begin{defn}
Set $\mathcal O%_{hol}%^{(n)}
(D^\times):=\{f\in\mathcal O_{hol}(D^\times)\,|\,\exists\, n\geq 0, z^nf\in \mathcal O_{hol}(D)\}$.
\end{defn}
Then $\mathcal O(D^\times)$ is the algebra of meromorphic functions on $D$ with only possible poles at $0$. 

\begin{lem}\label{lem:1:5:1304:0208}
$\mathcal O_{mod}(\tilde D^\times)^{\mathbb Z}=\mathcal O(D^\times)$. 
\end{lem}

\begin{proof} The inclusion $\mathcal O_{mod}%^{(n)}
(\tilde D^\times)^{\mathbb Z}\supset\mathcal O%_{hol}^{(n)}
(D^\times)$ is evident, let us show the opposite inclusion. Let $f\in \mathcal O_{mod}%^{(n)}
(\tilde D^\times)^{\mathbb Z}$. Then there exists $n>0$ and $C_{0,1}\in\mathbb R_+$, such that 
$|f(x+\mathrm iy)|\leq C_{0,1}e^{2\pi n y}$ for any $(x,y) \in [0,1]\times\mathbb R_+$. Since $f$ is $\mathbb Z$-invariant, 
$f\in\mathcal O_{hol}(D^\times)$, therefore $z^nf\in\mathcal O_{hol}(D^\times)$. For $z\in D^\times$, there exists a unique 
$(x,y) \in [0,1[\times\mathbb R_+$ such that $z=e(x+\mathrm iy)$. Then 
$|z^nf(z)|\leq C_{0,1} e^{-2 \pi n y}e^{2 \pi n y}=C_{0,1}$. By the removable singularity theorem (see \cite{L}, Thm. 3.1), $z^nf$ 
is the restriction to $D^\times$ of a function of $\mathcal O_{hol}(D)$, therefore $f\in\mathcal O%_{hol}^{(n)}
(D^\times)$.    
\end{proof}

\subsubsection{The algebra $\mathcal O_{mod}(\tilde C)$ of moderate growth functions}\label{subsect:mod:growth}

\begin{defn}
(a) $(\overline C,S)$ is the pair of a nonsingular projective algebraic curve $\overline C$ and a finite set $S$ of complex points of $C$ 
such that $C=\overline C\smallsetminus S$.

(b) For $s\in S$,  $\varphi_s : D\to \overline C$ is an injective holomorphic map, such that 
$0\mapsto s$ and $\varphi_s(D)\cap S=\{s\}$; we set $U_s:=\varphi_s(D)\subset\overline C$, so that $\varphi_s$ corestricts 
to a biholomorphic map  $D\stackrel{\sim}{\to}U_s$.

(c) For $s\in S$, $\varphi_s^\times : D^\times\to C$ is the injective holomorphic map obtained by restriction of $\varphi_s$; we set 
$U_s^\times:=U_s\smallsetminus \{s\}$, so that $\varphi_s^\times$ corestricts to a biholomorphic map $D^\times\stackrel{\sim}{\to}U_s^\times$.

(d) For $s\in S$, we set $\tilde U_s^\times:=p^{-1}(U_s^\times)$ and $X_s:=\pi_0(\tilde U_s^\times)$; for $x\in X_s$, 
we define $\tilde U_{s,x}^\times$ as the connected component of $\tilde U_s^\times$ corresponding to $x$. 

(e) For $s\in S$ and $x\in X_s$, $\tilde\varphi_{s,x}^\times : \tilde D^\times\to\tilde C$ is a holomorphic map with image 
contained in~$\tilde U^\times_{s,x}$, such that $p \circ \tilde\varphi_{s,x}^\times=\varphi_s^\times \circ e$. Then 
$\tilde\varphi_{s,x}^\times$ corestricts to a biholomorphic map $\tilde D^\times\to\tilde U^\times_{s,x}$. 
\end{defn}

The choice of $(\tilde\varphi_{s,x}^\times)_{s\in S,x\in X_s}$ is not unique, but any two choices are related by 
$\tilde\psi_{s,x}^\times=\tilde\varphi_{s,x}^\times\circ\theta^{a_{s,x}}$, where $(a_{s,x})_{s\in S,x\in  X_s}$ is in 
$\oplus_{s\in S}\mathbb Z^{X_s}$. 

For any $s\in S$, the group $\Gamma_C$ acts on $\tilde U_s^\times$, and therefore also on the set $X_s$. 
The latter action is transitive, and the stabilizer of any element $x\in X_s$ is a cyclic group, 
generated by an element $\theta_{s,x}\in\Gamma_C$ which restricts to an automorphism of $\tilde U_{s,x}^\times$, equal to the 
conjugation of $\theta$ by the corestriction of $\tilde\varphi_{s,x}^\times$ to an isomorphism $\tilde D^\times\to\tilde U^\times_{s,x}$. 

\begin{lem}\label{lem15:0208}
There exists a map $c : \Gamma_C\times X_s\to\mathbb Z$ satisfying the identity 
$c(\gamma'\gamma,x)=c(\gamma',\gamma x)+c(\gamma,x)$, such that for any $(\gamma,x)\in\Gamma_C\times
 X_s$, one has $\tilde\varphi_{s,\gamma x}^\times\circ\theta^{c(\gamma,x)}
=\gamma\circ \tilde\varphi_{s,x}^\times$ (equality of holomorphic maps $\tilde D^\times\to\tilde C$).  
\end{lem}

\begin{proof}
The existence of a map $c$ satisfying the identity $\tilde\varphi_{s,\gamma x}^\times\circ\theta^{c(\gamma,x)}
=\gamma\circ \tilde\varphi_{s,x}^\times$ follows from the fact that for any pair of holomorphic maps 
$\alpha,\beta : \tilde D^\times\to\tilde C$ such that $p\circ\alpha=\varphi_s^\times\circ e=p\circ\beta$, there exists 
$n\in\mathbb Z$ such that $\beta=\alpha\circ\theta^n$. Then $\tilde\varphi_{s,\gamma'\gamma x}^\times\circ\theta^{c(\gamma'\gamma,x)}
=\gamma'\gamma\circ \tilde\varphi_{s,x}^\times=\gamma'\circ \tilde\varphi_{s,\gamma x}^\times\circ\theta^{c(\gamma,x)}=
\tilde\varphi_{s,\gamma' \gamma x}^\times\circ\theta^{c(\gamma',\gamma x)}\circ\theta^{c(\gamma,x)}$ which implies 
the identity satisfied by $c$. 
\end{proof}

\begin{lem}\label{lem:1:8:0308}
(a) For any $s\in S$, the space $X_s\times\tilde D^\times$ is equipped with an action of $\Gamma_C$ given by 
$\gamma\cdot(x,d):=(\gamma x,\theta^{c(\gamma,x)}d)$. The induced right action of $\Gamma_C$ on 
$\prod_{x\in X_s}\mathcal O_{hol}(\tilde D^\times)$ is given by $(f_x)_{x\in X_s} \cdot \gamma=(g_x)_{x \in  X_s}$,  
where $g_x:=(\theta^{c(\gamma,x)})^*(f_{\gamma x})$. 

(b) The map $\coprod_{s\in S}X_s\times\tilde D^\times\to \tilde C$, $(s,x,d)\mapsto \tilde\varphi_{s,x}^\times(d)$ is $\Gamma_C$-equivariant, 
its source being 
equipped with the direct sum of the actions of $\Gamma_C$ defined in (a). It induces a $\Gamma_C$-equivariant algebra morphism  
\begin{equation}\label{alg:mor:0208}
\bigoplus_{s\in S}\prod_{x\in X_s}(\tilde\varphi_{s,x}^\times)^* : \mathcal O_{hol}(\tilde C)\to \bigoplus_{s\in S}\prod_{x\in 
 X_s}\mathcal O_{hol}(\tilde D^\times), 
\end{equation}
in which the target is equipped with the direct sum over $s\in S$ of the right actions from (a). 
\end{lem}

\begin{proof}
This follows from Lem. \ref{lem15:0208}. 
\end{proof}

\begin{defn}\label{def:1:8:0208}
$\mathcal O_{mod}%^{(n)}
(\tilde C)$ denotes the subset of $\mathcal O_{hol}(\tilde C)$ of all functions $f$ such that, for 
any $s\in S$ and $x\in X_s$, one has $(\tilde\varphi_{s,x}^\times)^*(f)\in
\mathcal O_{mod}%^{(n)}
(\tilde D^\times)$.
\end{defn}

\begin{prop}\label{lem:1:7:toto}
$\mathcal O_{mod}%^{(n)}
(\tilde C)$ is a subalgebra  of $\mathcal O_{hol}(\tilde C)$, stable under the 
action of $\Gamma_C$. 
\end{prop}

\begin{proof} 
This follows from the equality of $\mathcal O_{mod}%^{(n)}
(\tilde C)$ with the preimage of $\oplus_{s\in S}\prod_{x\in 
 X_s}\mathcal O_{mod}(\tilde D^\times)$ %$V_n$ 
under \eqref{alg:mor:0208}, and from the fact that this is a $\Gamma_C$-stable subalgebra of the target of  
\eqref{alg:mor:0208}. 
 \end{proof}

\subsubsection{Computation of $\mathcal O_{mod}(\tilde C)^{\Gamma_C}$}

\begin{lem}\label{lem:1331:0208}
For any $s\in S$, the diagonal embedding 
$\mathcal O(D^\times)\hookrightarrow 
\prod_{x\in  X_s}\mathcal O_{mod}(\tilde D^\times)$ gives rise to an isomorphism 
$\mathcal O(D^\times)\simeq(\prod_{x\in  X_s}\mathcal O_{mod}(\tilde D^\times))^{\Gamma_C}$. 
\end{lem}

\begin{proof}
The inclusion of the image of the diagonal embedding in  $(\prod_{x\in  X_s}
\mathcal O_{mod}(\tilde D^\times))^{\Gamma_C}$ is evident, let us prove the opposite inclusion. 
Let $y\in X_s$. The stabilizer of $y$ for the action of $\Gamma_C$ on $X_s$ is a cyclic 
group, generated by an element $\theta_y\in\Gamma_C$ such that $c(\theta_y,y)=1$. If now $\underline f:=
(f_x)_{x\in X_s}\in (\prod_{x\in X_s}\mathcal O_{mod}(\tilde D^\times))^{\Gamma_C}$, then 
$\underline f=\underline f\cdot \theta_y$, therefore for any $x\in X_s$,
$f_x=(\theta^{c(\theta_y,x)})^*(f_{\theta_y x})$, which for $x=y$ implies $f_x=\theta^*f_x$. By Lem. \ref{lem:1:5:1304:0208}(b), this 
implies $f_x\in\mathcal O%_{hol}^{(n)}
(D^\times)$. For any $\gamma\in\Gamma$, one has $\underline f=\underline f\cdot \gamma$, 
which given the $\theta$-invariance of each $f_x$, $x\in X_s$ implies $f_x=f_{\gamma x}$ for any 
$x\in  X_s$. Since the action of $\Gamma_C$ on $X_s$ is transitive, this implies that the map
$x\mapsto f_x$ is constant. 
\end{proof}

\begin{prop}\label{lem:2:12:1508}
$\mathcal O_{mod}(\tilde C)^{\Gamma_C}=\mathcal O(C)$. 
\end{prop}

\begin{proof}
Let $n\geq0$. It follows from Prop. \ref{lem:1:7:toto} that \eqref{alg:mor:0208} induces a linear and $\Gamma_C$-equivariant algebra
morphism
\begin{equation}\label{titi:0409}
\mathcal O_{mod}%^{(n)}
(\tilde C)\to %V_n.
\bigoplus_{s\in S}\prod_{x\in X_s}\mathcal O_{mod}(\tilde D^\times). 
\end{equation}
This map restricts to a linear map $\mathcal O_{mod}(\tilde C)^{\Gamma_C}\to (\oplus_{s\in S}\prod_{x\in X_s}
\mathcal O_{mod}(\tilde D^\times))^{\Gamma_C}$. The target is equal to $\oplus_{s\in S}(\prod_{x\in X_s}
\mathcal O_{mod}(\tilde D^\times))^{\Gamma_C}$, which by Lem. \ref{lem:1331:0208} is equal to $\oplus_{s\in S}\mathcal O(D^\times)$. 

On the other hand, $\mathcal O_{mod}(\tilde C)^{\Gamma_C}\subset \mathcal O_{hol}(\tilde C)^{\Gamma_C}=\mathcal O_{hol}(C)$. All this 
implies that $\mathcal O_{mod}(\tilde C)^{\Gamma_C}$ is the preimage of $\oplus_{s\in S}\mathcal O(D^\times)$ by the map 
$$
\oplus_{s\in S}(\varphi_{s}^\times)^* : \mathcal O_{hol}(C)\to\oplus_{s\in S}\mathcal O_{hol}(D^\times), 
$$
which is equal to $\mathcal O(C)$. 
\end{proof}

One checks that the subalgebra $\mathcal O_{mod}(\tilde D)\subset\mathcal O_{hol}(\tilde C)$ is independent of the choice of
the family $(\varphi_s)_{s\in S}$. 

\subsection{The module $\Omega_{mod}(\tilde C)$ of moderate growth differentials}\label{subsect:Omega:mod:2811}

If $M$ is a complex manifold, let $\Omega^\bullet_{hol}(M)$ be the dga of holomorphic differential forms on $M$. Then 
$\Omega^0_{hol}(M)=\mathcal O_{hol}(M)$. The assignment $M\mapsto\Omega^\bullet_{hol}(M)$ is a contravariant functor, so 
a morphism $f : M\to N$ of complex manifolds gives rise to a dga morphism  $f^* : \Omega^\bullet_{hol}(N)\to \Omega^\bullet_{hol}(M)$. 
If $M$ is 1-dimensional, we set $\Omega_{hol}(M):=\Omega^1_{hol}(M)$; then $\Omega_{hol}(M)$ is a $\mathcal O_{hol}(M)$-module, 
equipped with a derivation $d : \mathcal O_{hol}(M)\to\Omega_{hol}(M)$. 

\begin{lem}\label{lem:1151:0908}
$\Omega_{hol}(\tilde D^\times)$ is a free rank $1$ module over $\mathcal O_{hol}(\tilde D^\times)$ generated by $e^*(dz/z)$, so 
$f\mapsto f\cdot e^*(dz/z)$ gives rise to a $\mathcal O_{hol}(\tilde D^\times)$-module isomorphism 
$\mathcal O_{hol}(\tilde D^\times)\stackrel{\sim}{\to}\Omega_{hol}(\tilde D^\times)$. 
\end{lem}

\begin{proof}
This follows from the fact that $dz/z$ is an invertible differential in $\Omega_{hol}(D^\times)$, which implies the same about its 
pull-back by $\tilde D^\times\to D^\times$.  
\end{proof}

\begin{defn}\label{214:1408}
Define $\Omega_{mod}(\tilde D^\times):=\mathcal O_{mod}(\tilde D^\times)\cdot e^*(dz/z)$ as the 
image of $\mathcal O_{mod}(\tilde D^\times)$ under the isomorphism from Lem. \ref{lem:1151:0908}. 
\end{defn}

\begin{lem}\label{1302:0908}
$\Omega_{mod}(\tilde D^\times)$ is a free rank $1$ module over $\mathcal O_{mod}(\tilde D^\times)$. 
\end{lem}

\begin{proof}
The follows from the injectivity of the map $\mathcal O_{hol}(\tilde D^\times)\to
\Omega_{hol}(\tilde D^\times)$, $f\mapsto f\cdot e^*(dz/z)$. 
\end{proof}

The morphism 
$\prod_{s\in S}\prod_{x\in X_s}\varphi_{s,x} : \prod_{s\in S}\prod_{x\in X_s}\tilde D^\times\to\tilde C$ gives rise to a 
morphism 
\begin{equation}\label{mor:1207:0908}
\bigoplus_{s\in S}\prod_{x\in X_s}(\tilde\varphi_{s,x}^\times)^* : 
\Omega_{hol}(\tilde C)\to\bigoplus_{s\in S}\prod_{x\in X_s}\Omega_{hol}(\tilde D^\times).     
\end{equation}

\begin{defn}\label{defn:omega:mod:1408}
Define $\Omega_{mod}(\tilde C)$ as the preimage of $\oplus_{s\in S}\prod_{x\in X_s}\Omega_{mod}(\tilde D^\times)$
under  \eqref{mor:1207:0908}.
\end{defn}

\begin{lem}\label{lem:mod:1408}
$\Omega_{mod}(\tilde C)$ is a module over $\mathcal O_{mod}(\tilde C)$. 
\end{lem}

\begin{proof}
If $f\in\mathcal O_{mod}(\tilde C)$, $\omega\in\Omega_{mod}(\tilde C)$, then $f\omega\in\Omega_{hol}(\tilde C)$. 
If $s\in S$ and $x\in X_s$, then $i_{s,x}^*(f\omega)=i_{s,x}^*(f)i_{s,x}^*(\omega)\in \Omega_{mod}(\tilde D^\times)$
where the last relation follows from $i_{s,x}^*(f)\in\mathcal O_{mod}(\tilde D^\times)$, $i_{s,x}^*(\omega)\in
\Omega_{mod}(\tilde D^\times)$. Therefore $f\omega\in\Omega_{mod}(\tilde C)$. 
\end{proof}

Since $\tilde D^\times$ is simply-connected, the assignment $\omega\mapsto (z\mapsto \int_{\mathrm i}^z\omega)$ is a well-defined linear map 
$\mathrm{int}_{\mathrm i} : \Omega_{hol}(\tilde D^\times)\to\mathcal O_{hol}(\tilde D^\times)$. 

\begin{lem}\label{lem:1750:0908}
The map $\mathrm{int}_{\mathrm i}$ takes $\Omega_{mod}(\tilde D^\times)$ to $\mathcal O_{mod}(\tilde D^\times)$. 
\end{lem}

\begin{proof}
Let $\omega\in\Omega_{mod}(\tilde D^\times)$. Then there exists $f\in\mathcal O_{mod}(\tilde D^\times)$ such that 
$\omega=f\cdot e^*(dz/z)$. Let $n>0$ and $(a,b)\mapsto C_{a,b}$ be the integer and function associated with $f$
(see Def. \ref{1529:0908}). Then $\mathrm{int}_{\mathrm i}(\omega)\in\mathcal O_{hol}(\tilde D^\times)$ is the
function $u\mapsto 2\pi\mathrm i\int_{\mathrm i}^u f(u')du'$. Let $A\geq 0$ and $u=x+\mathrm iy$ with $x\in[-A,A]$; set 
$C_A:=C_{-A,A}$. 
Then the path of integration may be chosen as the sequence of paths $\mathrm i\to \mathrm iy\to u=x+\mathrm iy$, therefore 
$\int_{\mathrm i}^u f(u')du'=\int_{\mathrm i}^{\mathrm iy} f(u')du'+\int_{\mathrm iy}^{x+\mathrm iy} f(u')du'
=\mathrm i\int_1^y f(\mathrm it)dt+\int_{0}^{x} f(t+\mathrm iy)dt$. 

One has $|f(\mathrm it)|\leq C_{A}e^{2\pi nt}$ for $t\in[1,y]$, while $|f(t+\mathrm iy)|\leq C_A e^{2\pi ny}$
for $y\in[0,x]$. Therefore $|\mathrm i\int_1^y f(\mathrm it)dt|\leq C_{A}|\int_1^y e^{2\pi n t}dt|=C_A|e^{2\pi ny}-e^{2\pi n}|/(2\pi n)$, 
and $|\int_{0}^{x} f(t+\mathrm iy) e^{2\pi\mathrm i t}dt|\leq
AC_A e^{2\pi ny}$. Therefore 
\begin{equation}\label{1714:0908}
\bigg|\mathrm i\int_{\mathrm i}^u f(u')du'\bigg|\,\leq\, C_A\,\frac{|e^{2\pi ny}-e^{2\pi n}|}{2\pi n}+AC_A e^{2\pi ny}. 
\end{equation}
If $y\geq 1$, the right-hand side of \eqref{1714:0908} is $\leq C_A(A+1/(2\pi n))e^{2\pi ny}$ (expressing the absolute value as its
argument as the latter is $\geq 0$ and bounding the resulting expression from above by removing the negative term); if 
$y<1$, the right-hand side of \eqref{1714:0908} is $\leq C_Ae^{2\pi n}/(2\pi n)+AC_A e^{2\pi ny}$ (expressing the absolute value as 
the negative of its argument as the latter is $\leq 0$ and bounding the resulting expression form above by removing the negative term)
which is $\leq C_A(e^{2\pi n}/(2\pi n))e^{2\pi ny}+AC_A e^{2\pi ny}$ (as $e^{2\pi ny}\geq 1$), therefore 
the right-hand side of \eqref{1714:0908} is $\leq (A+(e^{2\pi n}/(2\pi n)))C_Ae^{2\pi ny}$. 

Set $D_A:=(A+(e^{2\pi n}/(2\pi n)))C_A$, then one obtains $|\mathrm i\int_{\mathrm i}^u f(u')du'|\leq D_Ae^{2\pi ny}$
for every $u\in[-A,A]+\mathrm i\mathbb R_+^\times$, and therefore 
$|\int_{\mathrm i}^u\omega|\leq 2\pi  D_Ae^{2\pi ny}$
for any $u\in[-A,A]+\mathrm i\mathbb R_+^\times$. 
This shows that $u\mapsto\int_{\mathrm i}^u\omega$ satisfies the condition of Def. \ref{1529:0908} with the pair 
$(n,(a,b)\mapsto 2\pi D_{\mathrm{max}(|a|,|b|)})$, and therefore belongs to $\mathcal O_{mod}(\tilde D^\times)$. 
\end{proof}

Fix a point $x_0\in\tilde C$. 

\begin{lem}\label{lem:2:19:0908}
The map $\mathrm{int}_{x_0}$ (see \S\ref{sect:121:2710}) takes $\Omega_{mod}(\tilde C)$ to $\mathcal O_{mod}(\tilde C)$. 
\end{lem}

\begin{proof}
Let $\omega\in\Omega_{mod}(\tilde C)$. Let $s\in S$ and $x\in X_s$. By additivity of the integral w.r.t. the composition of paths, 
one has 
\begin{equation}\label{1749:0908}
\textstyle (\tilde\varphi_{s,x}^\times)^*(\mathrm{int}_{x_0}(\omega))=\int_{x_0}^{\tilde\varphi_{s,x}^\times(\mathrm i)}\omega
+\mathrm{int}_{\mathrm i}((\tilde\varphi_{s,x}^\times)^*\omega) 
\end{equation}
(equality in $\mathcal O_{hol}(\tilde D^\times)$), where $\int_{x_0}^{\tilde\varphi_{s,x}^\times(\mathrm i)}\omega$ belongs to $\mathbb C$. 
Since $\omega\in\Omega_{mod}(\tilde C)$, one has $(\tilde\varphi_{s,x}^\times)^*\omega\in\Omega_{mod}(\tilde D^\times)$; it then follows from 
Lem. \ref{lem:1750:0908} that $\mathrm{int}_{\mathrm i}((\tilde\varphi_{s,x}^\times)^*\omega)\in\mathcal O_{mod}(\tilde D^\times)$. 
Since $\mathcal O_{mod}(\tilde D^\times)$ is an algebra containing $\mathbb C$, it follows that the right-hand side of \eqref{1749:0908}  
belongs to $\mathcal O_{mod}(\tilde D^\times)$. Eq. \eqref{1749:0908} then implies that $(\tilde\varphi_{s,x}^\times)^*
(\mathrm{int}_{x_0}(\omega))\in \mathcal O_{mod}(\tilde D^\times)$. As this holds for any $s\in S$ and $x\in X_s$, one derives 
$\mathrm{int}_{x_0}(\omega) \in \mathcal O_{mod}(\tilde C)$.  
\end{proof}

\subsection{Iterated integrals and moderate growth functions}\label{sect:iis:mod:0411}

\begin{lem}\label{lem:2:22:1408}
(a) The maps  $e^* : \mathcal O(D^\times)\to\mathcal O_{hol}(\tilde D^\times)$ and $e^* : \Omega(D^\times)\to\Omega_{hol}(\tilde D^\times)$ 
have their images respectively contained in $\mathcal O_{mod}(\tilde D^\times)$ and $\Omega_{mod}(\tilde D^\times)$. 

(b) The maps $p^* : \mathcal O(C)\to\mathcal O_{hol}(\tilde C)$ and $p^* : \Omega(C)\to\Omega_{hol}(\tilde C)$ 
have their images respectively contained in $\mathcal O_{mod}(\tilde C)$ and $\Omega_{mod}(\tilde C)$. 
\end{lem}

\begin{proof}
(a) Let $f\in\mathcal O(D^\times)$. There there exist $n\geq1$ and $g\in\mathcal O(D)$ such that $f=g/z^n$. Then $e^*f=e^*g/(u\mapsto 
e(nu))$. The function $g$ is bounded on $D$, therefore there exists $C\in\mathbb R_+$ such that for any $u\in\tilde D^\times$, 
$|(e^*f)(u)|\leq C|1/e(nu)|=C e^{2\pi n y}$. So $e^*f$ satisfies the condition of Def.~\ref{1529:0908} with the pair $(n,(a,b)\mapsto C)$, 
so $e^*f\in\mathcal O_{mod}(D^\times)$. 

Let $\omega\in\Omega(D^\times)$. There exists $f\in\mathcal O(D^\times)$ such that $\omega=f\cdot (dz/z)$. Then $e^*\omega=e^*f
\cdot e^*(dz/z) \in \mathcal O_{mod}(\tilde D^\times)\cdot e^*(dz/z)=\Omega_{mod}(\tilde D^\times)$ (see Def. \ref{214:1408}). 

(b) Let $f\in\mathcal O(C)$ and $s\in S$, $x\in X_t$. Then it follows from $\varphi_s^\times\circ e=p\circ\tilde\varphi_{s,x}^\times$
that $(\tilde\varphi_{s,x}^\times)^*p^*f=e^*(\varphi_s^\times)^*f$. Then $(\varphi_s^\times)^*f\in \mathcal O(D^\times)$, and (a) then implies 
$e^*(\varphi_s^\times)^*f\in \mathcal O_{mod}(\tilde D^\times)$. It follows that $(\tilde\varphi_{s,x}^\times)^*p^*f\in 
\mathcal O_{mod}(\tilde D^\times)$ for any pair $(s,x)$, therefore by Def. \ref{def:1:8:0208}, that $p^*f\in\mathcal O_{mod}(\tilde C)$. 

Similarly, let $\omega\in\Omega(C)$ and $s\in S$, $x\in X_s$. Then $(\tilde\varphi_{s,x}^\times)^*p^*\omega=e^*(\varphi_s^\times)^*\omega$, 
and since $(\varphi_s^\times)^*\omega\in \Omega(D^\times)$, (a) implies $e^*(\varphi_s^\times)^*\omega\in \Omega_{mod}(\tilde D^\times)$. 
It follows that $(\tilde\varphi_{s,x}^\times)^*p^*\omega\in \Omega_{mod}(\tilde D^\times)$ for any pair $(s,x)$, therefore by 
Def. \ref{defn:omega:mod:1408}, that $p^*f\in\Omega_{mod}(\tilde C)$. 
\end{proof}

\begin{prop}\label{lem:INOMOD:1408}
For any $x_0\in\tilde C$, the image of $I_{x_0} : \mathrm{Sh}(\Omega(C))\to\mathcal O_{hol}(\tilde C)$ is contained 
in $\mathcal O_{mod}(\tilde C)$, therefore it induces an algebra morphism $I_{x_0} : \mathrm{Sh}(\Omega(C))\to\mathcal O_{mod}(\tilde C)$. 
\end{prop}

\begin{proof}
We prove inductively on $n\geq 0$ that $\mathrm{im}(I_{x_0}^{(n)})\subset \mathcal O_{mod}(\tilde C)$. This is obvious if $n=0$. 
Assume that $\mathrm{im}(I_{x_0}^{(n-1)})\subset\mathcal O_{mod}(\tilde C)$. Let $\omega_1,\ldots,\omega_n\in\Omega(C)$. 
By the induction hypothesis, $I_{x_0}^{(n-1)}([\omega_1|\ldots|\omega_{n-1}])\in\mathcal O_{mod}(\tilde C)$, and by Lem. 
\ref{lem:2:22:1408}(b), 
$p^*\omega_n\in\Omega_{mod}(\tilde C)$; since $\Omega_{mod}(\tilde C)$ is a module over $\mathcal O_{mod}(\tilde C)$ 
(see Lem. \ref{lem:mod:1408}), then $I_{x_0}^{(n-1)}([\omega_1|\ldots|\omega_{n-1}])\cdot p^*\omega_n\in\Omega_{mod}(\tilde C)$. Lem.~\ref{lem:2:19:0908} then implies $\mathrm{int}_{x_0}(I_{x_0}^{(n-1)}([\omega_1|\ldots|\omega_{n-1}]) \cdot p^*\omega_n) 
\in\mathcal  O_{mod}(\tilde C)$. By \eqref{IND:1008}, the latter term is $I_{x_0}^{(n)}([\omega_1|\ldots|\omega_n])$, therefore 
$I_{x_0}^{(n)}([\omega_1|\ldots|\omega_n])\in\mathcal O_{mod}(\tilde C)$, so $\mathrm{im}(I_{x_0}^{(n)})\subset \mathcal O_{mod}(\tilde C)$. 
\end{proof}

\section{The isomorphism of filtered algebras $f_{J,x_0} : 
F_\bullet\mathrm{Sh}(\mathrm{H^{dR}_C})\otimes\mathcal O(C) \to F_\bullet\mathcal O_{mod}(\tilde C)$}\label{sect:3:2908}

We start this section by reminders on filtrations (§\ref{background:filtrations}). We compute the restricted dual 
$(\mathbb C\Gamma_C)'$ of 
the Hopf algebra $\mathbb C\Gamma_C$ in §\ref{sect:32:2811}. 
%and $F_\infty\mathrm{Sh}(\mathrm H_C)$ of $\mathrm{Sh}(\mathrm H_C)$ in §\ref{sect:33:2811}. 
We define a pairing~$p_{J,x_0}$ between this Hopf algebra and $F_\infty\mathrm{Sh}(\mathrm H_C)=\mathrm{Sh}(\mathrm H_C)$
in §\ref{proof:prop:iso:B:a:1} and prove in §\ref{sect:35:2811} that the induced Hopf algebra morphism $\nu(p_{J,x_0})$ is an 
isomorphism (Prop. \ref{prop:1508}). We define a filtered algebra morphism $f_{J,x_0}$ in §\ref{sect:36:2811} and use Prop. \ref{prop:1508} 
and Prop. \ref{lem:B:15:1508}(b) to show in Prop. \ref{prop:3:15:0409} that it is an isomorphism. In~§\ref{sect:37:2811}, we use the material 
of this proof together with Prop. \ref{lem:B:15:1508}(c) to prove a filtered formality statement. 

\subsection{Background on filtrations}\label{background:filtrations}

A {\it vector space filtration} of a $\mathbb C$-vector space $M$ is an increasing collection $F_\bullet M=(F_nM)_{n\geq 0}$ of vector 
subspaces of $M$. The filtration $F_\bullet M$ is called {\it exhaustive} if and only if $\cup_{n\geq 0}F_nM=M$. 
If $F_\bullet M$ is a filtration of $M$, then $F_\infty M:=\cup_{n\geq 0}F_nM$ is a vector subspace of $M$, called the {\it total vector 
space} of the filtration; then $F_\bullet M$ is an exhaustive filtration of this vector subspace. The {\it associated graded} of 
$F_\bullet M$ is the graded $\mathbb C$-vector space $\mathrm{gr}_\bullet(M):=\oplus_{n\geq 0}\mathrm{gr}_n(M)$, where 
$\mathrm{gr}_n(M)=F_nM/F_{n-1}M$ (with $F_{-1}M:=0$). 

Let $f : M\to N$ be a morphism of $\mathbb C$-vector spaces and $F_\bullet M$ be a filtration of $M$. Then~$f$ is said to be 
{\it compatible} with a filtration $F_\bullet N$ on $N$ if and only if $f(F_nM)\subset F_nN$ for any $n\geq0$. This is in particular the 
case if $F_\bullet N$ is the {\it image of $F_\bullet M$ by $f$} (denoted $f(F_\bullet M)$), defined by $F_nN:=f(F_nM)$ for any 
$n\geq 0$; one then has $F_\infty N=f(F_\infty M)$.   

\begin{lem}\label{graded:crit:iso}
If $M,N$ are filtered vector spaces and $f : M\to N$ is a linear map compatible with the filtrations, and such that  
$\mathrm{gr}_\bullet(f) : \mathrm{gr}_\bullet(M)\to\mathrm{gr}_\bullet(N)$ is an isomorphism of graded vector spaces,
then the maps $F_nf : F_nM\to F_nN$ for any $n\geq 0$, as well as $F_\infty f : F_\infty M\to F_\infty N$, are linear isomorphisms. 
\end{lem}

\begin{proof} Let us prove the first statement by induction on $n\geq0$. The map $F_0f : F_0M\to F_0N$ is the composition of the 
isomorphisms $F_0M\simeq \mathrm{gr}_0M$, $\mathrm{gr}_0N\simeq F_0N$ and $\mathrm{gr}_0(f)$, which is an isomorphism, therefore 
$F_0f : F_0M\to F_0N$ is an isomorphism of vector spaces. Let $n\geq0$ and assume that $F_nf$ is an isomorphism of vector spaces. 
The image of $\mathrm{Ker}(F_{n+1}f)\subset F_{n+1}M\to \mathrm{gr}_{n+1}M$ is contained in $\mathrm{Ker}(\mathrm{gr}_{n+1}f)$ which is 
$0$ by assumption, so this image is $0$, which implies that $\mathrm{Ker}(F_{n+1}f)\subset F_nM$; the restriction of $F_{n+1}f$ to 
$F_nM$ coincides with $F_nf$, which by the induction hypothesis is injective, therefore $\mathrm{Ker}(F_{n+1}f)=0$, so $F_{n+1}f$ is 
injective. For $y\in F_{n+1}N$, let $\overline y$ be its image in $\mathrm{gr}_{n+1}N$. By the surjectivity of 
$\mathrm{gr}_{n+1}(f)$, there exists $\alpha\in \mathrm{gr}_{n+1}M$ with image $\overline y$ by $\mathrm{gr}_{n+1}(f)$. 
Then if $x\in F_{n+1}M$ is any lift of $\alpha$, one has $F_{n+1}(x)\equiv y$ mod $F_nN$. Then $y-F_{n+1}(x)\in F_nN$. 
Since $F_nf : F_nM\to F_nN$ is surjective, there exists $x_0\in F_nM$ such that $F_n(x_0)=y-F_{n+1}(x)$. Then 
$y=F_n(x+x_0)$, which implies the surjectivity of $F_{n+1}f$. It follows that $F_{n+1}f$ is an isomorphism, proving the induction.  

One has $\mathrm{Ker}(F_\infty f)=\cup_{n\geq 0}(\mathrm{Ker}(F_\infty f)\cap F_nM)
=\cup_{n\geq 0}\mathrm{Ker}(F_n f)=0$, where the first equality follows from $F_\infty M=\cup_{n\geq0}F_nM$ and the last equality 
follows from the injectivity of $F_n f$ for $n\geq 0$; this shows the injectivity of $F_\infty f$. 

For any $n\geq 0$, one has $\mathrm{Im}(F_\infty M)\supset\mathrm{Im}(F_nM)$, and $\mathrm{Im}(F_nM)=F_nN$ by the surjectivity of $F_nf$. Then 
$\mathrm{Im}(F_\infty M)\supset\cup_{n\geq 0}F_nN=F_\infty N$, which shows the surjectivity of $F_\infty f$. It follows that $F_\infty f$
is a linear isomorphism. 
\end{proof}

If $F_\bullet M$ and $F_\bullet N$ are filtrations of $\mathbb C$-vector spaces $M$ and $N$, then a filtration $F_\bullet(M\otimes N)$
of their tensor product $M\otimes N$ is defined by $F_n(M\otimes N):=\sum_{p+q=n}F_pM\otimes F_qN$; we denote it by 
$F_\bullet M\otimes F_\bullet N$, and we call it the {\it tensor product of $F_\bullet M$ and $F_\bullet N$}. 

An {\it algebra filtration} of a $\mathbb C$-algebra $A$ is a vector space filtration $F_\bullet A$ of $A$, such that 
$F_nA\cdot F_mA\subset F_{n+m}A$ for $n,m\geq 0$. Then $F_\infty A$ is a subalgebra of $A$, called the {\it total algebra} of 
the filtration; $\mathrm{gr}_\bullet(A)$ is then a graded algebra. 
If $f : A\to B$ is an algebra morphism  and $F_\bullet A$ is a filtration of $A$, then $f(F_\bullet A)$ is an algebra 
filtration of $B$. If $F_\bullet A$ and $F_\bullet B$ are filtrations of $\mathbb C$-algebras $A$ and $B$, then $F_\bullet A\otimes F_\bullet B$
is an algebra filtration of $A\otimes B$. 

An example of a filtration of an algebra $A$ is the {\it trivial filtration}
$F^{triv}_\bullet A$ defined by $F^{triv}_nA=A$ for any $n\geq0$. If $A$ has a unit, another example is the {\it unit filtration}
$F^{unit}_\bullet A$ defined by $F^{unit}_0A=\mathbb C1$ and $F^{unit}_nA=A$ for any $n>0$. 

Similarly, a Hopf algebra filtration of a Hopf algebra $H$ with coproduct $\Delta_H$ is a vector space filtration $F_\bullet H$ of $H$, 
which is an algebra filtration and such that $\Delta_H(F_nH) \subset \sum_{p+q=n}F_pH\otimes F_qH$ for any $n\geq0$. 
Then $F_\infty H$ is a Hopf subalgebra of $H$ and $\mathrm{gr}_\bullet(H)$ is a graded Hopf algebra. 

\subsection{Computation of $(\mathbb C\Gamma_C)'$}\label{sect:32:2811}

In \S\ref{app:C:1:2308}, we recall the category $\mathbf{CHA}$ of complete Hopf algebras (CHAs) 
and the functor $\mathbf{HA}_{coco}\to\mathbf{CHA}$, $H\mapsto H^\wedge$ with source the category $\mathbf{HA}_{coco}$ of cocommutative 
Hopf algebras. 

\begin{lem}\label{lem:iso:1530:2807}
If $\Gamma$ is a free group, there is an isomorphism $(\mathbb C\Gamma)^\wedge\simeq\hat T(\Gamma^{ab}\otimes\mathbb C)$ of CHAs, where
$\Gamma^{ab}$ is the abelianization of $\Gamma$ and for $V$ a vector space, $\hat T(V)$ is the CHA defined as the degree completion of 
the tensor algebra of $V$, where the elements of $V$ are primitive. 
\end{lem}

\begin{proof}
By assumption, $\Gamma$ is the free group over a set $X$. Let $(\gamma_x)_{x\in X}$ be the corresponding generating family. 
Then $\Gamma^{ab}\otimes\mathbb C=\mathbb CX$;  let $(v_x)_{x\in X}$ be the canonical generating family of $\mathbb CX$. 
The assignment $\gamma_x\mapsto \mathrm{exp}(v_x)$ (defined as $\sum_{n\geq 0}v_x^{\otimes n}/n!$) for $x\in X$ defines a group morphism 
$\Gamma\to \hat T(\mathbb CX)^\times$, therefore an 
algebra morphism $\mathbb C\Gamma\to \hat T(\mathbb CX)$, which is checked to the compatible with coproducts. It is compatible with 
augmentations, therefore gives rise to a CHA morphism $(\mathbb C\Gamma)^\wedge\to \hat T(\mathbb CX)$. The assignment
$v_x\mapsto \mathrm{log}(\gamma_x)$ (defined as $\sum_{n\geq1}(-1)^{n+1}(\gamma_x-1)^n/n$) for $x\in X$ defines a linear map 
$\mathbb CX\to(\mathbb C\Gamma)^\wedge$, therefore an algebra morphism $T(\mathbb CX)\to(\mathbb C\Gamma)^\wedge$, which 
is checked to be compatible with coproducts. It is compatible with augmentations, therefore giving rise to a CHA morphism 
 $\hat T(\mathbb CX)\to(\mathbb C\Gamma)^\wedge$. The two constructed CHA morphisms can be checked to be inverses of each other. 
\end{proof}

In \S\ref{app:C:1:2308}, we define a subcategory $\mathbf{HA}_{fd}$ of finite-dimensional Hopf algebras of the category $\mathbf{HA}$ 
of Hopf algebras, and a duality functor $\mathbf{HA}_{fd}\to \mathbf{HA}$, $H\mapsto H'$. 
When $H$ is the group algebra of a finitely generated group, $H'$ is as in Def. \ref{defn:1:5}. 

\begin{lem}\label{comp:CGamma:prime} 
Let $\Gamma$ be the free group over a finite set of generators (for example, $\Gamma=\Gamma_C$). 
Then $\mathbb C\Gamma$ is an object in $\mathbf{HA}_{fd}$
and there is a Hopf algebra isomorphism  $(\mathbb C\Gamma)'\simeq\mathrm{Sh}((\Gamma^{ab}\otimes\mathbb C)^*)$ (we denote by 
$V^*$ the dual of a vector space $V$). 
\end{lem}

\begin{proof} In this proof, we set $V:=\Gamma^{ab}\otimes\mathbb C$. The first statement follows from the finite generation of $\Gamma$. 
It implies that $(\mathbb C\Gamma)^\wedge$ is an object in
$\mathbf{CHA}_{fd}$, and by Lem. \ref{lem:5:12:1927}, the duals $(\mathbb C\Gamma)'$ and $((\mathbb C\Gamma)^\wedge)'$ are well-defined 
isomorphic objects in $\mathbf{HA}_{coco}$. 

By Lem. \ref{lem:iso:1530:2807}, the CHAs $(\mathbb C\Gamma)^\wedge$ and $\hat T(V)$ are isomorphic. Since 
$(\mathbb C\Gamma)^\wedge$ is an object in $\mathbf{CHA}_{fd}$, so is $\hat T(V)$, so $(\mathbb C\Gamma)^\wedge$ and $\hat T(V)$
are isomorphic objects in $\mathbf{CHA}_{fd}$. By Lem. \ref{2237:1107}(b), this gives rise to an isomorphism  
$((\mathbb C\Gamma)^\wedge)'\simeq \hat T(V)'$ in $\mathbf{HA}_{coco}$. Since $V$ is finite dimensional, there is 
an isomorphism $\hat T(V)'\simeq\mathrm{Sh}(V^*)$. 

The result follows by composition of these isomorphisms. One knows that $\Gamma_C$  is a free group over a finite set of 
generators, which implies that it gives an example of the above statements. 
\end{proof}

\begin{rem}
Lem. \ref{comp:CGamma:prime} is proved in \cite{BGF}, Example 3.229 when $|X|=2$. 
\end{rem}

\subsection{A Hopf pairing $p_{J,x_0} : \mathrm{Sh}(\mathrm H_C) \otimes\mathbb C\Gamma_C\to
\mathbb C$}\label{proof:prop:iso:B:a:1}

Until the end of \S\ref{sect:3:2908}, an element $(J,x_0)\in \mathrm{MC}_{nd}(C)\times\tilde C$ will be fixed.

\begin{defn}\label{def:pairing:sigma:x_0:2308}
$p_{J,x_0}$ is the linear map $\mathbb C\Gamma_C\otimes\mathrm{Sh}(\mathrm H_C)\to\mathbb C$ such that 
$\gamma\otimes a\mapsto I_{x_0}(J_*(a))(\gamma x_0)$
for $\gamma\in\Gamma_C$, $a\in\mathrm{Sh}(\mathrm H_C)$.
\end{defn}

\begin{lem}\label{HA:pairing}
The pairing $p_{J,x_0} : \mathbb C\Gamma_C\otimes\mathrm{Sh}(\mathrm H_C)\to\mathbb C$
is a Hopf pairing (in the sense of \S\ref{sect:A:3:3108}).
\end{lem}

\begin{proof}
Let $p_{J,x_0} : \mathbb C\Gamma_C\otimes\mathrm{Sh}(\mathrm H_C)\to\mathbb C$ be the map defined in 
Def. \ref{HA:pairing}. %Prop. \ref{prop:form:2905}, (a). 
For any $\gamma,\gamma'\in\Gamma_C$ and $a\in\mathrm{Sh}(\mathrm H_C)$, one has 
\begin{align}\label{Hopf:pairing:a}
&p_{J,x_0}(\gamma\gamma',a)=I_{x_0}(%\underline 
J_*(a))(\gamma\gamma'x_0)
=I_{x_0}(%\underline 
J_*(a^{(1)}))(\gamma x_0)I_{\gamma x_0}(%\underline 
J_*(a^{(2)}))(\gamma\gamma'x_0)
\\ & \nonumber =I_{x_0}(%\underline 
J_*(a^{(1)}))(\gamma x_0)I_{x_0}(%\underline 
J_*(a^{(2)}))(\gamma'x_0)
=p_{J,x_0}(\gamma,a^{(1)})p_{J,x_0}(\gamma',a^{(2)}), 
\end{align}
where the second equality follows from Lem.  \ref{lem:chain:rule:2308} %\ref{lem:IC:copdt} 
and the fact that $%\underline 
J_* : \mathrm{Sh}(\mathrm H_C)\to \mathrm{Sh}(\Omega(C))$
is a Hopf algebra morphism, the third equality follows from the invariance of the image of $I_{x_0}$ by the diagonal action of 
$\Gamma_C$ (see Lem. \ref{lem:invce:2308}), % \ref{lem:IC:Gamma:invt}), 
and the first and last equalities follow from definitions. 

Denote by $\shuffle$ the product in the algebra $\mathrm{Sh}(\mathrm H_C)$. Let $\gamma\in\Gamma_C$ and $a,a'\in \mathrm{Sh}(\mathrm H_C)$. 
Then 
\begin{align}\label{Hopf:pairing:b}
& p_{J,x_0}(\gamma,a\shuffle a')=I_{x_0}(%\underline 
J_*(a\shuffle a'))(\gamma x_0)
=I_{x_0}(%\underline 
J_*(a))(\gamma x_0)I_{x_0}(%\underline 
J_*(a'))(\gamma x_0)
\\ & \nonumber 
=p_{J,x_0}(\gamma,a)p_{J,x_0}(\gamma,a')
=p_{J,x_0}(\gamma^{(1)},a)p_{J,x_0}(\gamma^{(2)},a')
\end{align}
where the first and third equalities follow from definitions, the second equality follows from the facts that 
$%\underline 
J_* : \mathrm{Sh}(\mathrm H_C)\to \mathrm{Sh}(\Omega(C))$ and $I_{x_0} : \mathrm{Sh}(\Omega(C))\to
\mathcal O_{hol}(\tilde C)$ are algebra morphisms, and the last equality follows from the group-likeness of $\gamma$ for 
the coproduct of $\mathbb C\Gamma_C$. 

The equalities \eqref{Hopf:pairing:a} and \eqref{Hopf:pairing:b} imply the statement. 
\end{proof}

\subsection{Proof that $\nu(p_{J,x_0}) : F_\infty\mathrm{Sh}(\mathrm H_C)\to(\mathbb C\Gamma_C)'$ is a Hopf algebra isomorphism}\label{sect:35:2811}

By Lem. \ref{lem:4:11:0307}, the Hopf algebra pairing $p_{J,x_0}$ (see Lem. \ref{HA:pairing}) gives rise to a  
Hopf algebra morphism $\nu(p_{J,x_0}) : F_\infty\mathrm{Sh}(\mathrm H_C)\to(\mathbb C\Gamma_C)'$, which we now study. 

\subsubsection{Construction of a Hopf algebra morphism $\mathrm{Sh}(\mathrm H_C)\to\mathrm{Sh}((\mathrm H_C^{\mathrm{B}})^*)$}

Let $\mathrm{H}_1(C,\mathbb Z)$ be the first singular homology group of $C$ with integer coefficients, and 
let us set $\mathrm{H}_C^{\mathrm{B}}:=\mathrm{H}_1(C,\mathbb Z)\otimes\mathbb C$. 

\begin{lem}\label{2123:3107}
There is a Hopf algebra isomorphism $(\mathbb C\Gamma_C)'\simeq \mathrm{Sh}((\mathrm{H}_C^{\mathrm{B}})^*)$. 
\end{lem}

\begin{proof}
Since $C$ is an affine curve, the group $\Gamma_C$ is free. Lem. \ref{comp:CGamma:prime} then implies that $(\mathbb C\Gamma_C)'$
is isomorphic to $\mathrm{Sh}((\Gamma_C^{ab}\otimes\mathbb C)^*)$. The choice of a point $x_0$ in $\tilde C$ induces an isomorphism
$\Gamma_C\simeq\pi_1(C,x_0)$, whose conjugation class is independent on this choice; this isomorphism induces an isomorphism 
$\Gamma_C^{ab}\simeq\pi_1(C,x_0)^{ab}=\mathrm{H}_1^{\mathrm{B}}(C,\mathbb Z)$ also independent on this choice, from which one derives 
an isomorphism $\Gamma_C^{ab}\otimes\mathbb C\simeq \mathrm{H}_C^{\mathrm{B}}$. 
\end{proof}

\begin{defn}\label{def:5:13:0308}
$q_{J,x_0} : \mathrm{Sh}(\mathrm H_C)\to\mathrm{Sh}((\mathrm{H}_C^{\mathrm{B}})^*)$ is the Hopf algebra 
morphism obtained by composition of 

(a) the Hopf algebra isomorphism $\mathrm{Sh}(\mathrm H_C)\stackrel{\sim}{\to}F_\infty\mathrm{Sh}(\mathrm H_C)$ (see Lem. \ref{lem:comp:Sh:heart:1628} with $V=\mathrm H_C$), 

(b) the Hopf algebra morphism $\nu(p_{J,x_0}) : F_\infty\mathrm{Sh}(\mathrm H_C)\to(\mathbb C\Gamma_C)'$, 
%arising by Lem. \ref{lem:4:11:0307} from the Hopf algebra pairing $p_{J,x_0}$ (see Lem. \ref{HA:pairing}), 

(c) the Hopf algebra isomorphism $(\mathbb C\Gamma_C)'\simeq \mathrm{Sh}((\mathrm{H}_C^{\mathrm{B}})^*)$ (see Lem. \ref{2123:3107}). 
\end{defn}

\subsubsection{A criterion for a Hopf algebra morphism $\mathrm{Sh}(V)\to\mathrm{Sh}(W)$ to be an isomorphism}

\begin{lem}\label{lem:514:0108}
Let $V,W$ be vector spaces and let $f : \mathrm{Sh}(V)\to \mathrm{Sh}(W)$ be a Hopf algebra morphism. 
Then $f(V)\subset W$, where $V,W$ are the degree $1$ subspaces of $\mathrm{Sh}(V),\mathrm{Sh}(W)$. Denote by 
$\mathrm{gr}_1(f) : V\to W$ the corresponding linear map. Then $f$ is a Hopf algebra isomorphism if and only if 
$\mathrm{gr}_1(f)$ is a vector space isomorphism. 
\end{lem}

\begin{proof}
 Since $f$ is an Hopf algebra morphism, Lem. \ref{prop:LA:filt}(c) implies that it induces a linear map
 $F_nf : F_n\mathrm{Sh}(V)\to F_n\mathrm{Sh}(W)$ for any $n\geq 0$. 
 When $n=1$, $F_1f$ is a linear map $\mathbb C\oplus V\to\mathbb C\oplus W$. The compatibility 
 of $F_1f$ with the units and counits on both sides implies that $F_1f$ is the direct 
 sum of $id_{\mathbb C} : \mathbb C\to \mathbb C$ and a linear map $V\to W$, which can be identified with the 
 associated graded of $f$ for the filtration $F_\bullet$. 

If $f$ is an Hopf algebra morphism, then for each $n\geq 0$, $F_nf : F_n\mathrm{Sh}(V)\to 
F_n\mathrm{Sh}(W)$ is a linear isomorphism, which when $n=1$ implies the same for $id_{\mathbb C}
\oplus \mathrm{gr}_1(f)$, which implies that $\mathrm{gr}_1(f)$ is a linear isomorphism. 

Assume now that $f : \mathrm{Sh}(V)\to \mathrm{Sh}(W)$ is a Hopf algebra morphism such that $\mathrm{gr}_1(f)$ 
is a linear isomorphism. The associated graded map $\mathrm{gr}_\bullet f : \mathrm{gr}_\bullet\mathrm{Sh}(V)
\to \mathrm{gr}_\bullet\mathrm{Sh}(W)$ can be identified, under the canonical isomorphisms 
$\mathrm{gr}_\bullet\mathrm{Sh}(X)\simeq \mathrm{Sh}(X)$ for $X=V,W$ (see Lem. \ref{lem:comp:Sh:heart:1628}), 
with $\mathrm{Sh}(\mathrm{gr}_1(f))$, 
which is an isomorphism of graded vector spaces. Since the filtrations $F_\bullet$ in the 
source and target are exhaustive and by Lem. \ref{graded:crit:iso}, this implies that $f$ is an isomorphism. 
\end{proof}

\subsubsection{Isomorphism status of the linear map $\mathrm{gr}_1(q_{J,x_0}) : 
\mathrm H_C\to (\mathrm{H}_C^{\mathrm{B}})^*$}

For $H$ a Hopf algebra, denote by $H_+$ the kernel of its counit morphism. 

\begin{lem}\label{lem:basic:pairing}
(a) The pairing $(\mathbb C\Gamma_C)_+\otimes \Omega(C)\to\mathbb C$ given by  
$(\gamma-1)\otimes\omega\mapsto \int_{x_0}^{\gamma x_0}\omega$ for $\gamma\in\Gamma_C$, $\omega\in\Omega$, 
is independent on $x_0$. It factors through a pairing 
\begin{equation}\label{basic:pairing}
\mathrm{H}_C^{\mathrm{B}}
\otimes \mathrm H_C=(\mathbb C\Gamma_C)_+/(\mathbb C\Gamma_C)_+^2\otimes (\Omega(C)/d\mathcal O(C))\to\mathbb C.     
\end{equation}

(b) The linear map $\mathrm H_C\to (\mathrm{H}_C^{\mathrm{B}})^*$ induced by \eqref{basic:pairing} is equal to 
$\mathrm{gr}_1(q_{J,x_0}) : \mathrm H_C\to (\mathrm{H}_C^{\mathrm{B}})^*$.  

(c) The linear map $\mathrm{gr}_1(q_{J,x_0}) : \mathrm H_C\to (\mathrm{H}_C^{\mathrm{B}})^*$ 
is an isomorphism.
\end{lem}

\begin{proof} (a) If $x_0,x_1\in\tilde C$, $\gamma\in\Gamma_C$ and $\omega\in\Omega(C)$, then 
$\int_{x_1}^{\gamma x_1}\omega-\int_{x_0}^{\gamma x_0}\omega
=(\int_{x_1}^{x_0}-\int_{\gamma x_1}^{\gamma x_0})\omega=0$ by the $\Gamma_C$-invariance of $\omega$; this implies the 
claimed independence.  
If $\gamma,\gamma'\in\Gamma_C$ and $\omega\in\Omega(C)$, then $(\int_{x_0}^{\gamma\gamma'x_0}-\int_{x_0}^{\gamma x_0}
-\int_{x_0}^{\gamma'x_0})\omega=(\int_{\gamma x_0}^{\gamma\gamma'x_0}-\int_{x_0}^{\gamma' x_0})\omega=0$
by the same reason; since the elements $(\gamma-1)(\gamma'-1)$ generate $(\mathbb C\Gamma_C)_+^2$, this implies the 
claimed factorization. The equality follows from $\mathrm{gr}^1(\mathbb C\Gamma_C)=\Gamma_C^{ab}\otimes\mathbb C
=\mathrm H_C^{\mathrm{B}}$ (see \cite{Q2}). 

(b) Let $O,H$ be Hopf algebras with $\mathrm{gr}^1(H)$ finite dimensional and let $p : O\otimes H\to\mathbb C$ be a Hopf algebra pairing. 
By Lem. \ref{lem:4:11:0307}, $p$ gives rise to a Hopf algebra morphism $\nu(p) : F_\infty O\to H'$, and by Prop. \ref{prop:LA:filt}(c), 
this morphism is 
compatible with the filtrations $F_\bullet$ on both sides. For any $n\geq 0$, the restriction of $p$ to 
$F_nO\otimes H$ induces a pairing $F_nO\otimes (H/F^{n+1}H)\to\mathbb C$, which gives rise to a linear map
$F_nO\to(H/F^{n+1}H)^*=F_nH^*$; composing this linear map with the identification from Lem. \ref{IDFILT} gives rise to 
a linear map $F_nO\to F_nH'$, which is equal to $F_n\nu(p)$. The morphism $\nu(p)$ is compatible 
the augmentation maps $\epsilon_O,\epsilon_{H'}$, and therefore gives rise to a linear map  
$F_nO\cap O_+\to F_nH'\cap (H')_+$, which for $n=1$
coincides with $\mathrm{gr}_1(\nu(p))$. It follows that $\mathrm{gr}_1(\nu(p))$ may be constructed as follows: 
the restriction of $p$ to $(F_1 O\cap J_O)\otimes F^1H$ induces a pairing 
$(F_1 O\cap O_+)\otimes (F^1H/F^2H)\to\mathbb C$; then $\mathrm{gr}_1(\nu(p))$
is the induced map $\mathrm{gr}_1(O)=F_1 O\cap O_+\to (F^1H/F^2H)^*
=\mathrm{gr}_1(H')$. 

It follows that $\mathrm{gr}_1(\nu(p_{J,x_0})) : \mathrm{H}_C
=\mathrm{gr}_1(\mathrm{Sh}(\mathrm{H}_C))\to \mathrm{gr}_1((\mathbb C\Gamma_C)')=(\Gamma_C^{ab}
\otimes\mathbb C)^*$ is induced by the restriction of $p_{J,x_0}$ to $\mathrm{H}_C\otimes 
(\mathbb C\Gamma_C)_+$. 
This restriction coincides with the lift of \eqref{basic:pairing} by Lem. \ref{GR:ULJ}(b), which implies the statement.
%{\color{red}by the properties of $J$}.  
%since $\sigma$ is a lift of the projection $\Omega(C)\to \mathrm{H}_C$.  

(c) The pairing \eqref{basic:pairing} coincides with the period pairing for $C$, which is nondegenerate. It follows that the map 
$\mathrm H_C\to (\mathrm H_C^{\mathrm{B}})^*$ induced by \eqref{basic:pairing} is an isomorphism. The statement then 
follows from (b). 
\end{proof}

\subsubsection{Proof that $\nu(p_{J,x_0}) : \mathrm{Sh}(\mathrm H_C)\to(\mathbb C\Gamma_C)'$ is a Hopf algebra isomorphism}

\begin{prop} \label{prop:1508} %(will be used in order to apply Lem 4.15)
$\nu(p_{J,x_0}) : \mathrm{Sh}(\mathrm H_C)\to(\mathbb C\Gamma_C)'$ is a Hopf algebra isomorphism. 
\end{prop}

\begin{proof}
By Lem. \ref{lem:basic:pairing}(c) and Lem. \ref{lem:514:0108}, $q_{J,x_0}$ is a Hopf algebra isomorphism. 
By Def. \ref{def:5:13:0308}, $\nu(p_{J,x_0})$ is obtained from $q_{J,x_0}$ by pre- and post-composition with Hopf algebra 
isomorphisms, which implies that it is a Hopf algebra isomorphism.
\end{proof}

\begin{rem}\label{rem:4:12:0712}
Prop. \ref{prop:1508} may be related to Chen's $\pi_1$ theorem as follows. Let $(\mathrm H^\bullet_{\mathrm{dR}}(C),
\varepsilon_{triv})$ be the augmented dga with $\mathrm H^\bullet_{\mathrm{dR}}(C):=\mathbb C\oplus \mathrm{H}_C$ with zero 
differential and $\varepsilon_{triv} :\mathrm H^\bullet_{\mathrm{dR}}(C)\to\mathbb C$ be the projection in degree 0, and let 
$(\mathcal E^\bullet(C),\varepsilon_{x_0})$ be the augmented dga of smooth differential forms on $C$, with $\varepsilon_{x_0}$
given by evaluation at $x_0$. Then the dga morphism $\mathrm H^\bullet_{\mathrm{dR}}(C)\to\mathcal E^\bullet(C)$ induced by $J$ is 
compatible with the augmentations, therefore it induces an isomorphism of commutative Hopf algebras  
$H^0(B(\mathrm H^\bullet_{\mathrm{dR}}(C),\varepsilon_{triv}))\to H^0(B(\mathcal E^\bullet(C),\varepsilon_{x_0}))$, 
where $H^0(B(-))$ is the zero-th cohomology of the bar-construction of an augmented dga. 
One easily constructs a Hopf algebras isomorphism 
$H^0(B(\mathrm H^\bullet_{\mathrm{dR}}(C),\varepsilon_{triv}))\simeq \mathrm{Sh}(\mathrm H_C)$. 
The combination of these isomorphisms with the Hopf algebra isomorphism 
$H^0(B(\mathcal E^\bullet(C),\varepsilon_{x_0})) \simeq (\mathbb C\Gamma_C)'$ from Chen's ``$\pi_1$ de Rham theorem'' (\cite{BGF}, Thm. 3.264) is the Hopf algebra isomorphism from Prop. \ref{prop:1508}.   
\end{rem}

\subsection{The isomorphism of filtered algebras $f_{J,x_0} : 
F_\bullet\mathrm{Sh}(\mathrm{H^{dR}_C})\otimes\mathcal O(C) \to F_\bullet\mathcal O_{mod}(\tilde C)$}\label{sect:36:2811}

\begin{lemdef}\label{lem:3:13:0409}
There is a unique linear map $f_{J,x_0} :  \mathrm{Sh}(\mathrm{H}_C)\otimes \mathcal O(C)\to
\mathcal O_{mod}(\tilde C)$ such that $f_{J,x_0}(a\otimes f):=I_{x_0}\circ J_*(a)\cdot p^*f$; it is an algebra 
morphism. 
\end{lemdef}

\begin{proof}
The fact that $f_{J,x_0}$ is well-defined as a linear map follows from Prop. \ref{lem:INOMOD:1408} and from the inclusion 
$\mathcal O(C) \subset\mathcal O_{mod}(\tilde C)$, which follows from Prop. \ref{lem:2:12:1508}.  
Follows from the decomposition of $f_{J,x_0}$ as 
$\mathrm{Sh}(\mathrm{H}_C)\otimes \mathcal O(C)\stackrel{(I_{x_0}\circ J_*)\otimes p^* }{\to}
\mathcal O_{mod}(\tilde C)^{\otimes 2}\stackrel{m_{
\mathcal O_{mod}(\tilde C)}}{\to}
\mathcal O_{mod}(\tilde C)$, where $m_{\mathcal O_{mod}(\tilde C)}$ is the product map of $\mathcal O_{mod}(\tilde C)$,
and from the algebra morphism status of $I_{x_0},p^*,J_*$ and $m_{\mathcal O_{mod}(\tilde C)}$ (the latter coming from the
commutativity of $\mathcal O_{mod}(\tilde C)$). 
\end{proof}

In App. \ref{sect:B:0512}, we introduce the notions of Hopf algebras with comodule algebra (HACA) and Hopf algebra with module algebra
(HAMA).  Then, by Prop. \ref{lem:B:15:1508}(a), the algebra $\mathcal O(C)$ and the Hopf algebra $\mathrm{Sh}(\mathrm{H}_C)$ 
give rise to a HACA 
structure $(\mathrm{Sh}(\mathrm{H}_C),\mathrm{Sh}(\mathrm{H}_C)\otimes\mathcal O(C))$; on the other hand,
a HAMA structure is constructed as follows.  

\begin{lem}\label{LEM:HAMA:0311}
The pair $(\mathcal O_{mod}(\tilde C),\mathbb C\Gamma_C)$ is equipped with a HAMA structure.  
\end{lem}

\begin{proof}
The HAMA structure is induced by the right $\Gamma_C$-action on $\mathcal O_{mod}(\tilde C)$ (see Prop. \ref{lem:1:7:toto} and Def. \ref{defn:B6}). 
\end{proof}

In App. \ref{sect:B:0512} (see Def. \ref{def:5:9:toto}), we also introduce the notion of pairing-morphism from a HACA $(O,A)$ 
to a HAMA $(B,H)$, and denote by $\mathbf{PM}((O,A),(B,H))$ the set of such structures. 
 
Recall the Hopf algebra pairing $p_{J,x_0} : \mathrm{Sh}(\mathrm{H}_C) \otimes \mathbb C\Gamma_C\to\mathbb C$ 
(see Lem. \ref{HA:pairing}). 

 \begin{lem}\label{lemA:1508}
$(p_{J,x_0},f_{J,x_0}) \in \mathbf{PM}((\mathrm{Sh}(\mathrm{H}_C),
\mathrm{Sh}(\mathrm{H}_C)\otimes\mathcal O(C)),
(\mathcal O_{mod}(\tilde C),\mathbb C\Gamma_C))$. 
\end{lem}

\begin{proof}
By Def. \ref{def:5:9:toto}, the identity to check is 
$$ 
(I_{x_0}(%\underline 
J_*(a))p^*(f))_{|\gamma}=I_{x_0}(%\underline 
J_*(a^{(2)}))p^*(f)p_{J,x_0}(\gamma\otimes a^{(1)})
$$
for any $a\in\mathrm{Sh}(\mathrm H_C)$, $f \in \mathcal O(C)$ and $\gamma \in \Gamma_C$, using the notation 
$\Delta_{\mathrm{Sh}(\mathrm H_C)}(a)=a^{(1)}\otimes a^{(2)}$. 
This follows from the invariance de $p^*f$ under the action of $\Gamma_C$ and from the identity 
$I_{x_0}(%\underline 
J_*(a))_{|\gamma}=I_{x_0}(%\underline 
J_*(a^{(2)}))p_{J,x_0}(\gamma\otimes a^{(1)})$ which is proved as follows: 
for any  $x\in\tilde C$, one has 
$$
I_{x_0}(%\underline 
J_*(a))_{|\gamma}(x)
=I_{x_0}(%\underline 
J_*(a))(\gamma x)
=I_{x_0}(%\underline 
J_*(a^{(1)}))(\gamma x_0)I_{\gamma x_0}(%\underline 
J_*(a^{(2)}))(\gamma x) 
=p_{J,x_0}(\gamma\otimes a^{(1)})I_{x_0}(%\underline 
J_*(a^{(2)}))(x). 
$$
where the second identity follows from Lem. \ref{lem:chain:rule:2308} and the third identity follows from the definition of 
$p_{J,x_0}$ and Lem. \ref{lem:invce:2308}. 
\end{proof}

\begin{lem}\label{LEM:HACA}
    The HAMA structure $(\mathcal O_{mod}(\tilde C),\mathbb C\Gamma_C)$ gives rise to an algebra 
    filtration $F_\bullet\mathcal O_{mod}(\tilde C)$ of $\mathcal O_{mod}(\tilde C)$, which fits in a HACA 
    $((\mathbb C\Gamma_C)',F_\infty\mathcal O_{mod}(\tilde C))$. 
\end{lem}

\begin{proof}
 The construction of the said HACA from this HAMA 
follows from Lem. \ref{1008:jeu} and the fact that $\mathbb C\Gamma_C$ is an 
object in $\mathbf{HA}_{fd}$ (see proof of Lem. \ref{comp:CGamma:prime}). 
The filtration $F_\bullet\mathcal O_{mod}(\tilde C)$ is then as in the Introduction 
(see Def. \ref{def:f:infty:o:mod:0311}).
\end{proof}

\begin{prop}\label{prop:3:15:0409}
(a) $(\nu(p_{J,x_0}),f_{J,x_0})$ induces an isomorphism of HACAs 
$(\mathrm{Sh}(\mathrm{H}_C),\mathrm{Sh}(\mathrm{H}_C)\otimes \mathcal O(C))\to 
((\mathbb C\Gamma_C)',F_\infty\mathcal O_{mod}(\tilde C))$. 

(b) $f_{J,x_0}$
induces an isomorphism of algebra filtrations $F_\bullet \mathrm{Sh}(\mathrm{H}_C)\otimes\mathcal O(C)  \to 
F_\bullet\mathcal O_{mod}(\tilde C)$. 
\end{prop}

\begin{proof}
Let $\mathbf a:=\mathcal O(C)$, $O:=\mathrm{Sh}(\mathrm{H}_C)$, $B:=\mathcal O_{mod}(\tilde C)$, $H:=\mathbb C\Gamma_C$, 
$(B,H)$ be the HAMA structure induced by the action of $\Gamma_C$ on $\tilde C$; it is an object in $\mathbf{HAMA}_{fd}$
(see Def. \ref{def:HAMA:fd}) since~$\Gamma_C$ is finitely generated.  Set $p:=p_{J,x_0} \in \mathbf{Pair}(O,H)$ (see Def. \ref{defn:A9}), 
$f:=f_{J,x_0}$. By Lem.~\ref{lemA:1508}, $(p,f) \in \mathbf{PM}((O,O\otimes a),(B,H))$. By Prop. \ref{prop:1508}, 
$\nu(p_{J,x_0}) : \mathrm{Sh}(\mathrm{H}_C)\to(\mathbb C\Gamma_C)'$ is a Hopf algebra isomorphism
and by Prop. \ref{lem:2:12:1508}, $f_{J,x_0}$ induces an algebra isomorphism $\mathbb C\otimes\mathcal O(C)  \to
\mathcal O_{mod}(\tilde C)^{\mathbb C\Gamma_C}$. The assumptions of Prop. \ref{lem:B:15:1508}(b) are therefore satisfied; 
the result is then a consequence of this statement. 
\end{proof}

\subsection{Filtered formality of the HACA $((\mathbb C\Gamma_C)',F_\infty\mathcal O_{mod}(\tilde C))$}\label{sect:37:2811}

In \S\ref{B'2:3110}, we introduce the definition of a filtered formal HACA. 
\begin{prop}\label{3:17:0711}
The pair  $((\mathbb C\Gamma_C)',F_\infty\mathcal O_{mod}(\tilde C))$ is filtered formal. 
\end{prop}

\begin{proof} 
It follows from the proof of Lem. \ref{prop:3:15:0409} that the data 
\begin{equation}\label{data}
    (O,\mathbf a):=(\mathrm{Sh}(\mathrm H_C),\mathcal O(C)),\quad (B,H)=(\mathcal O_{mod}(\tilde C),\mathbb C\Gamma_C),
    \quad (p,f)=(p_{J,x_0},f_{J,x_0}). 
\end{equation}
satisfy the hypotheses of Prop. \ref{lem:B:15:1508}(b),(c).  The statement is then a consequence of Prop.~\ref{prop:10:14:3110}.
\end{proof}

\section{Filtrations on $\mathcal O_{hol}(\tilde C)$, and the minimal stable subalgebra $A_C$}\label{sect:5:1412}

In \S\ref{sect:51:2811}, we study the filtration of $\mathcal O_{hol}(\tilde C)$ given by the image by $I_{x_0}$ of the filtration 
$F_\bullet(\mathrm{Sh}(\Omega(C))$ of $\mathrm{Sh}(\Omega(C))$, and identify it with the image by 
$f_{J,x_0}$ of the filtration $F_\bullet \mathrm{Sh}(\mathrm H_C)\otimes F_\bullet^{unit}\mathcal O(C)$ of 
$\mathrm{Sh}(\mathrm H_C)\otimes\mathcal O(C)$ (Prop. \ref{(a2)}). In §\ref{sect:52:2811}, we introduce and 
study the filtrations $F_\bullet^\delta\mathcal O_{hol}(\tilde C)$ and $F_\bullet^\mu\mathcal O_{hol}(\tilde C)$ of 
$\mathcal O_{hol}(\tilde C)$ inspired by \cite{Chen} and identify the latter with the image by $f_{J,x_0}$ of the filtration 
$F_\bullet \mathrm{Sh}(\mathrm H_C)\otimes F_\bullet^{triv}\mathcal O(C)$ of 
$\mathrm{Sh}(\mathrm H_C)\otimes\mathcal O(C)$ (Prop. \ref{(a3)}). 
We study the relation of the total space of these filtrations with the MSSA $A_C$ of $\mathcal O_{hol}(\tilde C)$ in 
§\ref{sect:53:2811}. In \S\ref{sect:proof:of:thms}, we prove Theorems A, B and C, and in §\ref{sect:5:5:1611}, we draw consequences 
of Theorem A on the algebras $\mathcal H_C(J)$ constructed in §\ref{sect:91:0711}, namely we show that each such algebra is a free 
$\mathcal O(C)$-module with an explicit basis. In §\ref{sect:55:2811}, we discuss the relation of this material with the study 
in \cite{Chen} of Picard-Vessiot extensions of the function algebra of a smooth manifold.

\subsection{An algebra filtration of $\mathcal O_{hol}(\tilde C)$ defined by $I_{x_0}$}\label{sect:51:2811}

In the present \S\ref{sect:51:2811}, a point $x_0\in\tilde C$ is fixed. 
Recall the algebra morphism $I_{x_0} : \mathrm{Sh}(\Omega(C))\to\mathcal O_{hol}(\tilde C)$ (Lem.-Def. \ref{def:2:20:0908} and 
Lem. \ref{lem:221:3108}) and the algebra filtration $F_\bullet\mathrm{Sh}(V)$ for an arbitrary vector space $V$
(see \S\ref{subsect:IIs:1008}). By \S\ref{background:filtrations}, these data give rise to an algebra filtration $I_{x_0}(F_\bullet\mathrm{Sh}(\Omega(C)))$ of $\mathcal O_{hol}(\tilde C)$, which we study in Prop. \ref{(a2)}.

\begin{lem}\label{lem:1607:0306}
One has 
\begin{equation}\label{eq:1701:0306}
\forall \,p,q>0,\quad I_{x_0}([\mathrm{Sh}_p(\Omega(C))|d\mathcal O(C)|\mathrm{Sh}_q(\Omega(C))])
\subset I_{x_0}(\mathrm{Sh}_{p+q}(\Omega(C))),  
\end{equation}
\begin{equation}\label{eq:1705:0306}
\forall\, n>0,\quad I_{x_0}([d\mathcal O(C)|\mathrm{Sh}_n(\Omega(C))])\subset I_{x_0}(\mathrm{Sh}_n(\Omega(C))),
\end{equation}
\begin{equation}\label{eq:1709:0306}
\forall\, n>0,\quad 
I_{x_0}([\mathrm{Sh}_n(\Omega(C))|d\mathcal O(C)])\subset p^*\mathcal O(C)\cdot I_{x_0}(\mathrm{Sh}_n(\Omega(C))).  
\end{equation}
\end{lem}

\begin{proof} Let us prove \eqref{eq:1701:0306}. The space $[\mathrm{Sh}_p(\Omega(C))|d\mathcal O(C)|\mathrm{Sh}_q(\Omega(C))]$
is linearly spanned by the elements $[\alpha_1|\ldots,|\alpha_p|df|\beta_1|\ldots|\beta_q]$, where 
$\alpha_i,\beta_j\in \Omega(C)$ for any $i,j$ and $f\in\mathcal O(C)$. Then 
$I_{x_0}([\alpha_1|\ldots,|\alpha_p|df|\beta_1|\ldots|\beta_q])=I_{x_0}([\alpha_1|\ldots,|\alpha_p|f\cdot\beta_1|\ldots|\beta_q]
-[\alpha_1|\ldots,|\alpha_p\cdot f|\beta_1|\ldots|\beta_q])\in I_{x_0}(\mathrm{Sh}_{p+q}(\Omega(C)))$. 

Let us prove \eqref{eq:1705:0306}. The space $[d\mathcal O(C)|\mathrm{Sh}_n(\Omega(C))]$ is linearly spanned by the elements 
$[df|\alpha_1|\ldots|\alpha_n]$, where $f\in\mathcal O(C)$ and $\alpha_i\in\Omega(C)$ for any $i$. Then 
$I_{x_0}([df|\alpha_1|\ldots|\alpha_n])=I_{x_0}([f\cdot\alpha_1|\ldots|\alpha_n])-f(x_0)\cdot I_{x_0}([\alpha_1|\ldots|\alpha_n])
\in I_{x_0}(\mathrm{Sh}_n(\Omega(C)))$. Eq. \eqref{eq:1709:0306} similarly follows from 
$I_{x_0}([\alpha_1|\ldots|\alpha_n|df])=p^*f\cdot I_{x_0}([\alpha_1|\ldots|\alpha_n])-I_{x_0}([\alpha_1|\ldots|\alpha_n\cdot f])
\in p^*\mathcal O(C)\cdot I_{x_0}(\mathrm{Sh}_n(\Omega(C)))$. 
\end{proof}

\begin{lem}\label{(a2)old}
For any $\sigma\in\Sigma_C$ and any $n\geq 0$, one has the inclusion 
$$
I_{x_0}(F_n\mathrm{Sh}(\Omega(C)))
\subset f_{J_\sigma,x_0}(F_n\mathrm{Sh}(\mathrm H_C)\otimes \mathbb C1+F_{n-1}\mathrm{Sh}(\mathrm H_C)\otimes\mathcal O(C)). 
$$
\end{lem}

\begin{proof}
By induction on $n\geq 0$. 
For $n=0$, the equality is obvious as both sides are equal to~$\mathbb C$. Let $n>0$, assume 
the equality for all steps $\leq n-1$ and let us prove it at step $n$. By the induction hypothesis, it suffices to prove the 
inclusion 
$I_{x_0}(\mathrm{Sh}_n(\Omega(C)))
\subset f_{J_\sigma,x_0}(F_n\mathrm{Sh}(\mathrm H_C)\otimes\mathbb C1
+F_{n-1}\mathrm{Sh}(\mathrm H_C)\otimes\mathcal O(C))$, i.e. 
\begin{equation}\label{opp:inclusion}
I_{x_0}(\mathrm{Sh}_n(\Omega(C)))
\subset I_{x_0}(F_n\mathrm{Sh}(\sigma(\mathrm H_C)))+p^*\mathcal O(C)\cdot I_{x_0}(F_{n-1}\mathrm{Sh}(\mathrm H_C)). 
\end{equation}
The space $\mathrm{Sh}_n(\Omega(C))$ is linearly spanned by the elements $[\omega_1|\ldots|\omega_n]$ with 
$\omega_1,\ldots,\omega_n\in\Omega(C)$. For all $i$, let $h_i\in \mathrm H_C$ be the projection of $\omega_i$ and choose 
$f_i\in \mathcal O(C)$ such that $\omega_i=\sigma(h_i)+df_i$. Then 
\begin{align*}
&[\omega_1|\ldots|\omega_n]
\in [\sigma(h_1)|\ldots|\sigma(h_n)]+[df_1|\mathrm{Sh}_{n-1}(\Omega(C))]
+[\mathrm{Sh}_{n-1}(\Omega(C))|df_{n-1}]\\ & +\sum_{i=1}^{n-1}[\mathrm{Sh}_{i-1}(\Omega(C))|df_i|\mathrm{Sh}_{n-i}(\Omega(C))]
\subset \mathrm{Sh}_n(\sigma(\mathrm H_C))
+[d\mathcal O(C)|\mathrm{Sh}_{n-1}(\Omega(C))]
\\ &+[\mathrm{Sh}_{n-1}(\Omega(C))|d\mathcal O(C)] +\sum_{i=1}^{n-1}[\mathrm{Sh}_{i-1}(\Omega(C))|d\mathcal O(C)|\mathrm{Sh}_{n-i}(\Omega(C))].
\end{align*}
Lem.  \ref{lem:1607:0306} then implies that
\begin{align}\label{1627:0306}
\nonumber I_{x_0}([\omega_1|\ldots|\omega_n])&\in I_{x_0}(\mathrm{Sh}_n(\sigma(\mathrm H_C)))
+p^*\mathcal O(C)\cdot I_{x_0}(\mathrm{Sh}_{n-1}(\Omega(C)))
+p^*\mathcal O(C)\cdot I_{x_0}(\mathrm{Sh}_{n-1}(\Omega(C)))
\\ & +\sum_{i=1}^{n-1}I_{x_0}(\mathrm{Sh}_{n-1}(\Omega(C)))
 = I_{x_0}(\mathrm{Sh}_n(\sigma(\mathrm H_C)))+p^*\mathcal O(C)\cdot I_{x_0}(\mathrm{Sh}_{n-1}(\Omega(C))).    
\end{align}
Moroever,
$$
p^*\mathcal O(C)\cdot I_{x_0}(\mathrm{Sh}_{n-1}(\Omega(C)))\subset
p^*\mathcal O(C)\cdot I_{x_0}(F_{n-1}\mathrm{Sh}(\Omega(C)))\subset 
p^*\mathcal O(C)\cdot I_{x_0}(F_{n-1}\mathrm{Sh}(\sigma(\mathrm H_C))),
$$
where the first inclusion follows from $\mathrm{Sh}_{n-1}(\Omega(C))\subset F_{n-1}\mathrm{Sh}(\Omega(C))$
and the second inclusion from the induction hypothesis, i.e. \eqref{opp:inclusion} at step $n-1$ by multiplication
by $p^*\mathcal O(C)$. Combining this inclusion with \eqref{1627:0306}, one obtains 
$$
I_{x_0}([\omega_1|\ldots|\omega_n])\in I_{x_0}(\mathrm{Sh}_n(\sigma(\mathrm H_C)))
+p^*\mathcal O(C)\cdot I_{x_0}(F_{n-1}\mathrm{Sh}(\sigma(\mathrm H_C))), 
$$
which is \eqref{opp:inclusion} at step $n$. 
\end{proof}

\begin{prop}\label{(a2)}
For any $J\in\mathrm{MC}_{nd}(C)$, one has the equality 
$$
I_{x_0}(F_\bullet\mathrm{Sh}(\Omega(C)))=f_{J,x_0}(F_\bullet\mathrm{Sh}(\mathrm H_C)\otimes\mathbb C1
+F_{\bullet-1}\mathrm{Sh}(\mathrm H_C)\otimes\mathcal O(C))
$$
of filtrations of $\mathcal O_{hol}(\tilde C)$, where $f_{J,x_0}$ is the algebra morphism from Lem.-Def. \ref{lem:3:13:0409}. 
%Thm. \ref{thm:03:0711}. 
\end{prop}

\begin{proof}
For $n\geq 0$, the space $f_{J,x_0}(F_n\mathrm{Sh}(\mathrm H_C)\otimes\mathbb C1 
+F_{n-1}\mathrm{Sh}(\mathrm H_C)\otimes\mathcal O(C))$
is equal to 
$I_{x_0}(J_*(F_n\mathrm{Sh}(\mathrm H_C) ))
+p^*(\mathcal O(C))\cdot I_{x_0}(J_*(F_{n-1}\mathrm{Sh}(\mathrm H_C) ))$, which is also equal to 
$I_{x_0}(J_*(F_n\mathrm{Sh}(\mathrm H_C) )
+d\mathcal O(C)\sha J_*(F_{n-1}\mathrm{Sh}(\mathrm H_C) ))$ since 
$p^*(\mathcal O(C))=\mathbb C+I_{x_0}([d\mathcal O(C)])$.    
This implies the equality 
\begin{equation}\label{toto:1608}
f_{J,x_0}(F_n\mathrm{Sh}(\mathrm H_C)\otimes\mathbb C1
+F_{n-1}\mathrm{Sh}(\mathrm H_C)\otimes\mathcal O(C))
=I_{x_0}(J_*(F_n\mathrm{Sh}(\mathrm H_C) )
+d\mathcal O(C)\sha J_*(F_{n-1}\mathrm{Sh}(\mathrm H_C) )). 
\end{equation}

Let us now prove 
\begin{equation}\label{eqBIS:1608}
\forall \,n\geq 0,\quad I_{x_0}(F_n\mathrm{Sh}(\Omega(C)))
=I_{x_0}(J_*(F_n\mathrm{Sh}(\mathrm H_C) )
+d\mathcal O(C)\sha J_*(F_{n-1}\mathrm{Sh}(\mathrm H_C) )),  
\end{equation}

The argument of $I_{x_0}$ in the right-hand side of \eqref{eqBIS:1608}
is contained in $F_n\mathrm{Sh}(\Omega(C))$, which implies the inclusion (left-hand side of \eqref{eqBIS:1608}) $\supset$ 
(right-hand side of \eqref{eqBIS:1608}). 

We now prove the opposite inclusion. There is a sequence of inclusions (in $\mathrm{Sh}(\Omega(C))$)  
\begin{align*}
&(\sigma_J)_*(F_n\mathrm{Sh}(\mathrm H_C))+d\mathcal O(C)\sha(\sigma_J)_*(F_{n-1}\mathrm{Sh}(\mathrm H_C))
\\ & 
\subset J_*(F_n\mathrm{Sh}(\mathrm H_C))+F_{n-1}\mathrm{Sh}(\Omega(C))
+d\mathcal O(C)\sha J_*(F_{n-1}\mathrm{Sh}(\mathrm H_C))+d\mathcal O(C)\sha F_{n-2}\mathrm{Sh}(\Omega(C))
\\ & 
=J_*(F_n\mathrm{Sh}(\mathrm H_C))+d\mathcal O(C)\sha J_*(F_{n-1}\mathrm{Sh}(\mathrm H_C))+F_{n-1}\mathrm{Sh}(\Omega(C)), 
\end{align*}
$\sigma_J$ being as in Lem-Def. \ref{lem:def:sigma:J:1012}, where the first inclusion follows from Lem. \ref{GR:ULJ}(b) and 
the second inclusion follows from $d\mathcal O(C)\sha F_{n-2}\mathrm{Sh}(\Omega(C)) \subset F_{n-1}\mathrm{Sh}(\Omega(C))$. 
One has therefore 
\begin{align}\label{titi:mrs:1608}
&\forall n\geq 0,\quad (\sigma_J)_*(F_n\mathrm{Sh}(\mathrm H_C))+d\mathcal O(C)\sha(\sigma_J)_*(F_{n-1}\mathrm{Sh}(\mathrm H_C))
\\ & \nonumber \subset J_*(F_n\mathrm{Sh}(\mathrm H_C))+d\mathcal O(C)\sha J_*(F_{n-1}\mathrm{Sh}(\mathrm H_C))
+F_{n-1}\mathrm{Sh}(\Omega(C)). 
\end{align}
For any $n\geq 0$, one then has 
\begin{align*}
&I_{x_0}(F_n(\mathrm{Sh}(\Omega(C)))) \subset f_{J_{\sigma_J},x_0}(F_n\mathrm{Sh}(\mathrm H_C)\otimes\mathbb C1+
F_{n-1}\mathrm{Sh}(\mathrm H_C)\otimes\mathcal O(C))
\\ & =I_{x_0}( (\sigma_J)_*(F_n\mathrm{Sh}(\mathrm H_C))+d\mathcal O(C)\sha(\sigma_J)_*(F_{n-1}\mathrm{Sh}(\mathrm H_C)) ) 
\\ & \subset I_{x_0}( J_*(F_n\mathrm{Sh}(\mathrm H_C))+d\mathcal O(C)\sha J_*(F_{n-1}\mathrm{Sh}(\mathrm H_C))+F_{n-1}\mathrm{Sh}(\Omega(C)) )
\\ & =I_{x_0}( J_*(F_n\mathrm{Sh}(\mathrm H_C))+d\mathcal O(C)\sha J_*(F_{n-1}\mathrm{Sh}(\mathrm H_C)) ) 
+I_{x_0}(F_{n-1}\mathrm{Sh}(\Omega(C)) )
\\ & =f_{J,x_0}(F_n\mathrm{Sh}(\mathrm H_C)\otimes\mathbb C1
+F_{n-1}\mathrm{Sh}(\mathrm H_C)\otimes\mathcal O(C))+I_{x_0}(F_{n-1}\mathrm{Sh}(\Omega(C)) )
\end{align*}
where the first relation follows from Lem. \ref{(a2)old}, the second relation follows from \eqref{toto:1608} applied to 
$J_{\sigma_J}$, the third relation follows from \eqref{titi:mrs:1608}, and the last relation follows from the second relation 
follows from \eqref{toto:1608} applied to $J$. The relation
$$
\forall n\geq 0,\quad I_{x_0}(F_n(\mathrm{Sh}(\Omega(C)))) \subset f_{J,x_0}(F_n\mathrm{Sh}(\mathrm H_C)\otimes\mathbb C1
+F_{n-1}\mathrm{Sh}(\mathrm H_C)\otimes\mathcal O(C))
$$ 
then follows by induction. Therefore
(left-hand side of \eqref{eqBIS:1608}) $\subset$ 
(right-hand side of \eqref{eqBIS:1608}), which ends the proof of \eqref{eqBIS:1608}. 

The result then follows from the combination of \eqref{eqBIS:1608} and \eqref{toto:1608}. 
\end{proof}

%\section{Algebra filtrations of $\mathcal O_{hol}(\tilde C)$ defined by the differential}

\subsection{The filtrations $F_\bullet^\delta\mathcal O_{hol}(\tilde C)$ and $F_\bullet^\mu\mathcal O_{hol}(\tilde C)$}
\label{sect:52:2811}\label{sect:5:2:1512}

In Def. \ref{def:filtr:0911}, we defined %\eqref{eq:1:1512}, \eqref{eq:2:1512}, \eqref{eq:3:1512}, we defined 
$F^\delta_\bullet\mathcal O_{hol}(\tilde C)$, $F^\mu_\bullet\mathcal O_{hol}(\tilde C)$,
$F^\delta_\infty\mathcal O_{hol}(\tilde C)$ and $F^\mu_\infty\mathcal O_{hol}(\tilde C)$ (see Thm. C). 

\begin{prop}\label{(a1)}
For any $x_0\in\tilde C$, one has the equality 
$$
F^\delta_\bullet\mathcal O_{hol}(\tilde C)=I_{x_0}(F_\bullet\mathrm{Sh}(\Omega(C)))
$$
of filtrations of $\mathcal O_{hol}(\tilde C)$. 
\end{prop}

\begin{proof} Let us prove 
\begin{equation}
\forall\, n\geq 0,\quad I_{x_0}(F_n\mathrm{Sh}(\Omega(C)))=F^\delta_n\mathcal O_{hol}(\tilde C)    
\end{equation}
by induction on $n$. For $n=0$, the equality holds since both sides are equal to $\mathbb C$. Assume the equality at step $n\geq 0$
and let us show it at step $n+1$. 

Let us first show the inclusion $I_{x_0}(F_{n+1}\mathrm{Sh}(\Omega(C)))\subset 
F^\delta_{n+1}\mathcal O_{hol}(\tilde C)$. For this, in view of the induction hypothesis, it suffices to 
prove $I_{x_0}(\mathrm{Sh}_{n+1}(\Omega(C)))\subset F^\delta_{n+1}\mathcal O_{hol}(\tilde C)$. The space is linearly 
spanned by the elements $[\omega_1|\ldots|\omega_{n+1}]$, where $\omega_1,\ldots,\omega_{n+1}\in\Omega(C)$. Then 
$d(I_{x_0}([\omega_1|\ldots|\omega_{n+1}]))=I_{x_0}([\omega_1|\ldots|\omega_{n}])\cdot \omega_{n+1}$, and 
$I_{x_0}([\omega_1|\ldots|\omega_{n}])\in I_{x_0}(\mathrm{Sh}_n(\Omega(C)))\subset F^\delta_n\mathcal O_{hol}(\tilde C)$
where the last inclusion follows from the induction hypothesis. This shows that $I_{x_0}([\omega_1|\ldots|\omega_{n+1}])
\in F^\delta_{n+1}\mathcal O_{hol}(\tilde C)$, therefore $I_{x_0}(\mathrm{Sh}_{n+1}(\Omega(C)))\subset F^\delta_{n+1}\mathcal O_{hol}(\tilde C)$
as wanted. 

Let us now show the inclusion $F^\delta_{n+1}\mathcal O_{hol}(\tilde C)\subset I_{x_0}(F_{n+1}\mathrm{Sh}(\Omega(C)))$. 
Let $f\in F^\delta_{n+1}\mathcal O_{hol}(\tilde C)$, then there exist elements $f_1,\ldots,f_k\in F^\delta_n\mathcal O_{hol}(\tilde C)$
and $\omega_1,\ldots,\omega_k\in \Omega(C)$ such that $df=\sum_i f_i\cdot p^*\omega_i$. By the induction hypothesis, there exist
$t_1,\ldots,t_k\in F_n\mathrm{Sh}(\Omega(C))$, such that $f_i=I_{x_0}(t_i)$ for any $i$. Then $df=\sum_i I_{x_0}(t_i)\cdot p^*\omega_i$. 
Integration gives $f=f(x_0)+\sum_i I_{x_0}([t_i|\omega_i])=I_{x_0}(f(x_0)+\sum_i[t_i|\omega_i])\in I_{x_0}(F_{n+1}(\mathrm{Sh}(\Omega(C))))$, 
which proves the claimed inclusion. 
\end{proof}

\begin{prop}\label{new:cor}
(a) Both $F^\delta_\bullet \mathcal O_{hol}(\tilde C)$ and $F^\mu_\bullet \mathcal O_{hol}(\tilde C)$ are algebra filtrations of 
$\mathcal O_{hol}(\tilde C)$. 

(b) For any $n\geq 0$, one has $F^\delta_n \mathcal O_{hol}(\tilde C) \subset F^\mu_n \mathcal O_{hol}(\tilde C) 
\subset F^\delta_{n+1}\mathcal O_{hol}(\tilde C)$.  

(c) One has $F^\delta_\infty \mathcal O_{hol}(\tilde C)=F^\mu_\infty \mathcal O_{hol}(\tilde C)$ (equality of subalgebras of 
$\mathcal O_{hol}(\tilde C)$).  
\end{prop}

\begin{proof}
Recall the shorthand $F_\bullet^{\delta/\mu}:=F_\bullet^{\delta/\mu}\mathcal O_{hol}(\tilde C)$ (see Lem. \ref{lem:intro}). 
By Prop. \ref{(a1)}, $F_\bullet^\delta$ is the image of the increasing algebra filtration 
$F_\bullet\mathrm{Sh}(\Omega(C))$ by the morphism $I_{x_0} : \mathrm{Sh}(\Omega(C))\to\mathcal O_{hol}(\tilde C)$,
which implies that $F_\bullet^\delta$ is an increasing algebra filtration of $\mathcal O_{hol}(\tilde C)$. 
Since $F_\bullet^\mu$ is obtained out of $F_\bullet^\mu$ by taking the product with the fixed subalgebra $\mathcal O(C)$ 
of $\mathcal O_{hol}(\tilde C)$, which implies that it is an increasing algebra filtration of $\mathcal O_{hol}(\tilde C)$ 
as well. This proves (a).  

For any $f\in \mathcal O(C)$, $df\in\Omega(C)=\Omega(C)\cdot F_0^\delta$, which implies $\mathcal O(C)\subset F^1_\delta$. 
For $n \geq 0$, one then has $F^\mu_n=\mathcal O(C)\cdot F^\delta_n \subset F^\delta_1\cdot F^\delta_n \subset F^\delta_{n+1}$. 
For $n\geq 0$, one clearly also has $F^\delta_n \subset F^\mu_n$, which implies (b). Statement (c) follows from (b). 
\end{proof}

\begin{prop}\label{(a3)}
For any $(J,x_0)\in\mathrm{MC}(C)\times\tilde C$, one has the equality 
$$
F^\mu_\bullet\mathcal O_{hol}(\tilde C)=f_{J,x_0}( F_\bullet\mathrm{Sh}(\mathrm H_C)\otimes\mathcal O(C))
$$
of filtrations of $\mathcal O_{hol}(\tilde C)$. 
\end{prop}

\begin{proof} Let $n\geq0$. In Prop. \ref{(a2)}, we proved the equality $I_{x_0}(F_n\mathrm{Sh}(\Omega(C)))=
I_{x_0}(J_*(F_n\mathrm{Sh}(\mathrm H_C)))+p^*\mathcal O(C)\cdot I_{x_0}(J_*(F_{n-1}\mathrm{Sh}(\mathrm H_C)))$. 
Multiplying it by $p^*\mathcal O(C)$, we obtain $p^*\mathcal O(C)\cdot I_{x_0}(F_n\mathrm{Sh}(\Omega(C)))=
p^*\mathcal O(C)\cdot I_{x_0}(J_*(F_n\mathrm{Sh}(\mathrm H_C)))$. By Prop. \ref{(a1)}, 
$I_{x_0}(F_n\mathrm{Sh}(\Omega(C)))=F_n^\delta\mathcal O_{hol}(\tilde C)$, and so $p^*\mathcal O(C)\cdot I_{x_0}(F_n\mathrm{Sh}(\Omega(C)))
=p^*\mathcal O(C)\cdot F_n^\delta\mathcal O_{hol}(\tilde C)=F_n^\mu\mathcal O_{hol}(\tilde C)$. It follows that 
$F_n^\mu\mathcal O_{hol}(\tilde C)=p^*\mathcal O(C)\cdot I_{x_0}(J_*(F_n\mathrm{Sh}(\mathrm H_C)))$; the right-hand side of this 
equality is equal to $f_{J,x_0}(F_n\mathrm{Sh}(\mathrm H_C)\otimes\mathcal O(C))$. 
\end{proof}

\subsection{Relation with $A_C$}\label{sect:53:2811}

Recall from §\ref{sect:context}:  

\begin{defn}
(a) A {\it stable subalgebra} (SSA) of $\mathcal O_{hol}(\tilde C)$ is a unital subalgebra $A$ of $\mathcal O_{hol}(\tilde C)$, such 
that for any $f\in A$ and $\omega\in \Omega(C)$, one has $\mathrm{int}_{x_0}(f\cdot p^*\omega)=(x\mapsto \int_{x_0}^x f\cdot p^*\omega)
\in A$. 

(b) $A_C:=\cap_{A\text{ a SSA of }\mathcal O_{hol}(\tilde C)}A$. 
\end{defn}

\begin{lemdef}\label{lem:5:10:2811} 
$A_C$ is a SSA of $\mathcal O_{hol}(\tilde C)$, contained in any SSA of $\mathcal O_{hol}(\tilde C)$; 
we therefore call $A_C$ the minimal stable subalgebra (MSSA) of $\mathcal O_{hol}(\tilde C)$. 
\end{lemdef}

\begin{proof} This follows from the fact that an intersection of two SSAs of $\mathcal O_{hol}(\tilde C)$ is a SSA of 
$\mathcal O_{hol}(\tilde C)$. \end{proof}

\begin{prop} \label{prop:AC:other}
For any $x_0\in\tilde C$, one has $A_C=I_{x_0}(\mathrm{Sh}(\Omega(C)))$. 
\end{prop}

\begin{proof}
For $\omega\in\Omega(C)$, let $\mathrm{prim}_\omega$ be the vector space endomorphism of $\mathcal O_{hol}(\tilde C)$ given by 
$f\mapsto (x\mapsto \int_{x_0}^x f\cdot p^*\omega)$. Let also $r_\omega$ be the linear endomorphism of $\mathrm{Sh}(\Omega(C))$
given by right concatenation with $\omega$; explicitly, $r_\omega([\omega_1|\ldots|\omega_k])=[\omega_1|\ldots|\omega_k|\omega]$
for any $\omega_1,\ldots,\omega_k\in\Omega(C)$. Then one checks the identity 
\begin{equation}\label{*:0102}
    \mathrm{prim}_\omega \circ I_{x_0}=I_{x_0} \circ R_\omega.
\end{equation}
It follows that $\mathrm{prim}_\omega(I_{x_0}(t))=I_{x_0}(R_\omega(t))$ for any $t\in\mathrm{Sh}(\Omega(C))$, 
which together with unitality implies that $I_{x_0}(\mathrm{Sh}(\Omega(C)))$ is a stable subalgebra of $\mathcal O_{hol}(\tilde C)$, hence $A_C\subset I_{x_0}(\mathrm{Sh}(\Omega(C)))$. 

Let now $A$ be a stable subalgebra of $\mathcal O_{hol}(\tilde C)$. Let $n\geq 0$. For any $\omega_1,\ldots,\omega_n\in\Omega(C)$, 
the element $\mathrm{prim}_{\omega_n}\circ \cdots\circ \mathrm{prim}_{\omega_1}(1)$ belongs to $A$ since $1\in A$ and by the stability 
of $A$. It follows from \eqref{*:0102} that this element is equal to $I_{x_0}([\omega_1|\ldots|\omega_n])$, and therefore $A$ contains 
$I_{x_0}(\mathrm{Sh}_n(\Omega(C)))$. This implies that A contains $I_{x_0}(\mathrm{Sh}(\Omega(C)))$, thus concluding the proof.
\end{proof}

\subsection{Proof of Theorems A,B,C}\label{sect:proof:of:thms}

\subsubsection{Proof of Theorem C}
In \eqref{thm:c:1st:line}, 
the first (resp. second) equation follows from Prop. \ref{prop:3:15:0409}(b) (resp. Prop. \ref{(a3)}). 
In \eqref{thm:c:2nd:line}, the first (resp. second) equation follows from Prop. \ref{(a1)} (resp. Prop. \ref{(a2)}). 
In \eqref{thm:c:3rd:line}, the first (resp. second, third, fourth, fifth) equality follows from 
Prop. \ref{prop:AC:other} (resp. Prop. \ref{(a1)} at infinity, Prop. \ref{new:cor}(c), Prop. \ref{(a3)} at infinity, 
Prop. \ref{prop:3:15:0409}(b) at infinity). 

\subsubsection{Proof of Theorem A}
Thm. A(a) is proved in Lem. \ref{lem:bibli:3005}. It follows from Prop. \ref{prop:3:15:0409} that the 
algebra morphism $f_{J,x_0}$ induces an algebra isomorphism $f_{J,x_0} : \mathrm{Sh}(\mathrm{H}_C)\otimes\mathcal O(C)  
\to F_\infty\mathcal O_{mod}(\tilde C)$. By Thm. C(b), one has $F_\infty\mathcal O_{mod}(\tilde C)=A_C$, 
which implies Thm. A(b). 

\subsubsection{Proof of Theorem B}

Thm. B follows from the combination of Prop. \ref{prop:3:15:0409}(a) and from the equality $F_\infty\mathcal O_{mod}(\tilde C)=A_C$, 
which follows from Thm. C(b).

\subsection{Consequences for hyperlogarithm functions}\label{sect:5:5:1611}

\begin{prop}\label{freeness:HL:funs}
Let $J\in\mathrm{MC}_{nd}(C)$, $x_0\in\tilde C$ and $(h_i)_{i\in[\![1,d]\!]}$ be a basis of $\mathrm H_C$. 

(a) The family $(\tilde f_{J,x_0}([h_{i_1}|\ldots|h_{i_k}]))_{\substack{k\geq 0,i_1,\ldots,i_k \in [\![1,d]\!]}}$ 
is a basis of the vector space $\mathcal H_C(J)$.

(b) The family in (a) is linearly independent over $\mathcal O(C)$, i.e. for any family $(\phi^{i_1,\ldots,i_k})_{k\geq 0,i_1,\ldots,i_k\in [\![1,d]\!]}$ in $\mathcal O(C)$, the 
relation 
$$
\sum_{k\geq 0,i_1,\ldots,i_k\in [\![1,d]\!]} p^*(\phi^{i_1,\ldots,i_k}) \tilde f_{J,x_0}([h_{i_1}|\ldots|h_{i_k}])=0
$$
%$\sum_{k\geq 0,i_1,\ldots,i_k\in [1,d]} p^*(f^{i_1,\ldots,i_k})L_{[\sigma(h_{i_1})|\ldots|\sigma(h_{i_k})]}=0$ 
implies the vanishing of $(\phi^{i_1,\ldots,i_k})_{k\geq 0,i_1,\ldots,i_k\in [\![1,d]\!]}$. 
\end{prop}

\begin{proof}
It follows from Prop. \ref{prop:3:15:0409}(a) that the algebra morphism $\tilde f_{J,x_0} : \mathrm{Sh}(\mathrm H_C)\to
\mathcal O_{hol}(\tilde C)$ is injective, which then implies (a). (b) follows from (a) and  Prop. \ref{prop:3:15:0409}(a). 
\end{proof}

Let now $C$ be as in §\ref{sect:92:0711}, so 
$C = \mathbb P^1_{\mathbb C}\smallsetminus S$, with $S$ a finite set containing $0$ and $\infty$. 
Let $\tilde f_{\sigma_0,0}: \mathrm{Sh}(\mathrm H_C)\to\mathcal O_{hol}(\tilde C)$ be as in Def. \ref{def:2011}. 

Define an algebra morphism $f_{\sigma_0,0} : \mathrm{Sh}(\mathrm H_C)\otimes\mathcal O(C)\to \mathcal O_{hol}(\tilde C)$
by $t\otimes f\mapsto p^*(f)\tilde f_{\sigma_0,0}(t)$. 

\begin{lem}\label{prop:mor:0:inj}
(a) The algebra morphism $f_{\sigma_0,0}$ is injective. 

(b) The map $\tilde f_{\sigma_0,0}$ is injective. 
\end{lem}

\begin{proof}
(a) Choose $z_0\in\tilde C$. Since 
$f_{\sigma_0,0}=m\circ (\tilde f_{\sigma_0,0}\otimes p^*)$ and $f_{\sigma_0,z_0}=
m\circ ((I_{z_0}\circ\sigma_0)\otimes p^*)$ (where $m : \mathcal O_{hol}(\tilde C)^{\otimes 2}\to \mathcal O_{hol}(\tilde C)$
is the product map), one has
\begin{equation}\label{equality:endos:mor}
f_{\sigma_0,0}=f_{\sigma_0,z_0}\circ(a_0^{z_0}\otimes \mathrm{id})    
\end{equation}
(equality of algebra morphisms $\mathrm{Sh}(\mathrm H_C)\otimes\mathcal O(C)\to\mathcal O_{hol}(\tilde C)$). 
The statement then follows from~\eqref{equality:endos:mor}, together with the injectivity of $f_{\sigma_0,z_0}$ 
(see Prop. \ref{prop:3:15:0409}(a)) and the automorphism status of $a_0^{z_0}$. 
Statement (b) follows from (a), as $\tilde f_{\sigma_0,0}$ is the composition of $f_{\sigma_0,0}$ with the canonical injection 
$-\otimes 1 : \mathrm{Sh}(\mathrm H_C)\to\mathrm{Sh}(\mathrm H_C)\otimes\mathcal O(C)$.\end{proof}

\begin{prop}[see also \cite{Br:these}, Cor. 5.6]
(a) The family of functions $(L_{[s_1|\ldots|s_k]})_{k\geq 0,s_1,\ldots,s_k\in S_\infty}$ of $\mathcal O_{hol}(\tilde C)$ 
(see notation after Def. \ref{def:2011}) is a basis of $\mathcal H_{\mathbb P^1_{\mathbb C}\smallsetminus S}(\sigma_0)$. 

(b) The  family of functions in (a) is linearly independent over $\mathcal O(C)$. 
\end{prop}

\begin{proof}
 Statement (a) follows from Lem. \ref{prop:mor:0:inj}(b) and from the second part 
of Prop.~\ref{prop:mor:0:moreover}. Statement (b) follows from (a) and from Lem. \ref{prop:mor:0:inj}. 
\end{proof}

\subsection{Relation with Chen's work}\label{sect:55:2811}
 
In the Introduction of \cite{Chen}, K.-T. Chen defines an algebra filtration $\tilde F_\bullet C^\infty(\tilde M)$ of 
$C^\infty(\tilde M)$
for any smooth manifold $M$, where $\tilde M\to M$ is the universal cover; it has the additional property that 
$\tilde F_n C^\infty(\tilde M)$ is a subalgebra of $C^\infty(\tilde M)$ for any $n\geq0$. If $X$ is a nonsingular complex algebraic variety, 
and $\tilde X$ is its universal cover, equipped with its natural structure of complex manifold, one can 
similarly define a filtration $\tilde F_\bullet \mathcal O_{hol}(\tilde X)$ of the algebra $\mathcal O_{hol}(\tilde X)$, 
by replacing in the definition of \cite{Chen} the spaces of smooth functions and 1-forms (denoted there $\Lambda^0(M)$ and 
$\Lambda^1(M)$) by the 
spaces of regular functions and differentials on $X$. When $X=C$, the explicit definition of $\tilde F_\bullet \mathcal O_{hol}(\tilde C)$
is as follows: $\tilde F_0\mathcal O_{hol}(\tilde C):=p^*\mathcal O(C)$ and 
$\tilde F_{r+1}\mathcal O_{hol}(\tilde C):=\mathbb C[f,\mathrm{int}_{x_0}(g\cdot p^*\omega)\,|\,f,g\in \tilde F_r\mathcal O_{hol}(\tilde C),\,
\omega\in\Omega(C)]$ for $r\geq 0$, where $\mathbb C[-]$ means the subalgebra generated by a family. 

\begin{lem} Let $x_0\in\tilde C$. 

(a) For any $r\geq 0$, $\tilde F_r\mathcal O_{hol}(\tilde C)\subset I_{x_0}(\mathrm{Sh}(\Omega(C)))$. 

(b) For any $r\geq 0$, $I_{x_0}(\mathrm{Sh}_r(\Omega(C)))\subset\tilde F_r\mathcal O_{hol}(\tilde C)$. 

(c) One has $\tilde F_\infty\mathcal O_{hol}(\tilde C)=I_{x_0}(\mathrm{Sh}(\Omega(C)))$ (equality of subalgebras of 
$\mathcal O_{hol}(\tilde C)$). 
\end{lem}

\begin{proof}
(a) $\tilde F_0\mathcal O_{hol}(\tilde C)=p^*\mathcal O(C)=I_{x_0}(\mathbb C\oplus[d\mathcal O(C)])
\subset I_{x_0}(\mathrm{Sh}(\Omega(C)))$. Let $r\geq 0$ and assume that $\tilde F_r\mathcal O_{hol}(\tilde C)\subset 
I_{x_0}(\mathrm{Sh}(\Omega(C)))$. Let $g\in \tilde F_r\mathcal O_{hol}(\tilde C),\,
\omega\in\Omega(C)$. One knows that for some $a\in \mathrm{Sh}(\Omega(C))$, $\omega=I_{x_0}(a)$.
Then $\mathrm{int}_{x_0}(g\cdot p^*\omega)=\mathrm{int}_{x_0}(I_{x_0}(a)\cdot p^*\omega)
=I_{x_0}([a|\omega])\in I_{x_0}(\mathrm{Sh}(\Omega(C)))$. Then 
$\tilde F_{r+1}\mathcal O_{hol}(\tilde C)=\mathbb C[f,\mathrm{int}_{x_0}(g\cdot p^*\omega)\,|\,f,g\in \tilde F_r\mathcal O_{hol}(\tilde C),\,
\omega\in\Omega(C)]\subset I_{x_0}(\mathrm{Sh}(\Omega(C)))$ by $\tilde F_r\mathcal O_{hol}(\tilde C)\subset 
I_{x_0}(\mathrm{Sh}(\Omega(C)))$ and the fact that $I_{x_0}(\mathrm{Sh}(\Omega(C)))$ is an algebra. 

(b) $I_{x_0}(\mathrm{Sh}_0(\Omega(C)))=\mathbb C\subset \tilde F_0\mathcal O_{hol}(\tilde C)$. Let $r\geq 0$ and assume that 
$I_{x_0}(\mathrm{Sh}_r(\Omega(C)))\subset\tilde F_r\mathcal O_{hol}(\tilde C)$. Let $a\in \mathrm{Sh}_{r+1}(\Omega(C))$. There exist
elements $(a_i)_{i=1,\ldots,k}$, $(\omega_i)_{i=1,\ldots,k}$ where $a_i\in\mathrm{Sh}_r(\Omega(C))$, $\omega_i\in\Omega(C)$ such that 
$a=\sum_{i=1}^k[a_i|\omega_i]$. Then 
$$
I_{x_0}(a)=\sum_{i=1}^k I_{x_0}([a_i|\omega_i])
=\sum_{i=1}^k\mathrm{int}_{x_0}(I_{x_0}(a_i)\cdot p^*\omega_i). 
$$
One has $I_{x_0}(a_i)\in \tilde F_r\mathcal O_{hol}(\tilde C)$
by assumption, therefore the final term of this equality belongs to $\tilde F_{r+1}\mathcal O_{hol}(\tilde C)$. Therefore 
$I_{x_0}(\mathrm{Sh}_{r+1}(\Omega(C)))\subset\tilde F_{r+1}\mathcal O_{hol}(\tilde C)$. The statement follows by induction. 

(c) By (a), we know that $\tilde F_\infty\mathcal O_{hol}(\tilde C)\subset I_{x_0}(\mathrm{Sh}(\Omega(C)))$. Moreover, (b) implies that, for any $r\geq 0$,
$I_{x_0}(\mathrm{Sh}_r(\Omega(C)))\subset\tilde F_\infty\mathcal O_{hol}(\tilde C)$, and so $I_{x_0}(\mathrm{Sh}(\Omega(C)))\subset\tilde F_\infty\mathcal O_{hol}(\tilde C)$.
\end{proof}
 
\begin{rem}
\S2.3 of \cite{Chen} contains the definition of another filtration $F_\bullet C^\infty(\tilde M)$. This definition is both 
an analogue of that of $F_\bullet^\mu\mathcal O_{hol}(\tilde C)$ (as both definitions give analogous values for the degree 0 term of 
the filtration) and  of $F_\bullet^\delta\mathcal O_{hol}(\tilde C)$ (as both definitions share the same induction step). However, the 
statement ``$F_r C^\infty(\tilde M)\cdot F_s C^\infty(\tilde M)\subset F_{r+s} C^\infty(\tilde M)$ for any $r,s\geq 0$'' from Prop. 2.3.1 
in \cite{Chen} appears to be wrong. Indeed, if $r=0$, $s=1$,  $F_0C^\infty(\tilde M)=p^*C^\infty(M)$ while $F_1C^\infty(\tilde M)=\{f\in 
C^\infty(\tilde M)\,|\,df\in F_0C^\infty(\tilde M)\cdot p^*\Lambda^1(M)\}$; since $F_0C^\infty(\tilde M)\cdot p^*\Lambda^1(M)=p^*\Lambda^1(M)$, 
the set $F_1C^\infty(\tilde M)$ is the set of functions on $\tilde M$ of the form $x\mapsto c+\int_{x_0}^xp^*\omega$, where 
$c\in\mathbb C$ and 
$ \omega\in\Lambda^1(M)$; this set in not stable under multiplication by $p^*C^\infty(M)$ (the mistake can be traced to the proof of 
Prop. 2.3.1, which overlooks the fact that the inclusion $dF_rC^\infty(\tilde M)\subset F_{r-1}C^\infty(\tilde M)\cdot p^*\Lambda^1(M)$
is valid in general only if one introduces $F_{-1}C^\infty(\tilde M)=\mathbb C$). 
\end{rem}

\part{Complementary results}

\section{Connections for HACAs}\label{sect:6:0212}

We introduce the notion of connection on a HACA in §\ref{sect:61:2911}. In §\ref{sect:nat:conn:2911}, we construct a natural connection 
on the HACA $((\mathbb C\Gamma_C)',F_\infty \mathcal O_{mod}(\tilde C))$. We compute its pull-back under the HACA isomorphism 
$(\nu(p_{J,x_0}),f_{J,x_0})$ from Prop. \ref{prop:3:15:0409} in §\ref{sect:preimage:conn:2911}. 

\subsection{Connections for HACAs}\label{sect:61:2911}

Let $\mathbf a$ be a commutative algebra. Recall that the $\mathbf a$-module of Kähler differentials of $\mathbf a$ is the quotient $\Omega_{\mathbf a}:=\mathrm{Ker}(m)/\mathrm{Ker}(m)^2$, where $m$ is the product map $\mathbf a \otimes\mathbf a\to\mathbf a$; the derivation 
$d : \mathbf a\to \Omega_{\mathbf a}$ is defined by $d(a)=a\otimes 1-1\otimes a+\mathrm{Ker}(m)^2$. One has 
$\Omega_{\mathcal O(C)}=\Omega(C)$. 

Let $(O,A)$ be a HACA with coaction morphism $\Delta_A : A\to O \otimes A$. 
One has $A^O=\{a \in A\,|\,\Delta_A(a)=1 \otimes a\}=F_0A$. If $A^O$ is a commutative algebra, then 
$A \otimes_{A^O} \Omega_{A^O}$ is a left $A$-module. 

\begin{defn}\label{def:81:0711}
Let $(O,A)$ be a HACA such that $A^O$ is central in $A$, so that $A \otimes_{A^O} \Omega_{A^O}$ 
is a right $A$-module. 
A connection for $(O,A)$ is a map $\nabla_A : A\to A \otimes_{A^O} \Omega_{A^O}$, which: \\
(a) is a derivation, i.e.  
$\nabla_A(aa')=a\nabla_A(a')+\nabla_A(a)a'$ for any $a,a' \in A$; \\
(b) is $O$-equivariant, i.e. 
$(\Delta_A \otimes id_{\Omega_{A^O}}) \circ \nabla_A=(id_O \otimes \nabla_A) \circ \Delta_A$; \\
(c) is such that 
$\nabla_A(a)=1 \otimes da$ for $a\in A^O$. 
\end{defn}
If $(O,A)$ is a HACA with connection $\nabla_A$, then one defines the pull-back of $\nabla_A$ by a HACA isomorphism 
$(O',A')\to(O,A)$, which is a connection for $(O',A')$. 

\begin{rem}
If $G$ is an algebraic group and $P$ is a principal $G$-bundle over an affine base, and $(O,A)$ is the pair 
of regular functions on these spaces, then a connection for $(O,A)$ is an algebraic version of a $G$-invariant 
Ehresmann connection on the bundle $P\to P/G$. 
\end{rem}

\subsection{A connection for $((\mathbb C\Gamma_C)',F_\infty\mathcal O_{mod}(\tilde C))$}\label{sect:nat:conn:2911}

\begin{prop}\label{lem:1:3:2905}
(a) The map $F_\infty\mathcal O_{mod}(\tilde C)\otimes_{\mathcal O(C)}\Omega(C)\to \Omega_{hol}(\tilde C)$
given by $f\otimes\omega\mapsto f\cdot p^*\omega$ is injective. 

(b) There exists a unique map $\nabla : F_\infty\mathcal O_{mod}(\tilde C)\to F_\infty\mathcal O_{mod}(\tilde C) 
\otimes_{\mathcal O(C)} \Omega(C)$ such that the diagram 
\begin{equation}\label{diag;1809}
\xymatrix{F_\infty\mathcal O_{mod}(\tilde C)\ar^{\!\!\!\!\!\!\!\!\nabla}[r]\ar@{^{(}->}[d] & F_\infty\mathcal O_{mod}(\tilde C)
\otimes_{\mathcal O(C)}\Omega(C)\ar@{^{(}->}[d]\\
\mathcal O_{hol}(\tilde C)\ar_{d}[r]& \Omega_{hol}(\tilde C)}
\end{equation}
commutes. 

(c) $\nabla$ is a connection for the HACA $((\mathbb C\Gamma_C)',F_\infty\mathcal O_{mod}(\tilde C))$
(see Lem. \ref{LEM:HACA}).  
\end{prop}

\begin{proof} (a) By Lem. \ref{prop:3:15:0409}, $f_{J_\sigma,x_0} : \mathrm{Sh}(\mathrm H_C)\otimes \mathcal O(C)\to 
F_\infty\mathcal O_{mod}(\tilde C)$ is an isomorphism of filtered $\mathcal O(C)$-modules. Its image by the functor 
$- \otimes_{\mathcal O(C)}\Omega(C)$ is an isomorphism of filtered vector spaces $\varphi_{\sigma,x_0} :
\mathrm{Sh}(\mathrm H_C)\otimes\Omega(C)\to F_\infty\mathcal O_{mod}(\tilde C)\otimes_{\mathcal O(C)} \Omega(C)$. 
The natural morphism $\mathrm{can} : F_\infty\mathcal O_{mod}(\tilde C)\otimes_{\mathcal O(C)}\Omega(C)\to \Omega_{hol}(\tilde C)$
is $\Gamma_C$-equivariant, therefore by Lem. \ref{MISSINGLEM}(b) is compatible with the filtrations induced by the action of 
$\Gamma_C$. The composed morphism  $\mathrm{can} \circ \varphi_{\sigma,x_0} : \mathrm{Sh}(\mathrm H_C)\otimes \Omega(C)\to 
\Omega_{hol}(\tilde C)$, given by $\varphi_{\sigma,x_0}(a\otimes \omega)=I_{x_0}(\sigma(a))\cdot p^*(\omega)$, is therefore compatible with the 
filtrations $F_\bullet\mathrm{Sh}(\mathrm H_C)\otimes \Omega(C)$ of the source and $F_\bullet\Omega_{hol}(\tilde C)$ of the 
target (see Def. \ref{MISSINGDEF}), therefore it gives rise to an associated graded map $\mathrm{gr}_\bullet (\mathrm{can}\circ\varphi_{\sigma,x_0}) : 
\mathrm{Sh}_\bullet(\mathrm H_C)\otimes \Omega(C)\to \mathrm{gr}_\bullet\Omega_{hol}(\tilde C)$. By Lem. \ref{INJLEM}(c), 
one attaches to the $\Gamma_C$-module $\Omega_{hol}(\tilde C)$ the injective graded map $\mathrm{gr}_\bullet\Omega_{hol}(\tilde C)
\hookrightarrow \oplus_{n\geq 0}\mathrm{Hom}_{\mathbb C-vec}(\mathrm{gr}^n(\mathbb C\Gamma_C),\Omega_{hol}(\tilde C)^{\Gamma_C})$, 
where $\Omega_{hol}(\tilde C)^{\Gamma_C}$ can be identified with the space of holomorphic differentials on $C$; it contains $\Omega(C)$ as a 
subspace.  

For any $n\geq0$, the composition of $\mathrm{gr}_n\Omega_{hol}(\tilde C)\hookrightarrow 
\mathrm{Hom}_{\mathbb C-vec}(\mathrm{gr}^n(\mathbb C\Gamma_C),\Omega_{hol}(\tilde C)^{\Gamma_C})$ 
with $\mathrm{gr}_n (\mathrm{can}\circ\varphi_{\sigma,x_0})$ is given by the composition of the isomorphism 
$\mathrm{Hom}_{\mathbb C-vec}(\mathrm{gr}^n(\mathbb C\Gamma_C),\mathbb C)\otimes 
\Omega_{hol}(\tilde C)^{\Gamma_C}\simeq \mathrm{Hom}_{\mathbb C-vec}(\mathrm{gr}^n(\mathbb C\Gamma_C),\Omega_{hol}(\tilde C)^{\Gamma_C})$
(due to the finite dimensionality of $\mathrm{gr}^n(\mathbb C\Gamma_C)$) with the tensor product of the injection 
$\Omega(C)\hookrightarrow \Omega_{hol}(\tilde C)^{\Gamma_C}$ with the map 
$\mathrm{Sh}_n(\mathrm H_C)\to \mathrm{Hom}_{\mathbb C-vec}(\mathrm{gr}^n(\mathbb C{\Gamma_C}),\mathbb C)$, which 
is injective by Lem. \ref{2123:3107}. It follows that $\mathrm{gr}_\bullet(\mathrm{can}\circ\varphi_{\sigma,x_0})$ is injective, which as 
the filtration of the source of $\mathrm{can}\circ\varphi_{\sigma,x_0}$ is exhaustive, implies the injectivity of
$\mathrm{can}\circ\varphi_{\sigma,x_0}$, which as
$\varphi_{\sigma,x_0}$ is an isomorphism implies the injectivity of $\mathrm{can}$. 

(b) In this proof, we abbreviate $F^\delta_\bullet\mathcal O_{hol}(\tilde C)$ into $F^\delta_\bullet$. For any $n\geq 0$, the inclusion $F_n^\delta\subset F_\infty\mathcal O_{mod}(\tilde C)$ gives rise to the inclusion of subspaces 
$\mathrm{im}(F_n^\delta\otimes \Omega(C)\to\Omega_{hol}(\tilde C))\subset \mathrm{im}(F_\infty\mathcal O_{mod}(\tilde C)
\otimes \Omega(C)\to\Omega_{hol}(\tilde C))\subset \Omega_{hol}(\tilde C)$. By the definition of $F_\bullet^\delta$, one has 
$d(F_{n+1}^\delta)\subset \mathrm{im}(F_n^\delta\otimes \Omega(C)\to\Omega_{hol}(\tilde C))$, therefore 
$d(F_{n+1}^\delta)\subset\mathrm{im}(F_\infty\mathcal O_{mod}(\tilde C)\otimes \Omega(C)\to\Omega_{hol}(\tilde C))$. This holds for any 
$n\geq 0$ and $F_\infty\mathcal O_{mod}(\tilde C)=\cup_{n\geq 0}F_n^\delta$ (see Prop. \ref{prop:3:15:0409}, Prop. \ref{new:cor} and 
Prop. \ref{(a3)}), therefore 
$d(F_\infty\mathcal O_{mod}(\tilde C))\subset\mathrm{im}(F_\infty\mathcal O_{mod}(\tilde C)\otimes \Omega(C)\to\Omega_{hol}(\tilde C))$. 
The linear map $F_\infty\mathcal O_{mod}(\tilde C)\otimes \Omega(C)\to\Omega_{hol}(\tilde C)$ admits a factorization 
$F_\infty\mathcal O_{mod}(\tilde C)\otimes \Omega(C)\twoheadrightarrow F_\infty\mathcal O_{mod}(\tilde C)\otimes_{\mathcal O(C)} 
\Omega(C)\to\Omega_{hol}(\tilde C)$
where the first map is surjective, therefore $\mathrm{im}(F_\infty\mathcal O_{mod}(\tilde C)\otimes \Omega(C)\to\Omega_{hol}(\tilde C))
=\mathrm{im}(F_\infty\mathcal O_{mod}(\tilde C)\otimes_{\mathcal O(C)} \Omega(C)\to\Omega_{hol}(\tilde C))$, which implies 
$dF_\infty\mathcal O_{mod}(\tilde C))\subset \mathrm{im}(F_\infty\mathcal O_{mod}(\tilde C)\otimes_{\mathcal O(C)} \Omega(C)\to
\Omega_{hol}(\tilde C))$. The claim then follows from (a). 

(c) The derivation property of $\nabla$ follows from the derivation property of $d$ and of the injectivity of the right vertical map of 
\eqref{diag;1809}. The equivariance of $\nabla$ follows from the same injectivity and from the $\Gamma_C$-equivariance of $d$; the identity 
$\nabla(f)=1\otimes df$ for $f\in\mathcal O(C)$ follows from the same injectivity.  
\end{proof}

\subsection{A connection for $(\mathrm{Sh}(\mathrm H_C),\mathrm{Sh}(\mathrm H_C)\otimes\mathcal O(C))$}
\label{sect:preimage:conn:2911}

\begin{lem}\label{lem:toto:0309}
(a) Let $O$ be a Hopf algebra and $\mathbf a$ be a commutative algebra. Then $(O,O\otimes\mathbf a)$, equipped 
with the coaction morphism $O\otimes \mathbf a\stackrel{\Delta_O\otimes id}{\longrightarrow} O\otimes(O\otimes\mathbf a)$,  
is a HACA satisfying the assumptions of Def. \ref{def:81:0711}. 

(b) Any linear map $\mu : O\to\mathbf a$ such that $\mu(fg)=\epsilon(f)\mu(g)+\mu(f)\epsilon(g)$ 
gives rise to a connection on $(O,O \otimes \mathbf a)$, 
where $\nabla : O \otimes\mathbf a\to O \otimes \Omega(\mathbf a)$ 
is given by $f \otimes a\mapsto f \otimes d(a)+f^{(1)} \otimes \mu(f^{(2)})a$.  
\end{lem}

\begin{proof}
(a) The fact that $(O,O\otimes\mathbf a)$  is a HACA follows from 
Prop. \ref{lem:B:15:1508}. The fact that it satisfies the assumptions of Def. \ref{def:81:0711} follows from $A^O=\mathbf a$.

(b) The statement follows from the axioms, and from the fact that the left and right action of $\mathbf a$ on 
$\Omega^1_{\mathbf a}$ coincide, as $\mathbf a$ is commutative. The axioms of 
invariance and restriction to $(O\otimes\mathbf a)^O=\mathbf a$ are immediate. 
\end{proof}

\begin{prop}\label{lem:8:5:2209} Let $J\in\mathrm{MC}(C)$. 

(a) The map $\nabla_{J} : \mathrm{Sh}(\mathrm H_C)\otimes\mathcal O(C)\to 
\mathrm{Sh}(\mathrm H_C)\otimes\Omega(C)$ given by $a\otimes f\mapsto a\otimes df+a^{(1)}
\otimes \mu_J(a^{(2)})f$, where $\mu_J$ is as in \S\ref{subsect:mu:J},
defines a connection for the HACA $(\mathrm{Sh}(\mathrm H_C),\mathrm{Sh}(\mathrm H_C)\otimes\mathcal O(C))$. 

(b) The connection $\nabla_{J}$ is the pull-back of $\nabla$ under the HACA isomorphism $(\nu(p_{J,x_0}),f_{J,x_0})$ 
(see Prop. \ref{prop:3:15:0409}).  
\end{prop}

\begin{proof}
(a) follows from Lems. \ref{lem:02:0208} and \ref{lem:toto:0309}. 

(b) For $a \in \mathrm{Sh}(\mathrm H_C)$, one has 
\begin{align*}
& \nabla(f_{J,x_0}(a \otimes 1))
=d(I_{x_0}(J_*(a)))
=I_{x_0}(J_*(a)^{(1)})\pi_{\mathrm{Sh}(\Omega(C))}(J_*(a)^{(2)})=I_{x_0}(J_*(a^{(1)}))\pi_{\mathrm{Sh}(\Omega(C))}(J_*(a^{(2)}))
\\ & 
=I_{x_0}(J_*(a^{(1)}))\mu_J(a^{(2)})
=f_{J,x_0} \otimes_{\mathcal O(C)} id_{\Omega(C)} \circ \nabla_J(f_{J,x_0}(a \otimes 1)) 
\end{align*}
where $\pi_{\mathrm{Sh}(\Omega(C))} : \mathrm{Sh}(\Omega(C))\to\Omega(C)$ is as in \S\ref{subsect:IIs:1008}, the second equality 
follows from \eqref{IND:1008}, the third equality follows from the fact that $J_*$ is an algebra morphism, and the fourth equality  
follows from the equality $\mu_J(a)=\pi(J_*(a))$ for any $a\in \mathrm{Sh}(\mathrm H_C)$, which follows from  
the definition of $\mu_J$.  One derives the commutativity of the diagram 
$$
\xymatrix{
\mathrm{Sh}(\mathrm H_C) \otimes \mathcal O(C)\ar^{f_{J,x_0}}[r]\ar_{\nabla_J}[d]&
F_\infty\mathcal O_{mod}(\tilde C)\ar^{\nabla}[d]
\\
\mathrm{Sh}(\mathrm H_C)\otimes_{\mathcal O(C)}\Omega(C)\ar_{f_{J,x_0}\otimes_{\mathcal O(C)}id_{\Omega(C)}}[r]&
F_\infty\mathcal O_{mod}(\tilde C)\otimes_{\mathcal O(C)}\Omega(C)
}
$$
\end{proof}

\begin{rem}
The HACA $(\mathrm{Sh}(\mathrm H_C),\mathrm{Sh}(\mathrm H_C)\otimes\mathcal O(C))$ corresponds to the 
trivial principal bundle over~$C$ with group $\mathrm{Spec}(\mathrm{Sh}(\mathrm H_C))$, and $\nabla_J$ is the flat 
connection $d+J$ over it; one has 
$\mathbb L((\mathrm H_C)^*)=\mathrm{Lie}\mathrm{Spec}(\mathrm{Sh}(\mathrm H_C))$. 
When $C=\mathbb P^1\smallsetminus S$ (see \S\ref{sect:92:0711}) and $J=J_{\sigma_0}$ (see Rem. \ref{rem:2:20}), 
$\nabla_{J_0}$ is the map $\mathrm{Sh}(\mathbb C \hat S_\infty) \otimes \mathbb C[z,1/(z-s),s \in S_\infty]\to 
\mathrm{Sh}(\mathbb C \hat S_\infty) \otimes \mathbb C[z,1/(z-s),s \in S_\infty]\cdot dz$ 
given by 
$[a_1|\ldots|a_k] \otimes f\mapsto [a_1|\ldots|a_k] \otimes df
+\sum_{s \in \hat S_\infty}(a_k)_s
[a_1|\ldots|a_{k-1}] \otimes f\cdot dz/(z-s)$, where for $a \in \mathbb C \hat S_\infty$, $a=\sum_{s \in S_\infty}
a_s\cdot \hat s$.
\end{rem}

\section{Local expansion of the elements of $F_\infty\mathcal O_{mod}(\tilde C)$}\label{sect:4:0711}

By \S\ref{mgfoad:0709}, one may view $u,z$ as elements of $\mathcal O_{hol}(\tilde D^\times)$ and $\mathcal O_{hol}(D^\times)$
respectively, such that $e^*z=e(u)$. We therefore use the notation $e^*\mathrm{log}z$ for $2\pi\mathrm{i}u$. 
Recall the action of $\mathbb Z$ on $\mathcal O_{mod}(\tilde D^\times)$, where $1$ acts by $\theta^*$ (cf. Lem. \ref{lem:debut:0924}). 
Let us denote by 
$F_\bullet\mathcal O_{mod}(\tilde D^\times)$ the algebra filtration of $\mathcal O_{mod}(\tilde D^\times)$ induced by this action 
according to Def. \ref{def:B5:0709} and Lem. \ref{1008:jeu}.
 
\begin{lem}\label{lem:local:0409}
The algebra morphism $\mathcal O(D^\times)[X]\to \mathcal O_{mod}(\tilde D^\times)$ given by $f\mapsto f$ for $f\in\mathcal O(D^\times)$ 
and $X\mapsto e^*\mathrm{log}z$ induces an isomorphism between the algebra filtrations $F_\bullet\mathcal O(D^\times)[X]$ of 
$\mathcal O(D^\times)[X]$ given by $F_n\mathcal O(D^\times)[X]:=\mathcal O(D^\times)[X]_{\leq n}$ for $n\geq 0$ 
and $F_\bullet\mathcal O_{mod}(\tilde D^\times)$ of $\mathcal O_{mod}(\tilde D^\times)$. In particular, one has $F_\infty\mathcal O_{mod}(\tilde D^\times)=\mathcal O(D^\times)[e^*\mathrm{log}z]$
(equality of subalgebras of $\mathcal O_{mod}(\tilde D^\times)$).
\end{lem}

\begin{proof}
Let us denote by $\mathrm{can} : \mathcal O(D^\times)[X]\to \mathcal O_{mod}(\tilde D^\times)$ the algebra morphism from the statement. The
image of $\mathcal O(D^\times)$ by $\mathrm{can}$ is contained in $F_0\mathcal O_{mod}(\tilde D^\times)$ and the image of 
$e^*\mathrm{log}z$ by $\mathrm{can}$ is contained in $F_1\mathcal O_{mod}(\tilde D^\times)$ 
since $(\theta^*-1)^2(e^*\mathrm{log}z)=(\theta^*-1)(1)=0$, therefore for any $n\geq 0$, 
$\mathrm{can}(F_n\mathcal O(D^\times)[X])=\mathrm{can}(\mathcal O(D^\times)[X]_{\leq n}) \subset F_n \mathcal O_{mod}(\tilde D^\times)$. 
Therefore  $\mathrm{can}$ is compatible with the algebra filtrations in its source and target. Let 
$$
\mathrm{gr}(\mathrm{can}) : \mathcal O(D^\times)[X]\to\mathrm{gr}(\mathcal O_{mod}(\tilde D^\times))
$$
be the corresponding graded algebra morphism. Its restriction to the degree 0 part of its source is given by 
$\mathcal O(D^\times) \ni f\mapsto f \in \mathcal O(D^\times)=F_0\mathcal O_{mod}(\tilde D^\times)=\mathrm{gr}_0
\mathcal O_{mod}(\tilde D^\times)$ and it is such that  $\mathcal O(D^\times)[X]\ni X\mapsto [e^*\mathrm{log}z] \in \mathrm{gr}_1
\mathcal O_{mod}(\tilde D^\times)$ (degree 1 elements). 

Let $n\geq 0$. By Lem. \ref{1008:jeu}(b), the linear map $\mu_n : F_n\mathcal O_{mod}(\tilde D^\times)\to \mathcal O_{mod}(\tilde D^\times)$ 
given by $f\mapsto (\theta^*-1)^n(f)/n!$ has its image contained in $F_0\mathcal O_{mod}(\tilde D^\times)$, whereas the image of the 
subspace $F_{n-1}\mathcal O_{mod}(\tilde D^\times)$ by this map is zero. 

Under the identification $\mathbb C\mathbb Z\simeq\mathbb C[X,X^{-1}]$ with $X$ group-like, the ideal 
$F_n\mathbb C\mathbb Z$ is identified with $((X-1)^n)$, therefore 
$\mathrm{Ker}(\mu_n)=\mathrm{Ann}((\theta^*-1)^n)=\mathrm{Ann}(F_n\mathbb C\mathbb Z)=F_{n-1}\mathcal O_{mod}(\tilde D^\times)$. 
It follows that $\mu_n$ induces an injective linear map 
$\mu'_n : \mathrm{gr}_n(\mathcal O_{mod}(\tilde D^\times) )\to F_0\mathcal O_{mod}(\tilde D^\times)$.
where by  Lem. \ref{lem:1:5:1304:0208}, the target space is equal to $\mathcal O(D^\times)$. 

Define a graded linear map 
$$
\mu : \mathrm{gr}(\mathcal O_{mod}(\tilde D^\times))\to\mathcal O(D^\times)[X]
$$
by $\mu(a):=\mu'_n(a)X^n$ 
for $a\in\mathrm{gr}_n(\mathcal O_{mod}(\tilde D^\times))$ and $n \geq 0$. As~$\mu$ is a direct sum of injective maps,~$\mu$
is injective. 

Let us show that for $n,m\geq 0$, the diagram
\begin{equation}\label{comm:0709}
\xymatrix{F_n\mathcal O_{mod}(\tilde D^\times)\otimes F_m\mathcal O_{mod}(\tilde D^\times)\ar^{\ \ \ \ \ \ \ \ \  
\mu'_n\otimes\mu'_m}[r]\ar[d] &
\mathcal O(D^\times)\otimes\mathcal O(D^\times)\ar[d]
\\ F_n\mathcal O_{mod}(\tilde D^\times) \ar_{\mu'_{n+m}}[r] & \mathcal O(D^\times)}
\end{equation}
is commutative, where the vertical maps are given by multiplication. The relation 
\begin{equation}\label{id;0709}
\tfrac{(X\otimes X-1)^{n+m}}{(n+m)!}\in \tfrac{(X-1)^n}{n!}\otimes  \tfrac{(X-1)^m}{m!}+((X-1)^{n+1})\otimes 
\mathbb C[X,X^{-1}]+\mathbb C[X,X^{-1}]\otimes ((X-1)^{m+1}) 
\end{equation}
in $\mathbb C[X,X^{-1}]^{\otimes 2}$  is a consequence of the relation
$$
\tfrac{(XY-1)^{n+m}}{(n+m)!}=\tfrac{((X-1)Y+(Y-1))^{n+m}}{(n+m)!}
\in 
\tfrac{((X-1)Y)^nY^m}{n!m!}+I=
\tfrac{(X-1)^n}{n!}\tfrac{(Y-1)^m}{m!}+I
$$
in $\mathbb C[X^{\pm1},Y^{\pm1}]$, 
where $I:=((X-1)^{n+1})+((Y-1)^{m+1})$. 
Then $f\in F_n\mathcal O_{mod}(\tilde D^\times)$, 
$g\in F_m\mathcal O_{mod}(\tilde D^\times)$, one has 
\begin{align*}
&\mu'_{n+m}(fg)=\tfrac{(\theta^*-1)^{n+m}(fg)}{(n+m)!}=\tfrac{(X-1)^{n+m}}{(n+m)!}\cdot (f\cdot g)=\Delta\Big(\tfrac{(X-1)^{n+m}}{(n+m)!}\Big)\cdot (f\otimes g)
\\ & 
\in \tfrac{(X-1)^n}{n!} \otimes \tfrac{(X-1)^m}{m!}+((X-1)^{n+1})\otimes\mathbb C[X,X^{-1}]+\mathbb C[X,X^{-1}]\otimes ((X-1)^{m+1}))
\cdot (f\otimes g)
\\ & =\tfrac{(X-1)^n}{n!}\cdot f\otimes\tfrac{(X-1)^m}{m!}\cdot g=\mu'_n(f)\cdot \mu'_m(g),  
\end{align*}
where the third equality follows from the Hopf algebra action properties, the inclusion relation follows from 
\eqref{id;0709} and the group-likeness of $X$, and the last equality follows from $(\theta^*-1)^{n+1}f=(\theta^*-1)^{m+1}g=0$
as $f\in F_n\mathcal O_{mod}(\tilde D^\times)$, $g\in F_m\mathcal O_{mod}(\tilde D^\times)$. 
The commutativity of \eqref{comm:0709} implies that $\mu : \mathrm{gr}(\mathcal O_{mod}(\tilde D^\times))\to\mathcal O(D^\times)[X]$ 
is a morphism of graded algebras. 

Then $\mu\circ\mathrm{gr}(\mathrm{can})$ is a graded algebra endomorphism of $\mathcal O(D^\times)[X]$. For $f\in \mathcal O(D^\times)$, one
has $\mu\circ\mathrm{gr}(\mathrm{can})(f)=\mu(f)=f$. One also has $\mu\circ\mathrm{gr}(\mathrm{can})(X)=\mu([e^*\mathrm{log}z])
=\mu'_1([e^*\mathrm{log}z])X=X$ as $\mu'_1([e^*\mathrm{log}z])=\mu_1(e^*\mathrm{log}z)=1$. It follows that 
$\mu\circ\mathrm{gr}(\mathrm{can})$ is the identity of $\mathcal O(D^\times)[X]$, therefore that $\mu$ is surjective. 
It is therefore an isomorphism, which using again $\mu\circ\mathrm{gr}(\mathrm{can})=id$ implies that 
$\mathrm{gr}(\mathrm{can})$ is an isomorphism. Lem. \ref{graded:crit:iso} then implies the statement.  
\end{proof}

\begin{prop}\label{toto:0409}
The algebra morphism $\bigoplus_{s\in S}\prod_{s\in X_s}(\tilde\varphi_{s,x}^\times)^*$ (see \eqref{alg:mor:0208}) is such that 
$$
\bigg(\bigoplus_{s\in S}\prod_{x\in X_s}(\tilde\varphi_{s,x}^\times)^*\bigg)(F_\infty\mathcal O_{mod}(\tilde C))
\,\subset\, \bigoplus_{s\in S}\prod_{x\in X_s}\mathcal O(D^\times)[e^*\mathrm{log}z]. 
$$
\end{prop}

\begin{proof} It follows from the proof of Prop. \ref{lem:2:12:1508} that \eqref{titi:0409} is a $\Gamma_C$-equivariant algebra morphism, 
where the action on the target is the direct sum over $s\in S$ of the actions of Lem. \ref{lem:1:8:0308}(a). 
For any $s\in S$, the composition of \eqref{titi:0409} with the canonical projection is a $\Gamma_C$-equivariant algebra morphism 
$\oplus_{x\in X_s}(\tilde\varphi_{s,x}^\times)^* : \mathcal O_{mod}(\tilde C)\to \oplus_{x\in X_s}\mathcal O_{mod}(\tilde D^\times)$. For 
$x\in X_s$, the decomposition of the target as $\mathcal O_{mod}(\tilde D^\times)\oplus (\oplus_{x'\in X_s\smallsetminus\{x\}}
\mathcal O_{mod}(\tilde D^\times))$ is preserved by the action of the stabilizer subgroup $\mathrm{Stab}_{\Gamma_C}(x)\subset \Gamma_C$ of 
$x\in X_s$ under the action of $\Gamma_C$. The map 
\begin{equation}\label{tutu:0409}
    (\tilde\varphi_{s,x}^\times)^* : \mathcal O_{mod}(\tilde C)\to\mathcal O_{mod}(\tilde D^\times)
\end{equation}
is therefore equivariant under the action of $\mathrm{Stab}_{\Gamma_C}(x)$. 

By \S\ref{subsect:mod:growth}, there is a group isomorphism $\mathbb Z\simeq \mathrm{Stab}_{\Gamma_C}(x)$ given by 
$1\mapsto\theta_{s,x}$, and \eqref{tutu:0409} is $\mathbb Z$-equivariant, the action of $\mathbb Z$ on the target being as in 
Lem. \ref{lem:local:0409}. Then 
\begin{align*}
& F_\infty\mathcal O_{mod}(\tilde C)=F_\infty^{\mathbb C\Gamma_C}\mathcal O_{mod}(\tilde C)
\subset F_\infty^{\mathbb C\mathbb Z}\mathcal O_{mod}(\tilde C)\subset
(\varphi_{s,x}^*)^{-1}(F_\infty^{\mathbb C\mathbb Z}\mathcal O_{mod}(\tilde D^\times))
\\ & 
=(\varphi_{s,x}^*)^{-1}(\mathcal O_{mod}(D^\times)[e^*\mathrm{log}z]) 
\end{align*}
where we use the notation of Lem. \ref{lem:for:local:expansion}, the first equality is the definition of 
$F_\infty\mathcal O_{mod}(\tilde C)$, the first inclusion follows from Lem. \ref{lem:for:local:expansion}(a), the second inclusion 
follows from Lem. \ref{lem:for:local:expansion}(b), and the last equality follows from Lem. \ref{lem:local:0409}. 
\end{proof}

\begin{rem}
Prop. \ref{toto:0409} implies that the elements of $F_\infty\mathcal O_{mod}(\tilde C)$ are Nilsson class functions
on $C$ in the sense of \cite{Ph}, p. 154. Indeed, for any $f\in \mathcal O(D^\times)[e^*\mathrm{log}z]$, there exists
$\alpha_0\in\mathbb Z$ such that $f=z^{\alpha_0}P_0(\mathrm{log}z)$ with $P_0$ as in \cite{Ph}, eq. (1.4), p. 151. The Nilsson class functions obtained in this way are not of the most general form, as the class in $\mathbb C/\mathbb Z$ of the $\alpha_i$ in their  expansion from 
\cite{Ph}, eq. (1.4), p. 151, is always $0$. 
\end{rem}

\section{Relation of $A_C$ with minimal acyclic extensions of dgas}\label{sect:54:2811}

Recall the definition of the dga $(\Omega_{hol}^\bullet(\tilde C),d)$ from \S\ref{subsect:Omega:mod:2811}. Denote by $(\Omega^\bullet(C),d)$ 
the dga of algebraic differential forms on $C$; it is concentrated in degrees 0 and 1, with $\Omega^0(C)=\mathcal O(C)$, $\Omega^1(C)=\Omega(C)$, 
algebra structure given by the algebra structure of $\mathcal O(C)$ and the module structure of $\Omega(C)$ over it, and differential 
given by $d$. The pull-back of $p : \tilde C\to C$ gives rise to an injective dga morphism 
$p^* : (\Omega^\bullet(C),d)\hookrightarrow (\Omega_{hol}^\bullet(\tilde C),d)$. 
%Recall the inclusion of dgas $(\Omega^\bullet(C),d)\hookrightarrow(\Omega^\bullet_{hol}(C),d)$ from \S\ref{sect:5:2:1512}. 
One has $\mathrm H^1(\Omega^\bullet(C))\simeq\mathrm H_C$, while $\mathrm H^1(\Omega_{hol}^\bullet(\tilde C))=0$.

\begin{defn}
(a) An acyclic extension (AE) of $\Omega^\bullet(C)$ is a dga $(E^\bullet,d)$ with $\Omega^\bullet(C)\subset E^\bullet\subset \Omega_{hol}^\bullet(\tilde C)$ and $\mathrm H^1(E^\bullet)=0$.  

(b) $E_C^\bullet:=\cap_{E^\bullet\text{ an AE of }\Omega^\bullet(C)}E^\bullet$. 
\end{defn}

\begin{lemdef}
$E_C^\bullet$ is an AE of $\Omega^\bullet(C)$, which is contained in any AE of $\Omega^\bullet(C)$; we therefore call it the 
{\it minimal AE} of $\Omega^\bullet(C)$. 
\end{lemdef}

\begin{proof}
Let us show that if $E^\bullet,F^\bullet$ are AEs of $\Omega^\bullet(C)$, then so is $E^\bullet\cap F^\bullet$. 
The intersection $E^\bullet\cap F^\bullet$ is obviously a dga containing $\Omega^\bullet(C)$. If $x\in E^1\cap F^1$, there exists 
$e\in E^0$ such that $x=d(e)$ since $\mathrm{H}^1(E^\bullet)=0$, and $f\in E^0$ such that $x=d(f)$ since $\mathrm{H}^1(F^\bullet)=0$. 
Then $d(e-f)=0$, therefore $e-f\in\mathbb C$, therefore $f\in e+\mathbb C\subset E^0+\mathbb C=E^0$, so $f\in E^0\cap F^0$. 
So $E^1\cap F^1\subset d(E^0\cap F^0)$, which implies $\mathrm{H}^1(E^\bullet\cap F^\bullet)=0$. It follows that 
$E^\bullet\cap F^\bullet$ is an AE of $\Omega^\bullet(C)$. This fact implies the statement. 
\end{proof}

\begin{lem}\label{lem:bij:AEs:AlgC}
The maps $E^\bullet \mapsto E^0$ and $A \mapsto (A \oplus dA,d)$ define inverse bijections between the sets  
$\{$AEs of $\Omega^\bullet(C)\}$ and $\mathrm{Alg}_C:=\{$algebras $A$ with $\mathcal O(C) \subset A \subset\mathcal O_{hol}(\tilde C)$ 
such that $d(A) \supset \Omega(C)$ 
and $A \cdot d(A)=dA$ (equality of subspaces of $\Omega_{hol}(\tilde C))\}$. 
\end{lem}

\begin{proof}
Let $E^\bullet$ be an AE of $\Omega^\bullet(C)$. Then $\mathcal O(C) \subset E^0 \subset \Omega_{hol}(\tilde C)$ since $\Omega^\bullet(C) 
\subset E^\bullet \subset \Omega^\bullet_{hol}(\tilde C)$. Since $\mathrm H^1(E^\bullet)=0$, one has $E^1=d(E^0)$ and since 
$\Omega^\bullet(C) \subset E^\bullet$, one has 
$E^1 \supset \Omega(C)$, therefore $d(E^0) \supset \Omega(C)$. The equality $E^1=d(E^0)$ implies the two extreme equalitites in 
$E^0\cdot d(E^0)=E^0\cdot E^1=E^1=d(E^0)$, while the middle equality follows from the fact that $E^\bullet$ is a dga with unit. It follows that 
$E^0$ belongs to $\mathrm{Alg}_C$. 

Let $A$ belong to $\mathrm{Alg}_C$. Since $A \cdot d(A)=d(A)$, the pair $(A \oplus d(A),d)$ is a sub-dga of $\Omega^\bullet_{hol}(\tilde C)$, 
and since $d(A) \supset \Omega(C)$ and 
$A \supset\mathcal  O(C)$, the dga $(A \oplus d(A),d)$ contains $\Omega^\bullet(C)$ as a sub-dga. One has clearly $\mathrm H^1(A \oplus dA,d)=0$. 
It follows that $(A \oplus d(A),d)$ is a AE of $\Omega^\bullet(C)$. 

The composed map $\mathrm{Alg}_C\to \{$AEs of $\Omega^\bullet(C)\}\to\mathrm{Alg}_C$ is obviously the identity, and the fact that the composed 
map $\{$AEs of $\Omega^\bullet(C)\}\to\mathrm{Alg}_C\to\{$AEs of $\Omega^\bullet(C)\}$ is the identity follows from the fact that if 
$E^\bullet$ is an AE of $\Omega^\bullet(C)$, then $E^1=d(E^0)$ due to $\mathrm H^1(E^\bullet)=0$. 
\end{proof}

\begin{lem}\label{lem:28112022}
(a) $\mathrm{Alg}_C \subset \{$SSAs of $\mathcal O_{hol}(\tilde C)\}$. 

(b) $A_C \in \mathrm{Alg}_C$. 

(c) $\cap_{A \in \mathrm{Alg}_C}A=A_C$.  
\end{lem}

\begin{proof}
(a) Let $A \in \mathrm{Alg}_C$. Then $A$ is unital. %Then $A \supset \mathcal O(C)$. 
Let $f \in A$, $\omega \in \Omega(C)$ and set 
$F:=(z\mapsto \int_{x_0}^z f\cdot p^*\omega)$. 
Since $\Omega(C) \subset d(A)$ (by $A \in \mathrm{Alg}_C$), there exists $g \in A$ with $dg=\omega$. Then 
$dF=f\cdot p^*\omega=f\cdot dg\in A\cdot dA=dA$, where the last equality follows from $A \in \mathrm{Alg}_C$. It follows that  
$F \in A$, therefore~$A$ is a SSA of $\mathcal O_{hol}(\tilde C)$.

(b) By Lem. \ref{lem:5:10:2811},  $A_C$ is a SSA of $\mathcal O_{hol}(\tilde C)$, therefore it contains $1$ and the functions
$z\mapsto \int_{x_0}^z df$ for any $f \in \mathcal O(C)$, therefore $A_C$ is a subalgebra of $\mathcal O_{hol}(\tilde C)$ 
containing $\mathcal O(C)$. 
Let $\omega \in \Omega(C)$ and set $F_\omega:=(z \mapsto \int_{x_0}^z p^*\omega) \in \mathcal O_{hol}(\tilde C)$. Then $F_\omega \in A_C$ by 
the stability properties of $A_C$. Then $d(A_C) \ni d(F_\omega)=\omega$, which implies $d(A_C) \supset \Omega(C)$.  

Let us show that $A_C\cdot dA_C=dA_C$. It suffices to prove that, for any $f,g \in A_C$, the element 
$h:=(z \mapsto \int_{x_0}^z f\cdot dg) \in \mathcal O_{hol}(\tilde C)$ belongs to $A_C$. 
By Prop. \ref{prop:AC:other} and  $F_\infty^\delta\mathcal O_{hol}(\tilde C)=\cup_{n \geq 0}F_n^\delta\mathcal O_{hol}(\tilde C)$, 
there exist $n,m \geq 0$ such that $f\in F_n^\delta\mathcal O_{hol}(\tilde C)$ and $g\in F_m^\delta\mathcal O_{hol}(\tilde C)$.
Then there exists a finite set~$I$ and maps $I\to \Omega(C)$, $I\to F_{m-1}^\delta\mathcal O_{hol}(\tilde C)$ denoted 
$i\mapsto \omega_i$, $i\mapsto k_i$ such that 
$dg=\sum_{i \in I} k_i\cdot p^*\omega_i$. Then $dh=\sum_{i\in I}fk_i\cdot p^*\omega_i$. 
For any $i \in I$, $fk_i \in F_n^\delta\mathcal O_{hol}(\tilde C)\cdot F_{m-1}^\delta\mathcal O_{hol}(\tilde C)
\subset F_{n+m-1}^\delta\mathcal O_{hol}(\tilde C)$, where the last inclusion follows from Prop. \ref{new:cor}(a). 
It follows that $h \in F_{n+m}^\delta\mathcal O_{hol}(\tilde C)$, and therefore $h \in A_C$ by 
$F_\infty^\delta\mathcal O_{hol}(\tilde C)=\cup_{n \geq 0}F_n^\delta\mathcal O_{hol}(\tilde C)$
and Prop. \ref{prop:AC:other}.     

(c) By (a), one has $\cap_{A \in \mathrm{Alg}_C}A \supset \cap_{A\text{ a SSA of }\mathcal O(C)}A=A_C$. On the other hand, 
$\cap_{A \in \mathrm{Alg}_C}A \subset A$ for any 
$A \in \mathrm{Alg}_C$, so by (b) $\cap_{A \in \mathrm{Alg}_C}A \subset A_C$.  
\end{proof}

\begin{prop}\label{prop:5:16:2811}
$E_C^\bullet=(A_C\oplus d(A_C),d)$. 
\end{prop}

\begin{proof}
Set $X:=\cap_{A\in\mathrm{Alg}_C}A$. One has
$$
E_C^\bullet=\cap_{E^\bullet\text{ a AE of }\Omega^\bullet(C)}E^\bullet=\cap_{A \in \mathrm{Alg}_C}(A \oplus d(A),d)
=(X\oplus d(X),d)=(A_C\oplus d(A_C),d)
$$
where the first equality follows from the definition of $E_C^\bullet$, the second equality follows from 
Lem.~\ref{lem:bij:AEs:AlgC}, the third equality follows from $(A \oplus d(A)) \cap (B \oplus d(B))=((A \cap B) \oplus 
d(A \cap B))$ (equality of subspaces of $\Omega^\bullet(\tilde C)$) for $A,B$ any pair of subspaces of $\mathcal O_{hol}(\tilde C)$ 
containing $1$, and the last equality follows from Lem. \ref{lem:28112022}(c). 
\end{proof}

\section{Computation of $\mathrm{Ker}(I_{x_0})$}\label{sect:8:0212}

In this section, we fix $x_0\in\tilde C$. The main result of this section is Thm. \ref{thm:7:1:0409}, where we compute 
$\mathrm{Ker}(I_{x_0})$ and exhibit a complement of this space in $\mathrm{Sh}(\Omega(C))$. These results are expressed in 
terms of an element $\sigma\in\Sigma_C$, which is fixed in the whole section. We also set $f_{\sigma,x_0}:=f_{J_\sigma,x_0}$. 

\begin{lemdef}
(a) The direct sum of the map $\sigma_*:\mathrm{Sh}(\mathrm H_C)\to\mathrm{Sh}(\Omega(C))$ and of its concatenation with the canonical inclusion $d\mathcal O(C)\hookrightarrow\Omega(C)$ is 
an injective map 
$\mathrm{Sh}(\mathrm H_C)\oplus [\mathrm{Sh}(\mathrm H_C)|d\mathcal O(C)]\to \mathrm{Sh}(\Omega(C))$. 

(b) We define  
\begin{equation}\label{def:Sub:sigma}
\mathrm{Sub}_\sigma:=\sigma_*(\mathrm{Sh}(\mathrm H_C))\oplus 
[\sigma_*(\mathrm{Sh}(\mathrm H_C))|d\mathcal O(C)]
\hookrightarrow \mathrm{Sh}(\Omega(C))
\end{equation}
to be the image of this map. 
\end{lemdef}

\begin{proof} (a) follows from $\Omega(C)=\sigma(\mathrm H_C)\oplus d\mathcal O(C)$. 
\end{proof}

Define a vector space filtration $F_\bullet\mathrm{Sub}_\sigma$
by $F_n\mathrm{Sub}_\sigma:=\sigma_*(F_n\mathrm{Sh}(\mathrm H_C))\oplus 
[\sigma_*(F_n\mathrm{Sh}(\mathrm H_C))|d\mathcal O(C)]$ for $n\geq 0$. 

\begin{lem}
One has for any $n\geq0$, 
\begin{equation}\label{incl2025}
I_{x_0}(F_n\mathrm{Sub}_\sigma)+I_{x_0}(F_n\mathrm{Sh}(\Omega(C)))=
f_{\sigma,x_0}(F_n\mathrm{Sh}(\mathrm H_C)\otimes\mathcal O(C))
+I_{x_0}(F_n\mathrm{Sh}(\Omega(C))), 
\end{equation}
\begin{equation}\label{incl:bis:0906}
f_{\sigma,x_0}(F_n\mathrm{Sh}(\mathrm H_C)\otimes \mathbb C1) 
\subset I_{x_0}(F_n\mathrm{Sub}_\sigma). 
\end{equation}
\end{lem}

\begin{proof} \eqref{incl2025} follows from the inclusion $F_{n-1}\mathrm{Sub}_\sigma \subset F_n\mathrm{Sh}(\Omega(C))$, from the fact that 
$\mathrm{gr}_n\mathrm{Sub}_\sigma$ is linearly spanned by the elements $[\sigma(h_1)|\ldots|\sigma(h_n)|df] \oplus 0$ and $0 \oplus [\sigma(h_1)|\ldots|\sigma(h_n)]$, where $h_1,\ldots,h_n\in \mathrm H_C$ and $f\in\mathcal O(C)$, and from the identities
\begin{equation}\label{useful:1}
I_{x_0}([\sigma(h_1)|\ldots|\sigma(h_n)|df]\oplus 0)=f_{\sigma,x_0}([h_1|\ldots|h_n]\otimes f)
-I_{x_0}([\sigma(h_1)|\ldots|\sigma(h_{n-1})|\sigma(h_n)p^*(f)]). 
\end{equation}
$$
I_{x_0}(0\oplus [\sigma(h_1)|\ldots|\sigma(h_n)])=f_{\sigma,x_0}([h_1|\ldots|h_n]\otimes 1). 
$$
Eq. \eqref{incl:bis:0906} follows from the equality of the maps 
$F_n\mathrm{Sh}(\mathrm H_C)\to \mathcal O_{hol}(\tilde C)$, $t\mapsto f_{\sigma,x_0}(t\otimes 1)$ and 
$t\mapsto I_{x_0}(\sigma(t)\oplus 0)$. 
\end{proof}

\begin{lem}
One has, for any $n\geq 0$, 
\begin{equation}\label{a:step}
f_{\sigma,x_0}(F_n\mathrm{Sh}(\mathrm H_C)\otimes\mathcal O(C))=I_{x_0}(F_n\mathrm{Sub}_\sigma). 
\end{equation}
\end{lem}

\begin{proof}
For any $n\geq 0$, one has 
\begin{align*}
& f_{\sigma,x_0}(F_n\mathrm{Sh}(\mathrm H_C)\otimes\mathcal O(C))
\subset I_{x_0}(F_n\mathrm{Sub}_\sigma)+I_{x_0}(F_n\mathrm{Sh}(\Omega(C)))
\\ &     
=I_{x_0}(F_n\mathrm{Sub}_\sigma)
+f_{\sigma,x_0}(F_n\mathrm{Sh}(\mathrm H_C)\otimes\mathbb C1+F_{n-1}\mathrm{Sh}(\mathrm H_C)\otimes\mathcal O(C))
\\ & 
=I_{x_0}(F_n\mathrm{Sub}_\sigma)
+f_{\sigma,x_0}(F_n\mathrm{Sh}(\mathrm H_C)\otimes\mathbb C1)
+f_{\sigma,x_0}(F_{n-1}\mathrm{Sh}(\mathrm H_C)\otimes\mathcal O(C))
\\ & 
=I_{x_0}(F_n\mathrm{Sub}_\sigma)
+f_{\sigma,x_0}(F_{n-1}\mathrm{Sh}(\mathrm H_C)\otimes\mathcal O(C))
\end{align*}
where the first inclusion follows from \eqref{incl2025}, the first equality follows from Prop. \ref{(a2)}, 
and the third equality follows from the inclusion \eqref{incl:bis:0906} applied to the 
two first summands of its left-hand side. Therefore one finds that 
$$
f_{\sigma,x_0}(F_n\mathrm{Sh}(\mathrm H_C)\otimes\mathcal O(C))
\subset I_{x_0}(F_n\mathrm{Sub}_\sigma)
+f_{\sigma,x_0}(F_{n-1}\mathrm{Sh}(\mathrm H_C)\otimes\mathcal O(C)). 
$$
Based on this inclusion, one proves inductively that, for any $n\geq 0$, one has 
\begin{equation}\label{temporary}
f_{\sigma,x_0}(F_n\mathrm{Sh}(\mathrm H_C)\otimes\mathcal O(C))
\subset I_{x_0}(F_n\mathrm{Sub}_\sigma). 
\end{equation}
On the other hand, 
for any $n\geq 0$, one has  
\begin{align*}
    & I_{x_0}(F_n\mathrm{Sub}_\sigma)\subset 
f_{\sigma,x_0}(F_n\mathrm{Sh}(\mathrm H_C)\otimes\mathcal O(C))
+I_{x_0}(F_n\mathrm{Sh}(\Omega(C))).
\\ & 
=f_{\sigma,x_0}(F_n\mathrm{Sh}(\mathrm H_C)\otimes\mathcal O(C))
+f_{\sigma,x_0}(F_n\mathrm{Sh}(\mathrm H_C)\otimes\mathbb C1+F_{n-1}\mathrm{Sh}(\mathrm H_C)\otimes\mathcal O(C))
\\ & 
=f_{\sigma,x_0}(F_n\mathrm{Sh}(\mathrm H_C)\otimes\mathcal O(C)) 
\end{align*}
where the inclusion follows from \eqref{incl2025}, the first equality follows from Prop. \ref{(a2)}, 
and the last equality follows from the inclusion of the second summand of its left-hand side in its first one. Therefore one finds that
$$
I_{x_0}(F_n\mathrm{Sub}_\sigma)\subset
f_{\sigma,x_0}(F_n\mathrm{Sh}(\mathrm H_C)\otimes\mathcal O(C)),
$$
which together with \eqref{temporary} implies the statement. 
\end{proof}

It follows from Prop. \ref{prop:3:15:0409}(a) that the corestriction of the map $f_{\sigma,x_0}$ defines a 
linear isomorphism
$\underline{f}_{\sigma,x_0} : F_n\mathrm{Sh}(\mathrm H_C)\otimes\mathcal O(C)\to 
f_{\sigma,x_0}(F_n\mathrm{Sh}(\mathrm H_C)\otimes\mathcal O(C))$. Define 
$$
\mathrm{map}_{\sigma,x_0} : \mathrm{Sub}_\sigma\to \mathrm{Sh}(\mathrm H_C)\otimes\mathcal O(C)
$$
to be the composition 
\begin{equation}\label{compos:1106}
\mathrm{Sub}_\sigma\stackrel{I_{x_0}}{\to}I_{x_0}(\mathrm{Sub}_\sigma)=
f_{\sigma,x_0}(\mathrm{Sh}(\mathrm H_C)\otimes\mathcal O(C))\stackrel{(\underline{f}_{\sigma,x_0})^{-1}}{\longrightarrow}
\mathrm{Sh}(\mathrm H_C)\otimes\mathcal O(C), 
\end{equation}
where the equality (equality of subspaces of $\mathcal O_{hol}(\tilde C)$) follows from the collection of all equalities \eqref{a:step} for $n\geq 0$. 

\begin{lem}\label{map:sigma:x0:iso}
The map $\mathrm{map}_{\sigma,x_0}$ is an isomorphism of vector spaces. 
\end{lem}

\begin{proof}
For any $n\geq0$, the composition \eqref{compos:1106} restricts to a composition 
$$
F_n\mathrm{Sub}_\sigma\stackrel{I_{x_0}}{\to}I_{x_0}(F_n\mathrm{Sub}_\sigma)=
f_{\sigma,x_0}(F_n\mathrm{Sh}(\mathrm H_C)\otimes\mathcal O(C))\stackrel{(\underline{f}_{\sigma,x_0})^{-1}}{\longrightarrow}
F_n\mathrm{Sh}(\mathrm H_C)\otimes\mathcal O(C), 
$$
where the equality follows from \eqref{a:step}. 

It follows that the map $\mathrm{map}_{\sigma,x_0}$ is compatible with the filtrations on both sides, and therefore 
induces, for any $n\geq 0$, a linear map 
$$
F_n \mathrm{map}_{\sigma,x_0} : F_n\mathrm{Sub}_\sigma\to F_n\mathrm{Sh}(\mathrm H_C)\otimes\mathcal O(C).
$$
The composition of the associated graded map with the canonical isomorphisms is then 
\begin{align}\label{graded:map}
&  \mathrm{map}_n : \sigma_*(\mathrm{Sh}_n(\mathrm H_C)) \oplus [\sigma_*(\mathrm{Sh}_n(\mathrm H_C))|d\mathcal O(C)]\simeq
\mathrm{gr}_n\mathrm{Sub}_\sigma
\stackrel{\mathrm{gr}_n \mathrm{map}_{\sigma,x_0}}{\to} \mathrm{gr}_n\mathrm{Sh}(\mathrm H_C)\otimes\mathcal O(C) 
\\ & \nonumber \simeq \mathrm{Sh}_n(\mathrm H_C)\otimes\mathcal O(C). 
\end{align}
Recall that a 2-step filtration of a vector space is the same as a vector subspace. The source and target of $\mathrm{map}_n$ are 
equipped with the 2-step filtrations associated respectively with the subspaces 
$\sigma_*(\mathrm{Sh}_n(\mathrm H_C))$
and $\mathrm{Sh}_n(\mathrm H_C)\otimes\mathbb C$. 

The composition \eqref{compos:1106} restricts to a composition 
$$
\sigma_*(\mathrm{Sh}(\mathrm H_C))\stackrel{I_{x_0}}{\to}
I_{x_0}\sigma_*(\mathrm{Sh}_n(\mathrm H_C)))=
f_{\sigma,x_0}(\mathrm{Sh}(\mathrm H_C)\otimes\mathbb C)\stackrel{(\underline{f}_{\sigma,x_0})^{-1}}{\to}
\mathrm{Sh}(\mathrm H_C)\otimes\mathbb C, 
$$
where the middle equality follows from the equality of maps $\mathrm{Sh}(\mathrm H_C)\to\mathcal O_{hol}(\tilde C)$, $t\mapsto I_{x_0}(\sigma_*(t))$
and $t\mapsto f_{\sigma,x_0}(t\otimes 1)$. Therefore $\mathrm{map}_{\sigma,x_0}$ restricts to the isomorphism  $\sigma_*(\mathrm{Sh}(\mathrm H_C))\to 
\mathrm{Sh}(\mathrm H_C)\otimes\mathbb C$ given by $\sigma_*(t)\mapsto t\otimes 1$. This map is compatible with the filtrations, therefore it induces 
an isomorphism $\sigma_*(F_n\mathrm{Sh}(\mathrm H_C))\to F_n\mathrm{Sh}(\mathrm H_C)\otimes\mathbb C$ for any $n\geq 0$, which by 
passing to the associated graded implies that $\mathrm{map}_n$ is compatible to the 2-step filtrations on both sides. Its associated graded for this
filtration is then a map 
$$
\oplus_{i=0}^1\mathrm{gr}_i \mathrm{map}_n : 
\sigma_*(\mathrm{Sh}_n(\mathrm H_C)) \oplus [\sigma_*(\mathrm{Sh}_n(\mathrm H_C))|d\mathcal O(C)]
\to \mathrm{Sh}_n(\mathrm H_C)\otimes(\mathbb C\oplus 
(\mathcal O(C)/\mathbb C)). 
$$
If follows from the fact that $\mathrm{map}_{\sigma,x_0}$ restricts to the isomorphism  
$\sigma_*(\mathrm{Sh}(\mathrm H_C))\to 
\mathrm{Sh}(\mathrm H_C)\otimes\mathbb C$ given by $\sigma_*(t)\mapsto t\otimes 1$ that $\mathrm{gr}_0 \mathrm{map}_n$ is the map 
$\sigma_*(\mathrm{Sh}_n(\mathrm H_C))\to F_n\mathrm{Sh}(\mathrm H_C)\otimes\mathbb C$, $\sigma_*(t)\mapsto t\otimes 1$. 
Moreover, \eqref{useful:1} implies 
\begin{align*}
    & I_{x_0}([\sigma(h_1)|\ldots|\sigma(h_n)|df]\oplus 0)\in f_{\sigma,x_0}([h_1|\ldots|h_n]\otimes f)
+I_{x_0}(F_n\mathrm{Sh}(\Omega(C))
\\ & \subset f_{\sigma,x_0}([h_1|\ldots|h_n]\otimes f)
+f_{\sigma,x_0}(F_\bullet\mathrm{Sh}(\mathrm H_C)\otimes\mathbb C1+F_{\bullet-1}\mathrm{Sh}(\mathrm H_C)\otimes\mathcal O(C))
\end{align*}
where the first relation follows from \eqref{useful:1} and the second follows from
Prop. \ref{(a2)}, thus implying that $\mathrm{gr}_1 \mathrm{map}_n$ is the map 
$[\sigma_*(\mathrm{Sh}_n(\mathrm H_C))|d\mathcal O(C)]\to \mathrm{Sh}_n(\mathrm H_C)\otimes(\mathcal O(C)/\mathbb C)$, 
$[\sigma_*(t)|df]\mapsto t\otimes [f]$. 

The maps $\mathrm{gr}_i\mathrm{map}_n$ are isomorphisms for $i=0,1$ and any $n\geq0$, which implies that $\mathrm{map}_n$ is an 
isomorphism for any $n\geq0$. This implies that $\mathrm{gr}_n \mathrm{map}_{\sigma,x_0}$ is an isomorphism for any $n\geq0$; since the filtrations 
on the source and target of $\mathrm{map}_{\sigma,x_0}$ are complete, one concludes that this map is an isomorphism. 
\end{proof}

\begin{lem}\label{lem:1:1117}
The restriction of $I_{x_0}$ to $\mathrm{Sub}_\sigma$ is injective. 
\end{lem}

\begin{proof}
This follows from the equality of this restriction with the composition $f_{\sigma,x_0}\circ \mathrm{map}_{\sigma,x_0}$, which is 
injective by Prop. \ref{prop:3:15:0409}(a) and Lem. \ref{map:sigma:x0:iso}.
\end{proof}

\begin{prop}\label{prop:2:1117}
One has $\mathrm{Sh}(\Omega(C))=\mathrm{Sub}_\sigma+\mathrm{im}(D_{x_0})$.
\end{prop}

\begin{proof}
Let us prove the inclusion 
\begin{equation}\label{incl:n}
    \mathrm{Sh}_n(\Omega(C))\subset\mathrm{Sub}_\sigma+\mathrm{im}(D_{x_0})+F_{n-1}\mathrm{Sh}(\Omega(C))
\end{equation}
for any $n\geq 0$. One has $\mathbb C1\in \mathrm{Sub}_\sigma$, which proves \eqref{incl:n} for $n=0$.  
One has $\Omega(C)=\sigma(\mathrm H_C)\oplus d\mathcal O(C)$, therefore $\mathrm{Sh}_1(\Omega(C))\subset\mathrm{Sub}_\sigma$; this 
proves  \eqref{incl:n} for $n=1$. Let $n\geq 2$. It follows from $\Omega(C)=d\mathcal O(C)+\sigma(\mathrm H_C)$ that  
\begin{equation}\label{incl:1:1040}
\mathrm{Sh}_n(\Omega(C))=[\sigma_*(\mathrm{Sh}_{n-1}(\mathrm H_C))|\Omega(C)]
+\sum_{k=0}^{n-1}[\mathrm{Sh}_{k-1}(\Omega(C))|d\mathcal O(C)|\mathrm{Sh}_{n-k}(\Omega(C))]. 
\end{equation}
One has 
\begin{equation}\label{incl:2:1040}
[\sigma_*(\mathrm{Sh}_n(\mathrm H_C))|\Omega(C)]=\sigma_*(\mathrm{Sh}_{n-1}(\mathrm H_C))\oplus 
[\sigma_*(\mathrm{Sh}_{n-1}(\mathrm H_C))|d\mathcal O(C)]\subset\mathrm{Sub}_\sigma. 
\end{equation}
Moreover for $k\in[\![0,n-1]\!]$, $f\in\mathcal O(C)$ and $\omega_i\in\Omega(C)$, $i\in [\![0,n]\!]\smallsetminus \{k\}$, 
one has 
\begin{align*}
&[\omega_1|\ldots|\omega_{k-1}|df|\omega_{k+1}|\ldots|\omega_n]
=D_{x_0}([\omega_1|\ldots|\omega_{k-1}]\otimes f\otimes [\omega_{k+1}|\ldots|\omega_n])
\\ & +[\omega_1|\ldots|\omega_{k-1}|p^*(f)\omega_{k+1}|\ldots|\omega_n]
-[\omega_1|\ldots|\omega_{k-1}p^*(f)|\omega_{k+1}|\ldots|\omega_n]
\\ & \in \mathrm{im}(D_{x_0})+F_{n-1}\mathrm{Sh}(\Omega(C))
\end{align*}
if $k>0$ and  
\begin{align*}
&[df|\omega_2|\ldots|\omega_n]
=D_{x_0}(1\otimes f\otimes [\omega_2|\ldots|\omega_n])
+[p^*(f)\omega_2|\ldots|\omega_n]
-f(x_0)[\omega_2|\ldots|\omega_n]
\\ & \in \mathrm{im}(D_{x_0})+F_{n-1}\mathrm{Sh}(\Omega(C))
\end{align*}
if $k=0$, which implies 
\begin{equation}\label{incl:3:1040}
[\mathrm{Sh}_{k-1}(\Omega(C))|d\mathcal O(C)|\mathrm{Sh}_{n-k}(\Omega(C))]\subset 
\mathrm{im}(D_{x_0})+F_{n-1}\mathrm{Sh}(\Omega(C)). 
\end{equation}
Then \eqref{incl:1:1040},  \eqref{incl:2:1040} and  \eqref{incl:3:1040} imply \eqref{incl:n}, which in its turn 
can be shown to imply  the statement by induction on $n$.  
\end{proof}

Let $\mathrm{Sh}_+(\Omega(C))$ be the augmentation ideal of $\mathrm{Sh}(\Omega(C))$. A left module structure of 
$\mathrm{Sh}_+(\Omega(C))$ over 
$\mathcal O(C)$ is defined by $f\cdot [\omega_1|\ldots|\omega_k]:=[f\omega_1|\ldots|\omega_k]$ and a right 
module structure of $\mathrm{Sh}(\Omega(C))$ over $\mathcal O(C)$ is defined by 
$[\omega_1|\ldots|\omega_k]\cdot f:=[\omega_1|\ldots|\omega_kf]$ if $k>0$ and $1\cdot f:=f(x_0)1$. Define a map 
$$
D_{x_0} : \mathrm{Sh}(\Omega(C)) \otimes \mathcal O(C)\otimes \mathrm{Sh}_+(\Omega(C))\to \mathrm{Sh}(\Omega(C)),\quad 
s\otimes f\otimes s'\mapsto [s|df|s']-[s|f\cdot s']+[s\cdot f|s'].  
$$
\begin{thm}\label{thm:7:1:0409}
(a) The sequence of maps 
$$
\mathrm{Sh}(\Omega(C)) \otimes \mathcal O(C)\otimes \mathrm{Sh}_+(\Omega(C))\stackrel{D_{x_0}}{\to} \mathrm{Sh}(\Omega(C))
\stackrel{I_{x_0}}{\to}\mathcal O_{hol}(\tilde C)
$$
is an exact complex, so that $\mathrm{Ker}(I_{x_0})=\mathrm{im}(D_{x_0})$. 

(b) There is a direct sum decomposition $\mathrm{Sh}(\Omega(C))=\mathrm{Sub}_\sigma\oplus\mathrm{Ker}(I_{x_0})$. 
\end{thm}

\begin{proof} (a) One checks that $I_{x_0}\circ D_{x_0}=0$, therefore $\mathrm{im}(D_{x_0})\subset\mathrm{Ker}(I_{x_0})$. Let us 
prove the opposite inclusion. The subspace $\mathrm{Sub}_\sigma\subset \mathrm{Sh}(\Omega(C))$ (see \eqref{def:Sub:sigma}) is such that
(i) the restriction of~$I_{x_0}$ to $\mathrm{Sub}_\sigma$ is injective (see Lem. \ref{lem:1:1117}), (ii) 
$\mathrm{Sh}(\Omega(C))=\mathrm{Sub}_\sigma+\mathrm{im}(D_{x_0})$ (see Prop. \ref{prop:2:1117}). Then~(i) implies 
$\mathrm{Sub}_\sigma\cap \mathrm{Ker}(I_{x_0})=0$, and therefore $\mathrm{Sub}_\sigma\cap \mathrm{im}(D_{x_0})=0$. Then (ii) implies that 
$\mathrm{Sh}(\Omega(C))=\mathrm{Sub}_\sigma\oplus\mathrm{im}(D_{x_0})$. The restriction of $I_{x_0}$ to the two summands are then respectively 
injective (by Lem. \ref{lem:1:1117}) and zero (by $I_{x_0}\circ D_{x_0}=0$), which implies $\mathrm{im}(D_{x_0})=\mathrm{Ker}(I_{x_0})$.

(b) This follows from combining the already proved equalities $\mathrm{Sh}(\Omega(C))=\mathrm{Sub}_\sigma\oplus\mathrm{im}(D_{x_0})$ and  
$\mathrm{im}(D_{x_0})=\mathrm{Ker}(I_{x_0})$.
\end{proof}

\appendix

\part{Appendices}

\section{Background on Hopf algebras}\label{sect:A:0112}

This section is devoted to constructions on Hopf algebras. In §\ref{sect:A:1:2308}, we define an endofunctor $O\mapsto 
F_\infty O$ of the category of Hopf algebras $\mathbf{HA}$, and in §\ref{sect:A2:0711}, a duality functor $\mathbf{HA} \supset 
\mathbf{HA}_{fd}\to\mathbf{HA}_{comm}$, $H\mapsto H'$. In §\ref{sect:A:3:3108}, we show that a Hopf algebra pairing 
$p : O \otimes H\to\mathbb C$ gives rise, under a finite-dimensionality assumption, 
to a Hopf algebra morphism $\nu(p) : F_\infty O\to H'$ (see Lem. \ref{lem:4:11:0307}).

\subsection{An endofunctor $O\mapsto F_\infty O$ of $\mathbf{HA}$}\label{sect:A:1:2308}

Let $O$ be a Hopf algebra with coproduct~$\Delta_O$. Recall that for 
$n\geq 0$, one defines $F_n  O:=\mathrm{Ker}(pr_O^{\otimes n+1}\circ\Delta_O^{(n+1)})\subset O$, where 
$\Delta_O^{(n)} : O\to O^{\otimes n}$ is the morphism obtained by iteration of 
$\Delta_O$ and $pr_O : O\to O/\mathbb C$ is the canonical projection (see Def. \ref{conilp:filtra}). 

\begin{lem}[see also \cite{Fr}, \S7.2.15] \label{lem:4:2:2106}
(a) For $n\geq 0$, $F_n  O\subset F_{n+1}  O$. 

(b) For $n\geq 0$ and $k\in[\![0,n+1]\!]$, one has $\Delta_O(F_n  O)\subset F_{k-1}  O\otimes O+O \otimes 
F_{n-k}  O$. 

(c) For $n\geq 0$, one has $\Delta_O(F_n  O)\subset \sum_{k=0}^n F_k O\otimes F_{n-k} O$.

(d) For $n,m\geq 0$, one has $(F_n   O)\cdot (F_m   O)\subset F_{n+m}   O$.  
\end{lem}

\begin{proof}
(a) Let $\eta_O,\epsilon_O$ be the unit and counit maps of $O$. One checks that 
$F_n   O=\mathrm{Ker}((id-\eta_O\epsilon_O)^{\otimes n+1}\circ\Delta_O^{(n+1)})$. One has 
$(id-\eta_O\epsilon_O)^{\otimes2}\circ\Delta_O=\Delta_O\circ (id-\eta_O\epsilon_O)$, which implies that 
$(id-\eta_O\epsilon_O)^{\otimes n+2}\circ\Delta_O^{(n+2)}=(\Delta_O\otimes id_O^{\otimes n})
\circ(id-\eta_O\epsilon_O)^{\otimes n+1}\circ\Delta_O^{(n+1)}$ (equality of linear maps $O\to O^{\otimes n+2}$).
Therefore $F_n   O=\mathrm{Ker}((id-\eta_O\epsilon_O)^{\otimes n+1}\circ\Delta_O^{(n+1)})
\subset \mathrm{Ker}((id-\eta_O\epsilon_O)^{\otimes n+2}\circ\Delta_O^{(n+2)})=F_{n+1}   O$. 

(b) Assume that $k\in [\![1,n]\!]$. One has $\Delta_O^{(n+1)}=(\Delta_O^{(k)}\otimes\Delta_O^{(n-k+1)})\circ\Delta_O$, and so $pr_O^{\otimes n+1}\circ\Delta_O^{(n+1)}=((pr_O^{\otimes k}\circ\Delta_O^{(k)})\otimes(pr_O^{\otimes 
n-k+1}\circ\Delta_O^{(n-k+1)}))\circ\Delta_O$. Therefore $\Delta_O(F_n   O)\subset \mathrm{Ker}(
(pr_O^{\otimes k}\circ\Delta_O^{(k)})\otimes(pr_O^{\otimes 
n-k+1}\circ\Delta_O^{(n-k+1)}))=\mathrm{Ker}(pr_O^{\otimes k}\circ\Delta_O^{(k)})\otimes O+O\otimes\mathrm{Ker}
(pr_O^{\otimes n-k+1}\circ\Delta_O^{(n-k+1)})=F_{k-1}   O\otimes O+O \otimes F_{n-k}   O$.

Assume that $k=0$. It follows from the statement with $k=1$ that $\Delta_O(F_n   O)\subset
\mathbb C\otimes O+O\otimes F_{n-1}   O$. It follows that for $f\in F_n   O$,  there exists $a\in O$ with 
$\Delta_O(f)\in 1\otimes a+O\otimes F_{n-1}   O$. Applying $\epsilon_O\otimes id$, one derives $f\in a+F_{n-1}   O$,
hence 
$a\in f+F_{n-1}   O\subset F_n   O$. Therefore $\Delta_O(F_n   O)\subset \mathbb C\otimes 
F_n   O+O\otimes F_{n-1}   O$. It follows that 
$\Delta_O(F_n   O)\subset O\otimes F_n   O$. 

For $k=n+1$, the proof of the statement $\Delta_O(F_n   O)\subset F_n   O\otimes O$ is similar, 
based on the statement for $k=n$. 

(c) It follows from the statements $\Delta_O(F_n   O)\subset O\otimes F_n   O$ and $\Delta_O(F_n   O)
\subset F_n   O\otimes O$ ((b) for $k=0,n$) that $\Delta_O(F_n   O)\subset F_n   O\otimes F_n   O$. 
Together with the statement (b) for $k\in[\![1,n]\!]$, this
implies that $\Delta_O(F_n   O)\subset F_{k-1}   O\otimes F_n   O+F_n   O\otimes F_{n-k}   O$ 
for any $k\in[\![1,n]\!]$. The statement then follows from 
 \begin{equation}\label{1450:2006}
\bigcap_{k=1}^n (F_{k-1}   O\otimes F_n   O+F_n   O\otimes F_{n-k}   O)=\sum_{k=0}^n 
F_k   O\otimes F_{n-k}   O,  
\end{equation}
which we now prove. For $i\in[\![0,n]\!]$, let $A_i$ be a complement of $F_i   O$ in $F_{i-1}   O$. One has 
then $F_n   O=\oplus_{i=0}^n A_i$. 
Then $F_{k-1}   O\otimes F_n   O+F_n   O\otimes F_{n-k}   O=\oplus_{(i,j)\in[\![0,n]\!]^2|
i\leq k-1\text{ or }j\leq n-k}A_i\otimes A_j$.
It follows that $\cap_{k=1}^n (F_{k-1}   O\otimes F_n   O+F_n   O\otimes F_{n-k}   O)
=\sum_{(i,j)\in S}A_i\otimes A_j$, where $S:=\{(i,j)\in[\![0,n]\!]^2\,|\,\forall\, k\in[\![1,n]\!],\text{ one 
has }i\leq k-1\text{ or }j\leq n-k\}$. One checks that 
$S=\{(i,j)\in[\![0,n]\!]^2\,|\,i+j\leq n\}$, therefore the left-hand side of \eqref{1450:2006} is equal to 
$\oplus_{(i,j)\in[\![0,n]\!]^2|i+j\leq n}A_i\otimes A_j\sum_{k=0}^n F_k   O\otimes F_{n-k}   O$, proving \eqref{1450:2006}. 

(d) It follows from (c) that 
\begin{equation}\label{equality:Deltas}
\Delta^{(n+m+1)}(F_n   O)\subset \sum_{\substack{(i_1,\ldots,i_{n+m+1})\in\mathbb Z_{\geq 0}^{n+m+1}|\\
i_1+\ldots+i_{n+m+1}=n}}
F_{i_1}   O\otimes\cdots\otimes F_{i_{n+m+1}}   O. 
\end{equation}
For $a\geq 1$ and $L\subset[\![1,a]\!]$, define $\varphi_L : [\![1,a]\!]\to\{0,1\}$
by $\varphi_L(x)=0$ if $x\in L$ and $\varphi_L(x)=1$ otherwise. Set then 
\begin{equation}\label{def;later:2106}
 O_L^{(a)}:=\otimes_{i=1}^{a}
F^{unit}_{\varphi_L(i)}O\subset O^{\otimes a},    
\end{equation}
where we recall $F^{unit}_0O=\mathbb C1$, $F^{unit}_1O=O$. 

Then for any $(i_1,\ldots,i_{n+m+1})\in\mathbb Z_{\geq 0}^{n+m+1}$, one has 
$F_{i_1}   O\otimes\cdots\otimes F_{i_{n+m+1}}   O\subset O_{\{j|i_j=0\}}^{(n+m+1)}$.
This and \eqref{equality:Deltas}, together with the fact that $|\{j\,|\,i_j=0\}|\geq m+1$ if
$(i_1,\ldots,i_{n+m+1})\in\mathbb Z_{\geq 0}^{n+m+1}$ is such that $i_1+\ldots+i_{n+m+1}\geq m+1$, imply that 
\begin{equation}\label{1521:2006}
\Delta^{(n+m+1)}(F_n   O)\subset \sum_{\substack{J\subset [\![1,n+m+1]\!]|\\|J|\geq m+1}}O_J^{(n+m+1)}. 
\end{equation}
Then 
\begin{align*}
&\Delta^{(n+m+1)}((F_n   O)\cdot(F_m   O))\subset \Delta^{(n+m+1)}(F_n   O)\cdot 
\Delta^{(n+m+1)}(F_m   O)
\\ & \subset \sum_{\substack{J,K\subset [\![1,n+m+1]\!]|\\|J|\geq m+1,|K|\geq n+1}}O_J^{(n+m+1)}\cdot O_K^{(n+m+1)}
\subset \sum_{\substack{J,K\subset [\![1,n+m+1]\!]|\\|J|\geq m+1,|K|\geq n+1}}O_{J\cap K}^{(n+m+1)}
\\ & \subset \sum_{\substack{L\subset [\![1,n+m+1]\!]|\\L\neq\emptyset}}O_{L}^{(n+m+1)}\subset \mathrm{Ker}(pr_O^{\otimes n+m+1})
\end{align*}
where the first inclusion follows from the fact that $\Delta^{(n+m+1)}$ is an algebra morphism, the second inclusion follows from 
\eqref{1521:2006}, the third inclusion follows from $O_J^{(n+m+1)}\cdot O_K^{(n+m+1)}\subset O_{J\cap K}^{(n+m+1)}$ for 
$J,K\subset [\![1,n+m+1]\!]$, the fourth inclusion follows from 
the fact that if $J,K\subset [\![1,n+m+1]\!]$ are such that $|J|\geq m+1,|K|\geq n+1$, then $J\cap K\neq\emptyset$, the last inclusion 
follows from the vanishing of the restriction of $pr_O^{\otimes n+m+1}$ to any $O_{L}$ where $\emptyset \neq L\subset [\![1,n+m+1]\!]$. 
The resulting inclusion $\Delta^{(n+m+1)}((F_n   O)\cdot(F_m   O))\subset\mathrm{Ker}(pr_O^{\otimes n+m+1})$ implies the 
statement. 
\end{proof}

\begin{prop}\label{prop:LA:filt}
(a) $F_\bullet O$ defines a Hopf algebra filtration of $O$.

(b) $F_\infty O$ is a Hopf subalgebra of $O$. The assignment $O\mapsto F_\infty O$ is an endofunctor of the category $\mathbf{HA}$
of Hopf algebras. 

(c) If $f:O_1\to O_2$ is a morphism in $\mathbf{HA}$, then $f$ is compatible with the filtrations $F_\bullet $ on both sides. 

(d) If $O$ is a $\mathbb Z_{\geq0}$-graded connected Hopf algebras, then $O=F_\infty O$. 
\end{prop}

\begin{proof}
(a) follows from Lem.  \ref{lem:4:2:2106}. (b) follows from the fact that $F_\infty O$ is the total space of $F_\bullet    O$ 
and from (a). The functoriality statement is obvious. (c) follows from 
$f^{\otimes n+1}\circ (id-\eta_{O_1}\epsilon_{O_1})^{\otimes n+1}\circ\Delta_{O_1}^{(n+1)}
=(id-\eta_{O_2}\epsilon_{O_2})^{\otimes n+1}\circ\Delta_{O_2}^{(n+1)}\circ f$ for any $n\geq0$. 
(d) If $n\geq 0$, then $\Delta_0^{(n+1)}(O[n])$ is contained in the sum of $O[k_1] \otimes\cdot \otimes O[k_{n+1}]$, 
where $(k_1,\ldots,k_{n+1})$ is such that $k_1+\ldots+k_{n+1}=n$. If $(k_1,\ldots,k_{n+1})$ is such a tuple, 
then there exists $i\in[\![1,n+1]\!]$ such that $k_i=0$, which by the connectedness of $O$ implies that the 
corresponding summand is contained in $O^{\otimes i-1} \otimes \mathbb C \otimes O^{\otimes n-i}$. 
Therefore $pr_O^{\otimes n+1} \circ \Delta_O^{(n+1)}(O[n])=0$, hence $O[n] \subset F_nO$. Therefore $O=F_\infty O$.   
\end{proof}
 
\subsection{A duality functor $\mathbf{HA}_{fd}\to \mathbf{HA}^{op}$, $H\mapsto H'$}\label{sect:A2:0711}

Let $H$ be a Hopf algebra with coproduct $\Delta_H$. Recall that $H_+$ is the augmentation ideal of $H$, and by 
$H_+^n$ the $n$-th power of this ideam. Set $F^nH:=H$ for $n=0$, $F^nH:=H_+^n$ for $n\geq 1$. %Set also $F_nH:=F^{-n}H$. 

\begin{lem}[see also \cite{Fr}, §8.1.1] \label{lem:elections:legislatives} % In (a)-(c), $H$ is an object in $\mathbf{HA}_{coco}$. 
 For $n,m\geq 0$, one has %$(F^nH)_{n\in \mathbb Z}$ is a decreasing Hopf algebra filtration of $H$, i.e. 
$F^nH\cdot F^mH\subset F^{n+m}H$ and 
$\Delta_{H}(F^nH)\subset\sum_{n'+n''=n}F^{n'}H\otimes F^{n''}H$. 
\end{lem}

\begin{proof} The decreasing character of $(F^nH)_{n\in\mathbb Z}$ is obvious. The inclusion 
$F^nH\cdot F^mH\subset F^{n+m}H$ follows from definitions. 
The last statement follows from $\Delta_H(H_+)\subset H\otimes H_++H_+\otimes H$, which is itself a consequence of the compatibility 
of $\Delta_{H}$ with the augmentation of $H$.  
\end{proof}

The coalgebra structure of $H$ induces an algebra structure on $H^*$. For $n\geq 0$, set $F_nH^*:=(F^{n+1}H)^\perp$. 

\begin{lemdef}\label{lem:def:H':3010}
$F_\bullet H^*$ is an algebra filtration of $H^*$, and $H':=\cup_{n\geq 0}F_nH^*$ is a subalgebra of~$H^*$. 
\end{lemdef}

\begin{proof}
Let $n,m\geq 0$, $\alpha\in (F^{n+1}H)^\perp$, $\beta\in (F^{m+1}H)^\perp$, then if $h\in F^{n+m+1}H$, one has 
$\Delta_H(h)\in F^nH\otimes H+H\otimes F^mH$ by the second statement of \ref{lem:elections:legislatives}, therefore 
$(\alpha\cdot\beta)(h)=(\alpha\otimes\beta)(\Delta_H(h))=0$, therefore $\alpha\cdot\beta\in (F^{n+m+1}H)^\perp$. 
This proves the first statement. The second statement follows from the first, as $H'$ is the total subspace of an 
algebra filtration. 
\end{proof}

\begin{defn}
$\mathbf{HA}_{fd}$ is the full subcategory of $\mathbf{HA}$ of Hopf algebras $H$ such that 
$\mathrm{gr}^1H:=F^1H/F^2H$ is finite dimensional. 
\end{defn}

\begin{lem}\label{lem:4:8:jeu}
If $H$ is an object of $\mathbf{HA}_{fd}$, then $H'$ is equipped with a linear map $\Delta_{H'} : H'\to H'\otimes H'$, uniquely determined 
by the identity $\Delta_{H'}(\alpha)(h\otimes h')=\alpha(hh')$ for $\alpha\in H'$ and $h,h'\in H$. Then $(H',\Delta_{H'})$ is a Hopf algebra. 
The assignment $H\mapsto H'$ is a functor $\mathbf{HA}_{fd}\to \mathbf{HA}^{op}$.  
\end{lem}

\begin{proof}
It follows from Lem. \ref{lem:elections:legislatives} that $F^\bullet H$ is a decreasing algebra filtration of $H$. The associated graded
algebra $\mathrm{gr}^\bullet H$ is such that $\mathrm{gr}^nH=F^nH/F^{n+1}H$ for any $n\geq 0$. Then $\mathrm{gr}^0H=\mathbb C$, 
and $\mathrm{gr}^\bullet H$ is generated by $\mathrm{gr}^1H$. As this space is finite dimensional, so is 
$\mathrm{gr}^kH$ for any $k\geq 0$. It follows that for any $n\geq 0$, $\oplus_{i=0}^{n-1}\mathrm{gr}^iH$ is finite dimensional. 
As this space is non-canonically isomorphic to the quotient space $H/F^nH$, this quotient is finite dimensional as well. 

It also follows from Lem. \ref{lem:elections:legislatives} that for any $n\geq 0$, $F^nH$ is a two-sided ideal of $H$, therefore
$H/F^nH$ is an algebra, and $H/F^{n+1}H\to H/F^nH$ is an algebra morphism. 

Since $H/F^nH$ is finite-dimensional, its associative algebra structure gives rise to a coassociative coalgebra structure
on its dual $(H/F^nH)^*=F_{n-1}H^*$ (with $F_{-1}H^*:=0$). It follows from the algebra morphism status of $H/F^{n+1}H\to H/F^nH$
that the canonical inclusion $i_{n,n+1} : F_{n-1}H^*\subset F_nH^*$ is a coalgebra morphism. 

For $n\geq 0$, define then an algebra morphism $\Delta_{H',n} : F_nH^*\to (H')^{\otimes 2}$ to be the composition 
$i_n^{\otimes2}\circ \Delta_{F_nH^*}$, where 
$\Delta_{F_nH^*}$ is the coproduct of the coalgebra structure of $F_nH^*$ and $i_n : F_nH^*\to H'$ is the canonical inclusion.  
One has 
\begin{equation}\label{compat:n:n+1}
 \Delta_{H',n+1}\circ i_{n,n+1}=\Delta_{H',n}.    
\end{equation}
Indeed, $\Delta_{H',n+1}\circ i_{n,n+1}=i_{n+1}^{\otimes2}\circ \Delta_{F_{n+1}H^*}\circ i_{n,n+1}=
i_{n+1}^{\otimes2}\circ i_{n,n+1}^{\otimes2}\circ \Delta_{F_nH^*}=
i_n^{\otimes2}\circ \Delta_{F_nH^*}=\Delta_{H',n}$, where the first and last equalities follow from the definitions of 
$\Delta_{H',n+1}$ and $\Delta_{H',n+1}$, the second equality follows from the coalgebra morphism status of $i_{n,n+1}$, and 
the third equality follows from $i_{n+1}\circ i_{n,n+1}=i_n$. It follows from \eqref{compat:n:n+1} and from 
$H'=\cup_{n\geq 0}F_nH^*$ that there is a unique linear map $\Delta_{H'} : H'\to 
(H')^{\otimes2}$, such that $\Delta_{H'}\circ i_n=\Delta_{H',n}$ for any $n\geq 0$. The identities relating 
$\Delta_{H',n}$ with the product implies that it satisfies the announced identity. The uniqueness statement 
follows from the fact that the annihilator of $H\otimes H$ in $F_nH^*\otimes F_mH^*$ is zero for any $n,m\geq 0$.
One checks that $\Delta_{H'}$ satisfies the Hopf algebra axioms as well as the functoriality statement. 
\end{proof}

If $H$ is an object of $\mathbf{HA}_{fd}$, it follows from Lem. \ref{lem:4:8:jeu} and Lem. \ref{prop:LA:filt} that $H'$ is equipped 
with a Hopf algebra filtration $F_\bullet  H'$; it is also equipped with the vector space filtration $F_\bullet H^*$
used to define it, given by $F_nH^*=(F^{n+1}H)^\perp$ for any $n\geq 0$. One has: 

\begin{lem}\label{IDFILT}
If $H$ is an object of $\mathbf{HA}_{fd}$, then $F_\bullet  H'=F_\bullet H^*$. 
One has $F_\infty H'=H'$.  
\end{lem}

\begin{proof}
Let $n \geq 0$ and let us show that $F_n  H'=F_nH^*$. Recall that $F_n  H'
=\mathrm{Ker}((id-\eta_{H'}\epsilon_{H'})^{\otimes n+1} \circ 
\Delta_{H'}^{(n+1)})$, where $\Delta_{H'},\eta_{H'}$ and $\epsilon_{H'}$ are the coproduct, unit and counit maps of $H'$. Recall that 
$H'=\cup_{m \geq 0}F_mH^*\subset H^*$ and that for each $m \geq 0$, $F_mH^*$ is a sub-coalgebra of $H'$; denote by $\Delta_{F_mH^*} : 
F_mH^*\to (F_mH^*)^{\otimes 2}$ the corresponding coproduct. 
Let $\epsilon_{F_mH^*} : F_mH^*\to\mathbb C$ be the composition of $\epsilon_{H'}$ with the inclusion $F_mH^* \subset H'$. 
The unit of $H'$ corresponds to the counit map of $H$, which as it vanishes on $F^{m+1}H$ defines an element in $(H/F^{m+1}H)^*=F_mH^*$; 
let $\eta_{F_mH^*}$ be the corresponding map $\mathbb C\to F_mH^*$. 
Then 
$$
F_n  H'\cap F_mH^*=\mathrm{Ker}\Big(F_mH^*\stackrel{\Delta_{F_mH^*}^{(n+1)}}{\longrightarrow}(F_mH^*)^{\otimes n+1}
\stackrel{(id-\eta_{F_mH^*}\epsilon_{F_mH^*})^{\otimes n+1}}{\longrightarrow}(F_mH^*)^{\otimes n+1}\Big)
$$ 
(a subspace of $F_mH^*$). Using that if $f:E\to F$ is a linear map, then 
$\mathrm{Ker}(f^* : F^*\to E^*)=\mathrm{im}(f)^\perp$, that the duals of 
$\Delta_{F_mH^*} : F_mH^*\to(F_mH^*)^{\otimes2}$, $\eta_{F_mH^*} : \mathbb C\to F_mH^*$ and $\epsilon_{F_mH^*} : F_mH^*\to\mathbb C$ 
are respectively the product map $m_{H/F^{m+1}H} : (H/F^{m+1}H)^{\otimes 2}\to H/F^{m+1}H$, the map $\epsilon_{H/F^m+1} : H/F^{m+1}\to\mathbb C$ 
induced by the counit of $H$, and the map $\eta_{H/F^{m+1}H} : \mathbb C\to H/F^{m+1}H$ induced the unit
of $H$, we see that 
$F_n  H' \cap F_mH^*$ is the annihilator (in $(H/F^{m+1}H)^*$) of the image of the map 
$$
(H/F^{m+1}H)^{\otimes n+1}\stackrel{(id-\eta_{H/F^{m+1}H} \epsilon_{H/F^{m+1}H})^{\otimes n+1}}{\longrightarrow}
(H/F^{m+1})^{\otimes 
n+1}\stackrel{m_{H/F^{m+1}H}^{(n+1)}}{\longrightarrow}
H/F^{m+1}H.
$$
Since the image of $id-\eta_H \epsilon_H : H\to H$ is $F^1H$, and since the image of $(F^1H)^{\otimes n+1}$ by 
$m_H^{(n+1)} : H^{\otimes n+1}\to H$ is $F^{n+1}H$, the said image is $(F^{n+1}H+F^{m+1}H)/F^{m+1}H$. Therefore the subspace 
$F_n  H' \cap F_mH^*$ of $(H/F^{m+1}H)^*$ is the annihilator of $(F^{n+1}H+F^{m+1}H)/F^{m+1}H$. If $m \geq n$, 
this subspace is the 
annihilator of $F^{n+1}H/F^{m+1}H$, which is the kernel of the canonical map $(H/F^{m+1}H)^*\to (F^{n+1}H/F^{m+1}H)^*$, which is the 
image of the injection $F_nH^* \hookrightarrow F_mH^*$. Therefore  $F_n  H'=F_nH^*$. 

One then has $F_\infty H'=\cup_{n \geq 0}(F_n H)^\perp=H'$, where the first equality follows from the 
previous statement and the second equality follows from the definition of $H'$. 
\end{proof}

\subsection{Hopf algebra pairings and Hopf algebra morphisms}\label{sect:A:3:3108}

Recall that if $O,H$ are Hopf algebras with coproducts $\Delta_O,\Delta_H$ and counits  $\epsilon_O,\epsilon_H$, 
then a Hopf algebra pairing between $O$ and $H$ is a linear map $p : 
O\otimes H\to\mathbb C$, such that $p(oo'\otimes h)=(p\otimes p)\circ\tau_2 (o\otimes o'\otimes \Delta_H(h))$, 
$p(o\otimes h')=(p\otimes p)\circ\tau_2(\Delta_O(o)\otimes h\otimes h')$ and $p(1\otimes h)=\epsilon_H(h)$, $p(o\otimes 1)=\epsilon_O(o)$
for $o,o'\in O$ and $h,h'\in H$ (for any $n\geq 2$, $\tau_n$ is the canonical map $O^{\otimes n}\otimes H^{\otimes n}
\to(O\otimes H)^{\otimes n}$).  

\begin{defn}\label{defn:A9}
For $O,H$ Hopf algebras, we denote by $\mathbf{Pair}(O,H)$ the set of Hopf algera pairings between $O$ and $H$. 
\end{defn}

\begin{lem}\label{lem:4:11:0307}
For $O$ an object of $\mathbf{HA}$ and $H$ an object of $\mathbf{HA}_{fd}$, there is a map
$$
\nu : \mathbf{Pair}(O,H)\to \mathbf{HA}(F_\infty O,H'). 
$$
For $p\in \mathbf{Pair}(O,H)$, $\nu(p)$ is such that the diagram 
$$
\xymatrix{F_\infty O\otimes H\ar^{\nu(p)\otimes id_H}[r]\ar_{i_O\otimes id_H}[d]&H'\otimes H\ar^{p_H}[d]\\ O\otimes H\ar_p[r]&\mathbb C}
$$
commutes, $i_O : F_\infty O\to O$ being the canonical injection and $p_H : H\otimes H'\to\mathbb C$ being the composition 
$H\otimes H'\hookrightarrow H\otimes H^*\to\mathbb C$. 
\end{lem}

\begin{proof}
Let $p\in \mathbf{Pair}(O,H)$. Let $n\geq 0$, then 
$p(F_n   O\otimes F^{n+1}H)=p(F_n   O\otimes H_+^{n+1})\subset p^{\otimes n+1}\circ \tau_{n+1}
(\Delta_O^{(n+1)}(F_n   O)\otimes H_+^{\otimes n+1})\subset\sum_{\emptyset\neq L\subset[\![1,n+1]\!]} 
p^{\otimes n+1}\circ \tau_{n+1}
(O_L^{(n+1)}\otimes H_+^{\otimes n+1})
=0$, where the first inclusion follows from the behavior of $p$ with respect to coproducts, the second
inclusion follows from $\Delta_O^{(n+1)}(F_n   O)\subset\sum_{\emptyset\neq L\subset[\![1,n+1]\!]}O_L^{(n+1)}$, where $O_L^{(n+1)}$ 
is as in \eqref{def;later:2106} (a consequence of the definition of $F_n   O$), the last inclusion follows from 
$p(1\otimes H_+)=0$, itself a consequence of the behavior of $p$ with respect to counits. 
It follows that $p(F_n   O\otimes F^{n+1}H)=0$. This implies that the restriction $p_{|F_n   O\otimes H}$ of $p$ to
$F_n   O\otimes H$ induces a linear map $p_n : F_n   O\otimes (H/F^{n+1}H)\to\mathbb C$. As $H/F^{n+1}H$ is 
finite dimensional, this gives rise to a linear map  $\nu(p)_n : F_n   O\to (H/F^{n+1}H)^*=F_nH^*$. 
One checks that $i_{n,n+1}^{H'}\circ \nu(p)_n=\nu(p)_{n+1}\circ i_{n,n+1}^O$, where $i_{n,n+1}^O$, $i_{n,n+1}^{H'}$
the canonical maps $F_n   O\to F_{n+1}   O$ and $F_nH^*\to F_{n+1}H^*$. It follows that there exists a unique linear map 
$\nu(p) : F_\infty O\to H'$, such that $\nu(p)\circ i_n^O=i_n^{H'}\circ \nu(p)_n$ for any $n\geq0$, where $i_n^O, i_n^{H'}$ 
are the canonical maps $F_n   O\to F_\infty O$ and $F_nH^*\to H'$. One checks that $\nu(p)$ defines a Hopf algebra morphism
as well as the announced commutative diagram. 
\end{proof}

\section{Background on Hopf algebras with (co)actions on algebras}\label{sect:B:0512}

In §\ref{sect:B1:0112}, we introduce a category $\mathbf{HACA}$ of Hopf algebras with comodule algebra (HACAs) and an endofunctor of 
that category extending that of §\ref{sect:A:1:2308}. In §\ref{B2:0112}, we introduce a category $\mathbf{HAMA}$ of Hopf algebras with 
module algebras (HAMAs), together with a diagram $\mathbf{HAMA} \supset \mathbf{HAMA}_{fd}\to\mathbf{HACA}$ extending that of 
§\ref{sect:A2:0711}. In §\ref{sect:B3:0112}, we introduce the notion of a pairing-morphism from a HAMA to a HACA and show that it gives 
rise, under a finite dimensionality assumption, to a HACA morphism extending the Hopf algebra morphism of §\ref{sect:A:3:3108}. In 
§\ref{sec:B4:0112}, we make use of the natural filtration attached to each HACA to construct two functors 
$\mathrm{gr},\Phi : \mathbf{HACA}\to\mathbf{Alg}_{gr}$ and a natural transformation $\mathrm{nat} : \mathrm{gr}\Rightarrow\Phi$ between 
these functors. §\ref{sect:B5:0112} is devoted to the main technical tool of this paper (Prop. \ref{lem:B:15:1508}), based on 
$\mathrm{nat}$ and giving a criterion for the HACA morphisms arising from pairing-morphisms of a certain type by the construction of 
§\ref{sect:B3:0112} to be isomorphisms. 

\subsection{An endofunctor $(O,A)\mapsto (F_\infty O,F_\infty A)$ of $\mathbf{HACA}$}\label{sect:B1:0112}

\begin{defn}\label{def:B1:0711}
A Hopf algebra with comodule algebra (HACA) is a pair $(O,A)$ where $O$ is a Hopf algebra and $A$ is an algebra, equipped with a left 
coaction of $O$; the coproduct map of $O$ being denoted $\Delta_O$ and the coaction by $\Delta_A : A\to O\otimes A$, 
in particular $\Delta_A$
is an algebra morphism and $(\Delta_O\otimes id_A)\circ\Delta_A=(id_O\otimes\Delta_A)\circ\Delta_A$. If $(O,A)$ and $(O',A')$
are HACAs, a morphism from $(O,A)$ to $(O',A')$ is the pair of $(f_A,f_O)$ of an algebra morphism $f_A : A\to A'$ and 
a Hopf algebra morphism $f_O : O\to O'$, such that $(f_O\otimes f_A)\circ\Delta_A=\Delta_{A'}\circ f_A$. Denote by $\mathbf{HACA}$
the resulting category.  
\end{defn}

Let $(O,A)$ be a HACA. 

\begin{defn}\label{def:FnA}
For $n\geq 0$, define $F_nA$ to be the preimage of $F_n   O\otimes A$ by the map $\Delta_A : A\to O\otimes A$. 
\end{defn}

\begin{lem}\label{lem:5:3:0307}
(a) $F_\bullet A$ is an algebra filtration of $A$. 

(b) For $n\geq 0$, $\Delta_A(F_nA)\subset F_n   O\otimes F_nA$. 

(c) For $n\geq 0$, $\Delta_A(F_nA)\subset \sum_{p+q=n}F_p   O\otimes F_qA$.

(d) The restriction of $\Delta_A$ to $F_\infty A$ corestricts to $F_\infty O\otimes F_\infty A$. Together with the structures of  
algebra of $F_\infty A$ and of Hopf algebra of $F_\infty O$, the resulting map  
$\Delta_{F_\infty A} : F_\infty A\to F_\infty O\otimes F_\infty A$ equips $(F_\infty O,F_\infty A)$ with a HACA structure.  
The assignment $(O,A)\mapsto (F_\infty O,F_\infty A)$ is an endofunctor of $\mathbf{HACA}$. 
\end{lem}

\begin{proof} (a) $A\otimes F_\bullet    O$ is an algebra filtration of the algebra $O\otimes A$; as $\Delta_A : 
A\to O\otimes A$ is an algebra morphism, the preimage of this filtration is an algebra filtration of the source; as this preimage is 
$F_\bullet A$, the latter is an algebra filtration. (b) Let $a\in F_nA$, then $\Delta_A(a)\in F_n   O\otimes A$. 
There exist families $(o_s)_{s\in S}$, $(a_s)_{s\in S}$ of elements of $O$ and $A$, indexed by a finite set $S$, such that 
$(o_s)_{s\in S}$ is free and $\Delta_A(a)=\sum_s o_s\otimes a_s$. Since $\Delta_O(F_n   O)\subset (F_n   O)^{\otimes2}$, 
one has $(\Delta_O\otimes id_A)\circ\Delta_A(a)\in (F_n   O)^{\otimes2}\otimes A$. By the coassociativity of the 
coaction, this term is equal to $(id_O\otimes\Delta_A)\circ\Delta_A(a)$, so 
$(id_O\otimes\Delta_A)\circ\Delta_A(a)\in (F_n   O)^{\otimes2}\otimes A$. Therefore $\sum_{s\in S}o_s\otimes
\Delta_A(a_s)\in (F_n   O)^{\otimes2}\otimes A$. As $(o_s)_{s\in S}$ is free, this implies $\Delta_A(a_s)
\in F_n   O\otimes A$ for any $s$, therefore $a_s\in F_nA$. The claim follows. 

(c) Let $a\in F_nA$, then by (b) $\Delta_A(a)\in F_n   O\otimes F_nA$. 
Let $U_i$ be a complement of $F_{i-1}   O$ in $F_i   O$, so $F_i   O=U_1\oplus\cdots\oplus U_i$. 
Then $F_n   O\otimes F_nA 
=(U_0\otimes  F_nA)\oplus\cdots\oplus(U_n\otimes  F_nA)$, giving rise to a decomposition 
$\Delta_A(a)=z_0+\cdots+z_n$. 

Since $\Delta_A(a)\in  F_n   O\otimes F_nA$, Lem.  \ref{lem:4:2:2106} (c) implies that 
$(\Delta_O\otimes id_A)\circ \Delta_A(a)\in 
\sum_{p+q=n}F_p   O\otimes F_q   O \otimes F_nA \subset (F_n   O)^{\otimes 2}\otimes F_nA$. By the 
coassociativity identity, this 
term is equal to $(id_O\otimes  \Delta_A)\circ \Delta_A(a)=\sum_{i=0}^n(id_O\otimes\Delta_A)(z_i)$. It follows that 
 $\sum_{i=0}^n(id_O\otimes\Delta_A)(z_i)\in \sum_{p+q=n}F_p   O\otimes F_q   O\otimes F_nA$.
There is a direct sum decomposition $(F_n   O)^{\otimes 2}\otimes F_nA=\oplus_{i,j\in[\![0,n]\!]}U_i\otimes U_j\otimes F_nA$; then 
 $\sum_{p+q=n}F_p   O\otimes F_q   O\otimes F_nA =\oplus_{i,j\in[\![0,n]\!],i+j\leq n}U_i\otimes U_j\otimes F_nA$, 
 and $(id_O\otimes \Delta_A)(z_i)\in U_i\otimes F_n   O\otimes F_nA=\oplus_{j\in[\![0,n]\!]}U_i\otimes U_j\otimes F_nA$. 
 Decomposing $(id_O\otimes\Delta_A)(z_i)=\sum_{j\in[\![0,n]\!]}t_{ij}$ for any $i\in[\![0,n]\!]$, one obtains $\sum_{i,j\in[\![0,n]\!]}t_{ij}\in 
 \oplus_{i,j\in[\![0,n]\!],i+j\leq n}U_i\otimes U_j\otimes F_nA$, therefore $t_{ij}=0$ if $i+j>n$, therefore for any $i\in[\![0,n]\!]$, one has 
 $(id_O\otimes\Delta_A)(z_i)\in \oplus_{j=0}^{n-i} U_i\otimes U_j\otimes F_nA=U_i\otimes F_{n-i}   O\otimes F_nA$. 
 There exist families $(o_s^i)_{s\in S_i}$, $(a^i_s)_{s\in S_i}$  of elements of $U_i$ and $F_nA$, indexed by a finite set 
 $S_i$, such that $(o_s^i)_{s\in S_i}$ is free and $z_i=\sum_{s\in S_i} o_s^i\otimes a_s^i$. 
Then $\sum_{s\in S_i}o_s^i\otimes \Delta_A(a_s^i) \in U_i\otimes F_{n-i}   O \otimes F_nA$. Since 
$(o_s^i)_{s\in S_i}$ is free, this 
implies $\Delta_A(a_s^i)\in F_{n-i}   O\otimes F_nA$ for any $s\in S_i$, therefore $a_s^i\in F_{n-i}A$. Therefore $z_i\in  F_i  O\otimes F_{n-i}A$, proving the claim. 

 (d) By (b), $\Delta_A(F_nA)\subset F_\infty O \otimes F_\infty A$ for any $n \geq 0$, which implies $\Delta_A(F_\infty A) \subset 
F_\infty O \otimes F_\infty A$. One can then define the announced map $\Delta_{F_\infty A} : F_\infty A\to F_\infty O \otimes 
F_\infty A$ and check it to have the announced properties. The endofunctor statement is then immediate. 
\end{proof}

\subsection{A category $\mathbf{HAMA}$ and a functor $\mathbf{HAMA}_{fd}\to\mathbf{HACA}$}\label{B2:0112}

For $M$ a right module over an algebra $A$ and $V\subset A$, we use the notation 
$\mathrm{Ann}(V)$ to denote $\{m\in M\,|\,m\cdot V=0\}$. 

\begin{defn}\label{MISSINGDEF}
For $H$ a Hopf algebra, $M$ a right $H$-module, define $F_nM:=\mathrm{Ann}(F^{n+1}H,M)$ for $n\geq0$. 
\end{defn}

Note that $F_0M$ is equal to $M^H$, the submodule of $M$ of invariants of the action of $H$.  

\begin{lem}\label{MISSINGLEM}\label{INJLEM}
Let $H$ be a Hopf algebra.  

(a) For $M$ a right $H$-module, $F_\bullet M$ is an increasing $H$-module filtration of $M$. 

(b) Any morphism $f : M\to N$ a morphism of right $H$-modules is compatible with the filtrations. 

(c) If $M$ is a right $H$-module and $n\geq 0$, then for any $m\in F_nM$, the map $F^nH\to M$, $h\mapsto m\cdot h$
takes its values in $M^H$. The resulting map $F_nM\to \mathrm{Hom}_{\mathbb C-vec}(F^nH,M^H)$ induces a map 
$\mathrm{gr}_nM\to \mathrm{Hom}_{\mathbb C-vec}(\mathrm{gr}^nH,M^H)$, which is injective. 
\end{lem}

\begin{proof}
(a) follows from the decreasing character of $n\mapsto F^{n+1}H$ and from the ideal nature of 
$F^{n+1}H$ for $n\geq0$. (b) is immediate. (c) If $m\in F_nM$, $h\in F^nH$ and $h'\in F_1H$, then $hh'\in F^{n+1}H$ so $m\cdot(hh')=0$, 
therefore $(m\cdot h)\cdot h'=0$, so $m\cdot h\in M^H$. The map $F_nM\to\mathrm{Hom}_{\mathbb C-vec}(F^nH,M^H)$ factors through a map $F_nM\to 
\mathrm{Hom}_{\mathbb C-vec}(\mathrm{gr}^nH,M^H)$ as $m\cdot h=0$ for any $m\in F_nM$, $h\in F^{n+1}H$. The restriction of this map to 
$F_{n-1}M$ is zero as $m\cdot h=0$ for any $m\in F_{n-1}M$, $h\in F^nH$. This leads to the announced map 
$\mathrm{gr}_nM\to \mathrm{Hom}_{\mathbb C-vec}(\mathrm{gr}^nH,M^H)$. As the map $\mathrm{Hom}_{\mathbb C-vec}
(\mathrm{gr}^nH,M^H)\to \mathrm{Hom}_{\mathbb C-vec}(F^nH,M^H)$ is injective, the kernel of the map $F_nM\to 
\mathrm{Hom}_{\mathbb C-vec}(\mathrm{gr}^nH,M^H)$ is equal to the kernel of the map $F_nM\to \mathrm{Hom}_{\mathbb C-vec}
(F^nH,M^H)$, which is $\{m\in F_nM\,|\,m\cdot h=0$ for any $h\in F^nH\}$, that it $F^{n-1}H$. It follows that the map 
$\mathrm{gr}_nM\to \mathrm{Hom}_{\mathbb C-vec}(\mathrm{gr}^nH,M^H)$ is injective. 
\end{proof}

\begin{defn}\label{defn:B6}
A Hopf algebra with module algebra (HAMA) is pair $(B,H)$ where $H$ is a
Hopf algebra and $B$ is an algebra equipped with a right action of $H$, i.e. a linear map $B\otimes H\to H$, $b \otimes h\mapsto b\cdot h$,  
such that $(b\cdot h)\cdot h'=b\cdot (hh')$ and $(bb')\cdot h=(b\cdot h^{(1)})(b'\cdot h^{(2)})$ for any $b,b'\in B$ and $h,h'\in H$, the 
coproduct of $h\in H$ being denoted $h^{(1)} \otimes h^{(2)}$ (Sweedler).  
If $(B,H)$ and $(B',H')$ are HAMAs, a morphism $(B,H)\to(B',H')$ is a pair $(f_H, f_B)$, where $f_H : H' \to H$ is a a Hopf algebra morphism and 
$f_B : B\to B'$ is an algebra morphism, such that $f_B(b\cdot f_H(h'))=f_B(b)\cdot h'$ for any $b \in B$ and $h' \in H'$. Morphisms 
are composed as follows: $(g_H,g_B)\circ (f_H,f_B):=(f_H \circ g_H,g_B \circ f_B)$. 
\end{defn}

One checks that HAMAs build up a category, which will be denoted $\mathbf{HAMA}$.

\begin{defn}\label{def:B5:0709}
Let $(B,H)$ be a HAMA. 
For $n\geq 0$, set $F_nB:=\mathrm{Ann}(F^{n+1}H,B)=\{b \in B\,|\,b\cdot F^{n+1}H=0\}$.  
\end{defn}

\begin{lem}\label{1008:jeu}
Let $(B,H)$ be a HAMA. 

(a) $F_\bullet B$ is an algebra filtration of $B$.

(b) Let $n,m\geq 0$. For $n\geq m$, then $(F_n B)\cdot (F^m H) \subset F_{n-m}B$. If $n<m$, then 
$(F_n B)\cdot (F^m H)=0$.

(c) For any $n\geq0$, the action map $B\otimes H\to B$ induces a linear map $F_nB\otimes H\to F_nB$, which factorizes through a linear map $F_nB\otimes (H/F^{n+1}H)\to F_nB$. 
\end{lem}

\begin{proof}
(a) Let $n,m\geq0$. By the second part of Lem. \ref{lem:elections:legislatives}, one has 
$\Delta_H(F^{n+m+1}H)\subset F^{n+1}H\otimes H+H\otimes F^{m+1}H$. 
Together with the Hopf algebra action axiom, this implies the inclusion in  
$((F_nB)(F_mB))\cdot (F^{n+m+1}H) \subset ((F_nB)\cdot(F^{n+1}H))((F_mB)\cdot H)+((F_nB)\cdot H)((F_mB)\cdot(F^{m+1}H))=0$, and the equality follows from $(F_nB)\cdot(F^{n+1}H)=
(F_mB)\cdot(F^{m+1}H)=0$. The resulting equality $(F_nB)(F_mB)\cdot (F^{n+m+1}H)=0$ implies $(F_nB)(F_mB)\subset F_{n+m}B$. 

(b) Assume $n\geq m\geq 0$. Then $((F_nB)\cdot(F^mH))\cdot (F^{n-m+1}H)\subset 
(F_nB)\cdot((F^mH)(F^{n-m+1}H))\subset (F_nB)\cdot(F^{n+1}H)=0$, where the first 
inclusion follows from the right action axioms, the second inclusion follows from the first 
part of Lem. \ref{lem:elections:legislatives}, and the equality follows from the definition of 
$F_nB$. This implies $(F_nB)\cdot(F^mH)\subset F_{n-m}B$, as claimed. In particular, 
for any $n\geq 0$, $(F_nB)\cdot(F^nH)\subset F_0B=B^H$. Then if $m>n$, one has 
$(F_nB)\cdot(F^mH)=(F_nB)\cdot((F^nH)(F^{m-n}H))\subset 
((F_nB)\cdot(F^nH))\cdot(F^{m-n}H)\subset B^H\cdot (F^{m-n}H)=0$, where the 
first equality follows from $F^mH=(F^nH)(F^{m-n}H)$, the first inclusion 
follows from the right action axioms, the second inclusion follows from the inclusion
$(F_nB)\cdot(F^nH)\subset B^H$, the last equality follows from $B^H\cdot H_+=0$. 
This completes the proof of the claim. 

(c) The first statement follows from (b) for $m=0$. The second statement follows from 
$F_nB\cdot (F^{n+1}H)=0$, which follows from the definition of $F_nB$.
\end{proof}

\begin{defn}\label{def:HAMA:fd}
$\mathbf{HAMA}_{fd}$ is the full subcategory of $\mathbf{HAMA}$ of objects $(B,H)$ such that $H$ is an object of $\mathbf{HA}_{fd}$. 
\end{defn}

\begin{lem}\label{lem:5:8:0307}
Let $(B,H)$ be an object of $\mathbf{HAMA}_{fd}$. 

(a) For any $n\geq 0$, the linear map $F_nB\otimes (H/F^{n+1}H)\to F_nB$
from Lem. \ref{1008:jeu} (c) induces a linear map $\Delta_{F_nB} : F_nB\to 
(H/F^{n+1}H)^*\otimes F_nB=F_nH^*\otimes F_nB$. 

(b) There is a linear map $\Delta_{F_\infty B} : F_\infty B\to H'\otimes F_\infty B$, 
uniquely determined by the condition that $\Delta_{F_\infty B}\circ i_n^{F_\infty B}=
(i_n^{H'}\otimes i_n^{F_\infty B})\circ \Delta_{F_nB}$ for any $n\geq0$, where $\Delta_{F_nB}$
is as in (a) and $i_n^{H'} : F_nH^*\to H'$, $i_n^{F_\infty B} : F_nB\to F_\infty B$ are the 
canonical injections. 

(c) Together with the algebra structure of $F_\infty B$ and Hopf algebra structure of $H'$, 
$\Delta_{F_\infty B}$ equips $(H',F_\infty B)$ with a HACA structure. The assignment
$(B,H)\mapsto(H',F_\infty B)$ is a functor $\mathbf{HAMA}_{fd}\to\mathbf{HACA}$. 
\end{lem}

\begin{proof}
(a) follows from the finite-dimensionality of $H/F^{n+1}H$, which has been established in the proof 
of Lem. \ref{lem:4:8:jeu}. 

(b)  Let $n\geq 0$  and $i_{n,n+1}^{F_\infty B} : F_nB\hookrightarrow F_{n+1}B$ 
and $i_{n,n+1}^{H'} : F_nH^*\hookrightarrow F_{n+1}H^*$ be the canonical inclusions. Recall that 
$i_{n,n+1}^{H'}$ is the dual of the projection map $p_{n+1,n}^H : H/F^{n+2}H\to H/F^{n+1}H$. 
The following diagram 
$$
\xymatrix{ &F_nB \otimes H\ar[r]\ar[ld]\ar[d] & F_nB \otimes (H/F^{n+1}B)\ar[dd]\\ F_{n+1}B\otimes H\ar[r]\ar[d]& B\otimes H\ar[rd]& \\ F_{n+1}B\otimes (H/F^{n+2}B)\ar[rr]& & B}
$$
is built up of two commutative square and a commutative triangle, therefore is commutative; it is the composition of the square whose commutativity expresses the equality  
$\Delta_{F_\infty B}\circ i_{n,n+1}^{F_\infty B}=
(i_{n,n+1}^{F_\infty B}\otimes i_{n,n+1}^{H'})\circ \Delta_{F_nB}$
with the morphisms $F_{n+1}B\to B$ and $F_nB\otimes H\to F_nB\otimes(H/F^{n+2}H)$, 
which are respectively injective and surjective. This implies the commutativity of this square, therefore 
the said equality, from which one derives the statement. 

(c) Let us show the coassociativity of $(H',F_\infty B)$. Recall that for $\xi\in H'$ and $b\in F_\infty B$, one has 
$\xi \cdot b=\langle id \otimes \xi,\Delta_B(b)\rangle$. Let $b \in F_\infty B$ and $\xi,\eta \in H'$, then 
$\xi\cdot(\eta\cdot b)=\langle id \otimes \xi,\Delta_B(\eta\cdot b)\rangle
=\langle id \otimes \xi,\Delta_B(\langle id \otimes \eta,\Delta_B(b)\rangle)\rangle
=\langle id \otimes \xi \otimes \eta,(\Delta_B \otimes id) \circ \Delta_B(b)\rangle
=\langle id \otimes \xi \otimes \eta,(id \otimes \Delta_H) \circ \Delta_B(b)\rangle
=\langle id \otimes (\xi\cdot \eta),\Delta_B(b)\rangle=(\xi\cdot\eta)\cdot b$. 
\end{proof}

\begin{lem}\label{lem:for:local:expansion}
(a) Let $(B,H)$ be a HAMA and let $H'\subset H$ be a Hopf subalgebra. Let us denote by $F_\bullet^HB$, $F_\bullet^{H'}B$ the 
algebra filtrations of $B$ attached to the actions of $H,H'$. Then $F_\infty^HB\subset F_\infty^{H'}B$. 

(b) Let $f : B\to C$ be an algebra morphism; let $H$ be a Hopf algebra right acting on~$B$ and~$C$; assume that $f$ is $H$-equivariant. 
Then $(f,id) : (B,H)\to (C,H)$ is a HAMA morphism, and $f(F_\infty B)\subset F_\infty C$. 
\end{lem}

\begin{proof}
(a) For $n \geq 0$, one has $F_n^HB=\{b \in B\,|\,h\cdot b=0$ for $h\in F^{n+1}H\} \subset \{b \in B\,|\,h\cdot b=0$ for $h\in F^{n+1}H'\}
=F_n^{H'}B$, where the inclusion follows from $F^{n+1}H' \subset F^{n+1}H$. This implies $F_\infty^HB \subset F_\infty^{H'}B$. 

(b) For $n\geq 0$, $b\in F_nB$ and $h\in F^{n+1}H$, one has $f(b)\cdot h=f(b\cdot h)=f(0)=0$ where the first equality follows 
the equivariance of $f$ and the second equality follows from $b\in F_nB$. Therefore $f(b)\in F_nC$. So $f(F_nB)\subset F_nC$, hence 
$f(F_\infty B)\subset F_\infty C$. 
\end{proof}

\subsection{Pairing-morphisms from a HACA to a HAMA and HACA morphisms}\label{sect:B3:0112}

\begin{defn}\label{def:5:9:toto}
Let $(O,A)$ be a HACA and $(B,H)$ be a HAMA. A pairing-morphism from $(O,A)$ to $(B,H)$
is a pair $(p,f)$, where $p : O\otimes H\to\mathbb C$ is a Hopf algebra pairing, and 
$f : A\to B$ is an algebra morphism, such that the following diagram commutes
\begin{equation}\label{DIAG29}
\xymatrix{A\otimes H\ar^{f\otimes id}[rr]\ar_{\Delta_A\otimes id}[d] && B\otimes H\ar[d]\\ O\otimes A\otimes H
\ar_{\sigma_{OA}\otimes id}[r]& A\otimes O\otimes H
\ar_{f\otimes p}[r] &B}
\end{equation}
where $\Delta_A : A\to O\otimes A$ is the coaction morphism of $(O,A)$, $\sigma_{OA} : O\otimes A\to A\otimes O$ is the permutation 
isomorphism, and the right vertical map is the action map of $(B,H)$. Equivalently, one requests the identity 
$$
f(a)\cdot h=f(a^{(2)})p(a^{(1)} \otimes h)
$$
to be satisfied for any $a\in A$ and $h\in H$, where one denotes $\Delta_A(a):=a^{(1)} \otimes a^{(2)}$. 

We denote by $\mathbf{PM}((O,A),(B,H))$ the set of pairing-morphisms from $(O,A)$ to $(B,H)$. 
\end{defn}

\begin{lem}\label{lem:5:10:0407}
Let $(O,A)$ an object in $\mathbf{HACA}$ and $(B,H)$ an object in $\mathbf{HAMA}_{fd}$. 
If $(p,f)\in \mathbf{PM}((O,A),(B,H))$, then $f(F_\bullet A)\subset F_\bullet B$, so that 
$f$ induces an algebra morphism $F_\infty f : F_\infty A\to F_\infty B$. The assignment
$(p,f)\mapsto(\nu(p),F_\infty f)$ defines a map 
$$
\tilde\nu :  \mathbf{PM}((O,A),(B,H))\to \mathbf{HACA}((F_\infty O,F_\infty A),(H',F_\infty B)). 
$$
\end{lem}

\begin{proof}
Let $n\geq 0$ and $a\in F_nA$. Then $\Delta_A(a)\in F_n   O \otimes A$. So $( 
\Delta_O^{(n+1)}\otimes id_A) \circ \Delta_A(a) \in (\sum_{\emptyset \neq L \subset [\![1,n+1]\!]}O_L^{(n+1)}) \otimes A  
\subset A \otimes O^{\otimes n+1}$. Then for $h_1,\ldots,h_{n+1} \in H_+$, one has 
\begin{align*}
    & f(a)\cdot (h_1\cdots h_{n+1})=\langle(id\otimes f)\circ\Delta_A(a), (h_1\cdots h_{n+1})\otimes id\rangle
    \\ & =\langle(id^{\otimes n+1}\otimes f)\circ\Delta_A^{(n+1)}(a), (h_1\otimes \cdots \otimes h_{n+1})\otimes id\rangle=0, 
\end{align*}
where the first equality follows from Def. \ref{def:5:9:toto}, the second equality follows from the Hopf algebra pairing axiom, the last equality 
follows from $\langle O_L^{(n+1)},H_+^{\otimes n+1}\rangle=0$, which follows from $\langle 1,H_+\rangle=0$. 
It follows that $f(a) \in F_nB$. 

Let us prove that $(\nu(p),F_\infty f)$ is a morphism in $\mathbf{HACA}$.
Let $n \geq 0$. Since $f(F_nA)\subset F_nB$, $f$ induces a linear map $F_nf : F_nA\to F_nB$. 

By Lem. \ref{lem:5:3:0307}(b), the coaction map $\Delta_A : A\to O \otimes A$ induces a map $\Delta_{F_nA} : F_nA\to 
F_n   O\otimes F_nA$. By Lem. \ref{1008:jeu}(c), the action map $\mathrm{act}_B : B \otimes H\to B$ induces a map 
$\mathrm{act}_{F_nB} : F_nB\otimes H \to F_nB$.    
Therefore \eqref{DIAG29} induces a commutative diagram 
\begin{equation}
    \xymatrix{F_nA \otimes H\ar_{\Delta_{F_nA} \otimes id}[d]\ar^{F_nf \otimes id}[r]&F_nB \otimes H\ar^{\mathrm{act}_{F_nB}}[d]\\
    F_n   O \otimes F_nA \otimes  H\ar_{(F_nf \otimes p)\circ(\sigma_{OA}\otimes id)}[r]&F_nB}
\end{equation}
By the proof of Lem. \ref{lem:4:11:0307}, the restriction of $p$ to $F_n   O \otimes F^{n+1}H$ is zero, which induces a pairing 
$p_n : F_n   O \otimes (H/F^{n+1}H)\to\mathbb C$. By Lem. \ref{1008:jeu}(c), the restriction of $\mathrm{act}_{F_nB}$ to 
$F_nB \otimes (H/F^{n+1}H)$ is zero, inducing a linear map $\underline{\mathrm{act}}_{F_nB} : F_nB \otimes (H/F^{n+1}H)\to F_nB$. 
The above diagram therefore gives rise to a commutative diagram
$$
\xymatrix{F_nA \otimes (H/F^{n+1}H)\ar^{F_nf \otimes id}[r]\ar_{\Delta_{F_nA} \otimes id}[d]&F_nB \otimes (H/F^{n+1}H)
\ar^{\underline{\mathrm{act}}_{F_nB}}[d]\\
F_n   O \otimes F_nA \otimes (H/F^{n+1}H)\ar_{(F_nf \otimes p_n)\circ(\sigma_{OA}\otimes id)}[r]&F_nB}
$$
Since $H/F^{n+1}H$ is finite dimensional, the map $p_n$ gives rise to a linear map $\nu(p)_n : F_n   O\to F_nH^*$ (see proof of 
Lem. \ref{lem:4:11:0307}) 
and the map $\underline{\mathrm{act}}_{F_nB}$ gives rise to the map $\Delta_{F_nB} : F_nB\to F_nH^*\otimes F_nB$ (see Lem. 
\ref{lem:5:8:0307}(a)). The above commutative diagram then gives rise to a commutative diagram  
$$
\xymatrix{F_nA \ar^{F_nf}[r]\ar_{\Delta_{F_nA}}[d]&F_nB
\ar^{\Delta_{F_nB}}[d]\\
F_n   O \otimes F_nA  \ar_{\nu(p)_n \otimes F_nf}[r]&F_nH^*\otimes F_nB}
$$
This commutativity means that the restriction to $F_nA$ of the two composed maps of the diagram 
$$
\xymatrix{F_\infty A \ar^{F_\infty f}[r]\ar_{\Delta_{F_\infty A}}[d]&F_\infty B
\ar^{\Delta_{F_\infty B}}[d]\\
F_\infty O \otimes F_\infty A \ar_{ \nu(p)\otimes F_\infty f}[r]&H'\otimes F_\infty B}
$$
are equal. Since $F_\infty A=\cup_{n\geq0}F_nA$, this diagram is commutative, therefore the pair $(\nu(p),F_\infty f)$ is a morphism in 
$\mathbf{HACA}$. 
\end{proof}

\subsection{Two functors $\mathbf{HACA}\to\mathbf{Alg}_{gr}$ and a natural transformation} \label{sec:B4:0112}

\begin{defn}
$\mathbf{Alg}_{gr}$ is the category of $\mathbb Z_{\geq0}$-graded associative algebras. 
\end{defn}

By Prop. \ref{prop:LA:filt}(a), a Hopf algebra $O$ is naturally equipped with a filtration $F_\bullet O$. Set 
$\mathrm{gr}(O):=\oplus_{i\geq0}\mathrm{gr}_i(O)$, where $\mathrm{gr}_i(O):=F_i   O/F_{i-1}   O$ for $i\geq 0$, 
then $\mathrm{gr}(O)$ is a $\mathbb Z_{\geq0}$-graded Hopf algebra, therefore also 
a $\mathbb Z_{\geq0}$-graded associative algebra. Recall that an object $(O,A)$ of $\mathbf{HACA}$ gives rise to an algebra filtration 
$F_\bullet A$ on $A$ (see Lem. \ref{lem:5:3:0307}(a)). Then $\mathrm{gr}(A):=\oplus_{i\geq0}\mathrm{gr}_i(A)$, where 
$\mathrm{gr}_i(A):=F_iA/F_{i-1}A$ for $i\geq0$ is a $\mathbb Z_{\geq0}$-graded associative algebra. 

\begin{lem}
(a) The assignment $(O,A)\mapsto\mathrm{gr}(A)$ defines a functor $\mathrm{gr}:\mathbf{HACA}\to\mathbf{Alg}_{gr}$.  

(b) For $(O,A)$ an object in $\mathbf{HACA}$, equip $\mathrm{gr}(O)\otimes F_0A$ with the tensor product $\mathbb Z_{\geq 0}$-graded 
associative algebra structure, where $F_0A$ is concentrated in degree $0$. Then the assignment $(O,A)\mapsto \mathrm{gr}(O)\otimes
F_0A$ defines a functor $\Phi:\mathbf{HACA}\to\mathbf{Alg}_{gr}$.  
\end{lem}

\begin{proof}
(a) The assignment $(O,A)\mapsto(A,F_\bullet A)$ defines a functor $\mathbf{HACA}\to\mathbf{Alg}_{fil}$ where $\mathbf{Alg}_{fil}$ is the 
category of filtered algebras (in the sense of \S\ref{background:filtrations}). The said assignment is the composition of this functor with the 
``associated-graded functor'' 
$\mathbf{Alg}_{fil}\to\mathbf{Alg}_{gr}$. (b) The assignments $(O,A)\mapsto F_0A$ and $(O,A)\mapsto\mathrm{gr}(O)$ define functors 
$\mathbf{HACA}\to \mathbf{Alg}$ and $\mathbf{HACA}\to \mathbf{Alg}_{gr}$, where $\mathbf{Alg}$ is the category of associative  
algebras. The said assignment is the composition of the product of these functors, of the ``degree-0 functor'' $\mathbf{Alg}\to\mathbf{Alg}_{gr}$, 
and of the tensor product functor $\mathbf{Alg}_{gr}^2\to\mathbf{Alg}_{gr}$.  
\end{proof}

Let $(O,A)$ be an object in $\mathbf{HACA}$. By Lem. \ref{lem:5:3:0307}(c), the coaction map $\Delta_A:A\to O\otimes A$ is 
compatible with the filtrations on both 
sides, therefore gives rise to a morphism of graded algebras $\mathrm{gr}_\bullet(\Delta_A):\mathrm{gr}_\bullet(A)\to
\mathrm{gr}_\bullet(O\otimes A)=\mathrm{gr}_\bullet (O)\otimes\mathrm{gr}_\bullet(A)$. There is a unique morphism 
of graded algebras $pr_0:\mathrm{gr}_\bullet(A)\to F_0A$ given by the identity in degree 0 and 0 on all degree components of degree $>0$. 

\begin{defn}
$\mathrm{nat}_{(O,A)}:\mathrm{gr}_\bullet(A)\to \mathrm{gr}_\bullet (O)\otimes F_0A$ is the composition of 
$\mathrm{gr}_\bullet(\Delta_A)$ with $id\otimes pr_0$. 
\end{defn}

\begin{lem}\label{lem:5:14:0407}
(a) The assignment $(O,A)\mapsto \mathrm{nat}_{(O,A)}\in\mathbf{Alg}_{gr}(\mathrm{gr}(A),\Phi(O,A))$ is a natural transformation relating the functors 
$\mathrm{gr},\Phi:\mathbf{HACA}\to\mathbf{Alg}_{gr}$. 

(b) For any object $(O,A)$ of $\mathbf{HACA}$, the morphism $\mathrm{nat}_{(O,A)}$ is injective. 
\end{lem}

\begin{proof}
(a) $\mathrm{nat}_{(O,A)}$ is a morphism of graded algebras as it is a composition of such morphisms. The naturality is obvious. (b) Let us prove 
that for any $n\geq 0$, the degree $n$ component $\mathrm{nat}_{(O,A)}^n$ of $\mathrm{nat}_{(O,A)}$ is injective. The double inclusion of vector spaces of 
$O\otimes A$
$$
F_{n-1}(O\otimes A)\subset \sum_{p=1}^n F_{n-p}   O\otimes F_pA\subset F_n(O\otimes A)
$$
where $F_k(O\otimes A):=\sum_{i=0}^k F_{k-i}   O\otimes F_iA$, 
gives rise to the map  
$$
F_n(O\otimes A)/F_{n-1}(O\otimes A)\to F_n(O\otimes A)/(\sum_{p=1}^n F_{n-p}   O\otimes  F_pA), %\to 0, 
$$
which can be identified with the projection 
$id\otimes pr_0 : \mathrm{gr}_n(O\otimes A)\to \mathrm{gr}_n (O)\otimes F_0A$.  It follows 
that the map $\mathrm{nat}_{(O,A)}^n$ is the vertical cokernel of the diagram 
$$
\xymatrix{F_nA \ar^{\Delta_A}[r]& F_n(O\otimes A)\\ F_{n-1}A\ar[r]\ar@{^{(}->}[u]& \sum_{p=1}^n 
F_{n-p}   O\otimes F_pA\ar@{^{(}->}[u]}
$$
therefore that its kernel is the image in $F_nA/F_{n-1}A$ of the preimage by $\Delta_A : F_nA\to F_n(O\otimes A)$ of the subspace  
$\sum_{p=1}^n F_{n-p}   O\otimes F_pA$ of $F_n(O\otimes A)$. This subspace is contained $F_{n-1}   O\otimes A$, which 
together with Def. \ref{def:FnA} implies that this preimage is contained in $F_{n-1}A$. This implies the vanishing of 
the kernel of $\mathrm{nat}_{(O,A)}^n$ and therefore the injectivity of $\mathrm{nat}_{(O,A)}^n$. 
\end{proof}

\subsection{Isomorphisms in $\mathbf{HACA}$}\label{sect:B5:0112}

\begin{prop}\label{lem:B:15:1508}
(a) A pair $(O,\mathbf a)$ of a Hopf algebra $O$ and an associative algebra $\mathbf a$ gives rise to an object $(O,O\otimes 
\mathbf a)$ of $\mathbf{HACA}$, with coaction morphism given by $\Delta_{O\otimes \mathbf a}:=\Delta_O\otimes id_{\mathbf a}$ 
(where $\Delta_O$ is the coproduct of $O$). 
Its image by the endofunctor of $\mathbf{HACA}$ from Lem. \ref{lem:5:3:0307}(d) is the pair $(F_\infty O,F_\infty O\otimes 
\mathbf a)$, where the coaction morphism is $\Delta_{F_\infty O\otimes \mathbf a}:=\Delta_{F_\infty O}\otimes id_{\mathbf a}$ 
(where $\Delta_{F_\infty O}$ is the coproduct of $F_\infty O$). 

(b) If $(O,\mathbf a)$ is a pair of a Hopf algebra and an associative algebra, if $(B,H)$ is an object of $\mathbf{HAMA}_{fd}$, 
and if $(p,f)\in\mathbf{PM}((O,O\otimes\mathbf a),(B,H))$ is a pairing-morphism such that $\nu(p) : F_\infty O\to H'$ is a Hopf 
algebra isomorphism and the algebra morphism $f : O\otimes\mathbf a\to B$ induces an algebra isomorphism 
$\mathbb C\otimes \mathbf a\stackrel{\sim}{\to} B^H$ between the 
subalgebras $\mathbb C\otimes\mathbf a$ and $B^H$ of both sides, then the morphism 
$(\nu(p),F_\infty f) : (F_\infty O,F_\infty O\otimes\mathbf a)\to (H',F_\infty B)$ in $\mathbf{HACA}$ is an isomorphism; 
in particular, $F_\infty f : F_\infty O\otimes\mathbf a\to F_\infty B$ is an filtered algebra isomorphism.

(c) In the situation of (b), the algebra morphism $f:O\otimes\mathbf a\to B$ is compatible with the filtrations $F _\bullet O\otimes
\mathbf a$ and $F_\bullet B$, and induces an isomorphism $\mathrm{gr} _\bullet(O)\otimes\mathbf a\to\mathrm{gr}_\bullet(B)$ 
in $\mathbf{Alg}_{gr}$.  
\end{prop}

\begin{proof} (a) If $o\in O$ is such that $\Delta_O(o)\in F_n   O \otimes O$, then 
$o=(id\otimes \epsilon_O)\circ \Delta_O(o)\in F_n   O$, where~$\epsilon_O$ is the counit of $O$, therefore 
$\{o\in O\,|\,\Delta_O(o)\in F_n   O\otimes O\}\subset F_n   O$. On the other hand, Lem. \ref{lem:4:2:2106}(c)
implies $F_n   O\subset\{o\in O\,|\,\Delta_O(o)\in F_n   O\otimes O\}$, therefore $F_n   O
=\{o\in O\,|\,\Delta_O(o)=F_n   O\otimes O\}$. Then $F_n(O\otimes \mathbf a)=\{o\in O\,|\,
\Delta_O(o)\in F_n   O\otimes O\}\otimes\mathbf a=F_n   O\otimes \mathbf a$. It follows that 
$F_\infty(O\otimes \mathbf a)=\cup_{n\geq0}F_n   O\otimes \mathbf a=F_\infty O\otimes \mathbf a$. The fact 
that $(F_\infty O,F_\infty O\otimes \mathbf a)\to(O,O\otimes \mathbf a)$ is a morphism in $\mathbf{HACA}$ (see 
Lem. \ref{lem:5:3:0307}(d)) implies that the coaction morphism of $F_\infty O\otimes \mathbf a$ has the claimed value.  

(b) and (c) Lem. \ref{lem:5:10:0407} applied to the pairing-morphism $(p,f)$ implies that $f : O \otimes \mathbf a\to B$ is a morphism of 
filtered algebras. By (a), there is an isomorphism $\mathrm{gr}_\bullet(O\otimes \mathbf a)
=\mathrm{gr}_\bullet (O)\otimes \mathbf a$. Composing it with the associated graded algebra morphism 
$\mathrm{gr}_\bullet(f) : \mathrm{gr}_\bullet(O\otimes \mathbf a)\to\mathrm{gr}_\bullet(B)$, one obtains a morphism 
$$
\mathrm{gr}_\bullet(f) : \mathrm{gr}_\bullet (O)\otimes \mathbf a\to\mathrm{gr}_\bullet(B)
$$
of graded algebras. Since $(B,H)$ is an object of $\mathbf{HAMA}_{fd}$, one may apply to it the functor from Lem. 
\ref{lem:5:8:0307}(c) to obtain the object $(H',F_\infty B)$ of $\mathbf{HACA}$. By Lem. \ref{lem:5:14:0407}, the 
latter object gives rise to an injective morphism of graded algebras 
$$
\mathrm{gr}_\bullet(B)=\mathrm{gr}_\bullet(F_\infty B)\stackrel{\mathrm{nat}_{(H',F_\infty B)}}{\to} 
\Phi(H',F_\infty B)=\mathrm{gr}_\bullet (H')\otimes(F_\infty B)^H
=\mathrm{gr}_\bullet (H')\otimes B^H. 
$$
The composition of these morphisms is the morphism 
$$
\mathrm{nat}_{(H',F_\infty B)}\circ \mathrm{gr}_\bullet(f) : \mathrm{gr}_\bullet (O)\otimes\mathbf a\to 
\mathrm{gr}_\bullet (H')\otimes B^H
$$
of graded algebras. As $\nu(p) : F_\infty O\to H'$ is a Hopf algebra isomorphism, it induces an isomorphism 
between the filtrations $F_\bullet $ on both sides (see Prop. \ref{prop:LA:filt}(c)), and an associated graded isomorphism 
$\mathrm{gr}_\bullet (\nu(p)) : \mathrm{gr}_\bullet (F_\infty O)\to \mathrm{gr}_\bullet (H')$, 
which we identify with its composition with 
the equality $\mathrm{gr}_\bullet (O)=\mathrm{gr}_\bullet (F_\infty O)$. Let us show that  
\begin{equation}\label{key:0407}
\mathrm{nat}_{(H',F_\infty B)}\circ \mathrm{gr}_\bullet(f)=\mathrm{gr}_\bullet (\nu(p))\otimes f_0, 
\end{equation}
(equality of morphisms of graded algebras $\mathrm{gr}_\bullet (O)\otimes\mathbf a\to 
\mathrm{gr}_\bullet (H')\otimes B^H$), 
where $f_0$ is the isomorphism $f_0 : \mathbb C\otimes \mathbf a\to B^H$. The composition 
$\mathbb C \otimes\mathbf a=F_0(O\otimes \mathbf a)\stackrel{\mathrm{gr}_0(f)}{\to} F_0B=B^H$ is $f_0$, which together with 
$\mathrm{gr}_0 (\nu(p))=1$ proves the degree 0 part of \eqref{key:0407}. Let $n\geq 0$ and $o\in \mathrm{gr}_n (O)$.
The image of $o$ under $\mathrm{nat}_{(H',F_\infty B)}\circ \mathrm{gr}_n(f)$ is equal to the image of $\tilde o\otimes1$
by the horizontal composition  
$$
\xymatrix{
&&F_n  H' \otimes B\ar[r]& \mathrm{gr}_n (H') \otimes B\\
F_n   O \otimes \mathbf a\ar_{F_nf}[r]&
F_nB\ar_{\!\!\!\!\!\!\!\!\!\!\!\!\!\!\!\!\!\!\!\!\!\!\!\!\!\!\!\Delta_{F_nB}}[r]&
\sum_{k=0}^n F_{n-k}  H' \otimes F_kB\ar[r]\ar@{^{(}->}[u]& 
\mathrm{gr}_n (H')\otimes F_0B\ar@{^{(}->}[u]
}
$$
where $\tilde o\in F_n   O$ is a lift of $o$. Since $F_nf$ is compatible with the coaction maps, the image of this element in 
$\sum_{k=0}^n  F_{n-k}  H' \otimes F_kB$ is equal to $\nu(p)(\tilde o^{(1)}) \otimes f(\tilde o^{(2)} \otimes 1)$, where
$\tilde o^{(1)}\otimes \tilde o^{(2)}=\Delta_O(\tilde o)\in \sum_{k=0}^n F_{n-k}   O\otimes F_k   O\subset O^{\otimes2}$. 
One has $\tilde o^{(1)}\otimes \tilde o^{(2)}\in \tilde o\otimes 1+F_{n-1}   O \otimes O$, therefore 
$\nu(p)(\tilde o^{(1)}) \otimes f(\tilde o^{(2)} \otimes 1)\in \nu(p)(\tilde o)\otimes 1+F_{n-1}  H' \otimes B$ (inclusion in 
$F_n  H' \otimes B$). It follows that the image of $\tilde o \otimes 1$ in $\mathrm{gr}_n (H')\otimes B$ is 
$\mathrm{im}(\nu(p)(\tilde o)\in F_n  H'\to\mathrm{gr}_n  H')\otimes 1=\mathrm{gr}_n (\nu(p))(o)\otimes1$. 
This implies that both sides of \eqref{key:0407} coincide when restricted to $\mathrm{gr}_n (O)\otimes\mathbb C$
for any $n$, therefore to $\mathrm{gr}_\bullet (O)\otimes\mathbb C$. \eqref{key:0407} then follows from the fact that 
each of its sides is an algebra morphism, and that they agree of $\mathrm{gr}_\bullet (O)\otimes\mathbb C$ and 
$\mathbb C \otimes \mathbf a$, which generate the source algebra. 

It follows from \eqref{key:0407} that $\mathrm{nat}_{(H',F_\infty B)}\circ \mathrm{gr}_\bullet(f)$ is an isomorphism of graded 
algebras. The injectivity of $\mathrm{nat}_{(H',F_\infty B)}$ then implies that both $\mathrm{nat}_{(H',F_\infty B)}$ and $\mathrm{gr}_\bullet(f)$ 
are isomorphisms of graded algebras. By Lem. \ref{lem:5:10:0407}, $f$ induces a morphism of filtered algebras 
$F_\infty f : F_\infty O\otimes\mathbf a \to F_\infty B$, and $\mathrm{gr}_\bullet(f)=\mathrm{gr}_\bullet(F_\infty f)$. It follows that 
$\mathrm{gr}_\bullet(F_\infty f)$ is an isomorphism of graded algebras, which together with the fact that the filtrations of the source and 
target of $F_\infty f$ are exhaustive, implies that $F_\infty f$ is an isomorphism of filtered algebras (see Lem. \ref{graded:crit:iso}). 
\end{proof}

\section{Filtered formality for Hopf algebras and HACAs}\label{sect:C:0112}

In §\ref{sect:C1:0112}, we introduce a notion of filtered formality for Hopf algebras; this notion is related in §\ref{sect:D:0112} 
to the similar notion for discrete groups, introduced in \cite{SW}. We extend this to a notion of filtered formality for HACAs in 
§\ref{B'2:3110}. The main result of this section is Prop. \ref{prop:10:14:3110}, which shows that a HACA constructed in the context of 
Prop. \ref{lem:B:15:1508} is filtered formal. 

\subsection{Filtered formality for Hopf algebras}\label{sect:C1:0112}

In Prop. \ref{prop:LA:filt}(a), we attach to a Hopf algebra~$O$ a filtration $F_\bullet  O$. 
Let $\mathrm{gr}(O)$ be the associated graded vector space. 

\begin{lem}\label{lem:10:1:a:2909}
 (a) If $O$ is a Hopf algebra, then $\mathrm{gr}(O)$ is a graded Hopf algebra, which is commutative if $O$ is, and such that 
 for each $n\geq0$, $F _n(\mathrm{gr} (O))=\mathrm{gr} _{\leq n}(O)$. 
 
 (b) The assignment $O\mapsto \mathrm{gr}(O)$ defines an endofunctor of the category 
$\mathbf{HA}_{comm}$ of commutative Hopf algebras. 
\end{lem}

\begin{proof} Let us show (a). 
The first statement follows from the fact that $F_\bullet  O$ is a Hopf algebra filtration (see Prop. \ref{prop:LA:filt}(a)). 
Let us show the second statement. Let $n\geq0$. The inclusion $F_n (\mathrm{gr}(O))\supset 
\mathrm{gr} _{\leq n} (O)$ follows from the fact that $\mathrm{gr}(O)$ is a graded and connected Hopf algebra. 
Let us show the opposite inclusion. Since $\mathrm{gr}(O)$ is a graded Hopf algebra, $F_n (\mathrm{gr}(O))$
is the direct sum of its intersection with the homogeneous components of $\mathrm{gr}(O)$. 
If $k\geq1$, the intersection $F_n (\mathrm{gr}(O))\cap\mathrm{gr} _{n+k}(O)$ is contained in $F_n (\mathrm{gr}(O))$, which is the kernel of 
$(id-\eta_{\mathrm{gr}(O)}\epsilon_{\mathrm{gr}(O)})^{\otimes n+1}\circ\Delta_{\mathrm{gr}(O)}^{(n+1)} : 
\mathrm{gr}(O)\to(\mathrm{gr}(O))^{\otimes n+1}=\mathrm{gr} (O^{\otimes n+1})$, therefore
$F_n (\mathrm{gr}(O))\cap\mathrm{gr} _{n+k}(O)$ is contained in the kernel of the map 
$\mathrm{gr}_{n+k}(O)\to\mathrm{gr}_{n+k} (O^{\otimes n+1})$ induced by 
$(id-\eta_{\mathrm{gr}(O)}\epsilon_{\mathrm{gr}(O)})^{\otimes n+1}\circ\Delta_{\mathrm{gr}(O)}^{(n+1)}$. 
This map is the degree $n+k$ part of the associated graded of the map 
$(id-\eta_O\epsilon_O)^{\otimes n+1}\circ\Delta_O^{(n+1)} : O\to O^{\otimes n+1}$. 
It follows that if $V$ is the preimage of $F_n (\mathrm{gr}(O))\cap\mathrm{gr} _{n+k}(O)$
in $F_{n+k}  O$, one has $(id-\eta_O\epsilon_O)^{\otimes n+1}\circ\Delta_O^{(n+1)}(V)\subset 
F_{n+k-1} (O^{\otimes n+1})$. The map $((id-\eta_O\epsilon_O)^{\otimes k}\circ\Delta_O^{(k)})\otimes id^{\otimes n}$
maps $F_{n+k-1} (O^{\otimes n+1})$ to $F_{n+k-1} (O^{\otimes n+k})$, therefore 
$(id-\eta_O\epsilon_O)^{\otimes n+k}\circ\Delta_O^{(n+k)}(V)\subset 
F_{n+k-1} (O^{\otimes n+k})$. Now $O=\mathbb C1\oplus F_1  O$; one has 
$(id-\eta_O\epsilon_O)^{\otimes n+k}\circ\Delta_O^{(n+k)}(V)\subset (F_1  O)^{\otimes n+k}$
while $F_{n+k-1} (O^{\otimes n+k})\subset \sum_{i=1}^{n+k}O^{\otimes i-1}\otimes\mathbb C1\otimes 
O^{\otimes n+k-i}$. As the second terms of both inclusions have zero intersection in $O^{\otimes n+k}$, one has 
$(id-\eta_O\epsilon_O)^{\otimes n+k}\circ\Delta_O^{(n+k)}(V)=0$. Therefore $V\subset F_{n+k-1}  O$, which 
implies $F_n (\mathrm{gr}(O))\cap\mathrm{gr} _{n+k}(O)=0$.  Therefore 
$F_n (\mathrm{gr}(O))=\mathrm{gr} _{\leq n}(O)$. 

(b) follows from the naturality of the assignment $O\mapsto F_\bullet  O$. 
\end{proof}

\begin{defn}\label{def:FF:HA:3110}
The Hopf algebra $\mathcal O$ is called {\it filtered formal} if there is an isomorphism of Hopf algebra 
$O\to\mathrm{gr} (O)$ whose associated graded for the grading $\mathrm{gr} $ is the identity. 
\end{defn}
This terminology is justified by Prop. \ref{prop:C3:3110}.

\subsection{Filtered formality for HACAs}\label{B'2:3110}

Let $\mathbf{HACA}_{gr}$ be the category of $\mathbb Z_{\geq0}$-graded HACAs.  

\begin{lem}\label{lem:10:12:2410}
The assignment $(O,A)\mapsto(\mathrm{gr}(O),\mathrm{gr}(A))$ is a functor 
$(\mathrm{gr} ,\mathrm{gr}) : \mathbf{HACA}\to\mathbf{HACA}_{gr}$.  
\end{lem}

\begin{proof} Follows from the fact that for $(O,A)$ a HACA, $(F_\bullet  O,F_\bullet A)$ is a HACA filtration. 
\end{proof}

\begin{defn}\label{def:C4:0711}
The HACA $(O,A)$ is called {\it filtered formal} if there exists a isomorphism $(O,A)\to(\mathrm{gr}(O),\mathrm{gr}(A))$
in $\mathbf{HACA}$, which is compatible with the filtrations $(F_\bullet  O,F_\bullet A)$ in the source and induced by the 
grading in the target, and whose image by the functor $(\mathrm{gr} ,\mathrm{gr}) : \mathbf{HACA}\to\mathbf{HACA}_{gr}$ is the 
identity endomorphism of $(\mathrm{gr}(O),\mathrm{gr}(A))$. 
\end{defn}

If the HACA $(O,A)$ is filtered formal, then the Hopf algebra $O$ is filtered formal in the sense of Def. \ref{def:FF:HA:3110}. 

\begin{prop}\label{prop:10:14:3110}
Let $(O,\mathbf a),(B,H),(p,f)$ be as in the hypothesis of Prop. \ref{lem:B:15:1508}(b),(c). Then the HACA $(H',F_\infty B)$
is filtered formal. 
\end{prop}

\begin{proof} Recall from Lem. \ref{lem:10:12:2410} that $(\tilde O,A)\mapsto(\mathrm{gr}(\tilde O),\mathrm{gr}(A))$ is a 
functor $\mathbf{HACA}\to\mathbf{HACA}_{gr}$; one checks that the same is true of $(\tilde O,A)\mapsto(\mathrm{gr}(\tilde O),
\Phi(\tilde O,A))$. For $(\tilde O,A)$ a HACA, $(id_{\tilde O},\mathrm{nat}_{(\tilde O,A)})$ is a morphism $(\mathrm{gr} 
(\tilde O),\mathrm{gr}(A))\to(\mathrm{gr}(\tilde O),\Phi(\tilde O,A))$ in $\mathbf{HACA}_{gr}$, which is an isomorphism if and only if 
$\mathrm{nat}_{(\tilde O,A)}$ is. 

The proof of Prop. \ref{lem:B:15:1508}(c) implies that $\mathrm{nat}_{(H',F_\infty B)}$ is 
an isomorphism in $\mathbf{Alg}_{gr}$, therefore
\begin{equation}\label{(*):0610:GAL}
(id_{\mathrm{gr}(H')},\mathrm{nat}_{(H',F_\infty B)}) 
: (\mathrm{gr}(H'),\mathrm{gr}(B))
\to (\mathrm{gr}(H'),\mathrm{gr}(H') \otimes B^H) 
\end{equation}
is an isomorphism in $\mathbf{HACA}_{gr}$. 

By the assumptions of Prop. \ref{lem:B:15:1508}(b), $\nu(p) : F_\infty O\to H'$ is an isomorphism in $\mathbf{HA}$. 
Prop. \ref{lem:10:1:a:2909} then implies that $\mathrm{gr}(\nu(p)) : \mathrm{gr} (O)\to \mathrm{gr} (H')$
is an isomorphism in $\mathbf{HA}_{gr}$. The assumptions of Prop. \ref{lem:B:15:1508}(b) also imply that 
$f$ restricts to an algebra isomorphism $\mathbb C\otimes \mathbf a\to B^H$, we denote by $f_0$ the composed isomorphism 
$\mathbf a\simeq\mathbb C\otimes \mathbf a\to B^H$. Both isomorphisms $\mathrm{gr}(\nu(p))$ and $f_0$ induce an isomorphism 
\begin{equation}\label{(**):0610:GAL}
    (\mathrm{gr}(\nu(p))^{-1},\mathrm{gr}(\nu(p))^{-1}\otimes f_0^{-1}) : 
    (\mathrm{gr} (H'),\mathrm{gr} (H')\otimes B^H)\to(\mathrm{gr} (O),\mathrm{gr} (O)\otimes\mathbf a). 
\end{equation}
in $\mathbf{HACA}_{gr}$. 
By Prop. \ref{lem:B:15:1508}(b) and (c), the pair 
\begin{equation}\label{(***):0610:GAL}
(\nu(p),F_\infty f) : (F_\infty O,F_\infty O\otimes\mathbf a)\to(H',F_\infty B)
\end{equation}
is an isomorphism in $\mathbf{HACA}$. The composition of \eqref{(*):0610:GAL}, \eqref{(**):0610:GAL} and \eqref{(***):0610:GAL}
gives rise to an isomorphism 
\begin{align}\label{(****):0610:GAL}
%& \nonumber
\big(\nu(p) \circ (\mathrm{gr}(\nu(p))^{-1}),F_\infty f \circ (\mathrm{gr} (\nu(p))^{-1} 
\otimes f_0^{-1}) \circ \mathrm{nat}_{(\mathrm{gr}(H'),F_\infty B)}\big) :  
%\\ & 
(\mathrm{gr}(H'),\mathrm{gr}(B))\to (H',F_\infty B) 
\end{align}
in $\mathbf{HACA}$. 

By Prop. \ref{prop:LA:filt}(c), $\nu(p)$ is compatible with the filtrations $F_\bullet $; it follows that 
$\nu(p) \circ \mathrm{gr} (\nu(p))^{-1}$ is compatible with these filtrations as well, and one computes 
$\mathrm{gr} (\nu(p)\circ \mathrm{gr} (\nu(p))^{-1})=id$. 

The algebra morphisms in \eqref{(*):0610:GAL} and \eqref{(**):0610:GAL} are compatible with the filtrations since they are graded, 
and the algebra morphism in \eqref{(***):0610:GAL} is compatible with the filtrations. It follows that the algebra morphism in 
\eqref{(****):0610:GAL} is compatible with the filtrations. Its associated graded is the composed morphism 
\begin{equation}\label{MAP:2010:GAL}
\mathrm{gr}(B)
\stackrel{\mathrm{nat}_{(H',F_\infty B)}}{\to}
\mathrm{gr}(H') \otimes B^H
\stackrel{\mathrm{gr} (\nu(p))^{-1}\otimes f_0^{-1}}{\to}
F_\infty O \otimes\mathbf a
\stackrel{\mathrm{gr}f}{\to}
\mathrm{gr}(B). 
\end{equation}
It follows from the proof of Prop. \ref{lem:B:15:1508}(b) that relation \eqref{key:0407} holds, and that $\mathrm{nat}_{(H',F_\infty B)}$ and 
$\mathrm{gr}f$ are both isomorphisms. The combination of these facts then implies that the map \eqref{MAP:2010:GAL} is the identity. 
Therefore the algebra morphism in \eqref{(****):0610:GAL} is compatible with the filtrations, and its associated graded is the identity. 
\end{proof}

\section{Hopf algebra duality and prounipotent completions}\label{sect:D:0112}

§\ref{app:C:1:2308} relates the Hopf algebraic constructions of §\ref{sect:A:0112} to complete Hopf algebras (CHAs) in the sense of \cite{Q}. 
This is applied in §\ref{sect:D:2:0711} to obtain an expression of the function algebra of the prounipotent completion of a finitely generated 
group in the terms of the functors of §\ref{sect:A:0112}. The main result of §\ref{sect:D3:0112} is Prop. \ref{prop:C3:3110}, which 
relates the notion of filtered formality for a finitely generated group (\cite{SW}) with the similar notion for Hopf algebras, introduced 
in §\ref{sect:C:0112}. 

\subsection{Completion and duality of Hopf algebras}\label{app:C:1:2308}

Recall from \cite{Q}, A1 and A2 the categories $\mathbf{CAA}$ of complete augmented algebras (CAAs) and $\mathbf{CHA}$ of complete Hopf 
algebra (CHAs): $\mathbf{CAA}$ is the full subcategory, in the category of pairs 
$(A,F^\bullet A)$ of an augmented algebra~$A$ and a decreasing algebra filtration $A=F^0 A\supset F^1A\supset\ldots$, of CAAs, i.e. 
of pairs such that $A$ is complete and Hausdorff for the topology of $F^\bullet A$, such that $F^1A$ is the augmentation ideal of $A$ and 
such that 
$\mathrm{gr}(A)$ is generated by $\mathrm{gr}^1(A)$; a monoidal structure is defined on $\mathbf{CAA}$, given at the level of objects  
by $((A,F^\bullet A),(B,F^\bullet B))\mapsto (A\hat\otimes B,F^\bullet(A\hat\otimes B))$, where $A\hat\otimes B:=\varprojlim_n 
(A/F^nA)\otimes (B/F^nB)$ and $F^n(A\hat\otimes B)=\varprojlim_m \mathrm{im}(\sum_{n'+n''=n}
F^{n'}A\otimes F^{n''}B\to (A/F^mA)\otimes (B/F^mB))$;  
a CHA is a pair $(A,\Delta_A)$, where $A$ is a CAA, 
$\Delta_A : A\to A \hat\otimes A$ is a cocommutative and coassociative algebra morphism, which admits the augmentation of $A$ 
as a counit, and a morphism in $\mathbf{CHA}$ is a morphism in $\mathbf{CAA}$ which is compatible with coproducts. 

\begin{lem}\label{lem:HA:CHA}
(a)  Let $H$ be a cocommutative Hopf algebra with coproduct $\Delta_H$ and counit $\epsilon_H$. Recall $H_+=\mathrm{Ker}(\epsilon_H)$. 
 Set $H^\wedge:=\varprojlim_m H/H_+^m$ and for $n\geq 0$, $F^nH^\wedge:=\varprojlim_m F^nH/F^{\mathrm{max}(n,m)}H$. Then $H^\wedge$ 
 is a complete augmented algebra. There is a unique continuous extension of $\Delta_H$ to a map $\hat\Delta_H : H^\wedge\to H^\wedge 
 \hat\otimes H^\wedge$, which equips $H^\wedge$ with the structure of a complete Hopf algebra. The assignment $H\mapsto H^\wedge$ is a 
 functor $\mathbf{HA}_{coco}\to\mathbf{CHA}$, where $\mathbf{HA}_{coco}$ is the category of cocommutative Hopf algebras. 

(b) If $V$ is a vector space and $H$ is the cocommutative Hopf algebra $T(V)$ with coproduct defined by the condition that the elements of 
$V$ are primitive, then $T(V)^\wedge=\hat T(V)$ (the degree completion of $T(V)$).  
\end{lem}

\begin{proof}
(a) The first statement follows from \cite{Q}, A1, Example 1.2. One has $\Delta_H(H_+)\subset H_+\otimes H+H\otimes H_+$, 
which implies that for any $n\geq 0$, one has $\Delta_H(H_+^n)\subset\sum_{p,q|p+q=n} H_+^p\otimes H_+^q$. It follows that 
$\Delta_H$ defines a collection of maps $H/H_+^{2n}\to (H/H_+^n)^{\otimes 2}$ indexed by $n\geq 0$, which are compatible for various 
$n$ and which therefore induce a map 
$\hat\Delta_H : H^\wedge\to H^\wedge\hat\otimes H^\wedge$, which is a continuous extension of $\Delta_H$. The uniqueness of this 
extension follows from the Hausdorff property of $H^\wedge$. The verification of the other properties of $\hat\Delta_H$ is standard.  
(b) follows from the fact that $T(V)_+^n$ is the part of $T(V)$ of degree $\geq n$, and that $\hat T(V)$ is the 
completion w.r.t. the corresponding topology. 
\end{proof}

\begin{defn}\label{def:c2:2509}
(a) $\mathbf{CHA}_{fd}$ is the full subcategory of $\mathbf{CHA}$ of CHAs $A$ 
such that $\mathrm{gr}^1A$ is finite dimensional. 

(b) For $H$ a Hopf algebra, set $\mathrm{gr}^1(H):=H_+/H_+^2$; where we recall $H_+=\mathrm{Ker}(\epsilon_H)$, and $\epsilon_H$
is the counit of $H$. 

(c) $\mathbf{HA}_{coco,fd}$ is the full subcategory of $\mathbf{HA}_{coco}$ of all cocommutative Hopf algebras $H$ 
such that $\mathrm{gr}^1H$ is finite dimensional. 

(d) $\mathbf{Gp}_{fg}$ is the full subcategory of $\mathbf{Gp}$ of finitely generated groups. 
\end{defn}

\begin{lem}\label{lem:1954:1107}
(a) The functor $\mathbf{Gp}\to\mathbf{HA}_{coco}$ given by $\Gamma\mapsto \mathbb C\Gamma$ induces a functor 
$\mathbf{Gp}_{fg}\to\mathbf{HA}_{coco,fd}$. 

(b) The functor $\mathbf{HA}_{coco}\to\mathbf{CHA}$, $H\mapsto H^\wedge$ from Lem. \ref{lem:HA:CHA} induces a functor 
$\mathbf{HA}_{coco,fd}\to\mathbf{CHA}_{fd}$. 
\end{lem}

\begin{proof}
(a) If $\Gamma$ is a group, then $\mathrm{gr}^1(\mathbb C\Gamma)=\Gamma^{ab}\otimes\mathbb C$ (see \cite{Q2}). 
Moreover, if $\Gamma$ is finitely generated, 
then $\Gamma^{ab}$ is a finitely generated abelian group, which together with the above equality implies the finite dimensionality of 
$\mathrm{gr}^1(\mathbb C\Gamma)$. (b) follows from the fact that the natural map $\mathrm{gr}^1H\to\mathrm{gr}^1(H^\wedge)$ is a linear
isomorphism. 
\end{proof}

\begin{defn}\label{def:dual:CHAfd}
For $H$ a CHA with associated filtration $H=F^0H\supset\cdots$, set $H':=\cup_{n\geq 0}F^nH^\perp \subset H^*$. 
\end{defn}

\begin{lem}\label{2237:1107}
(a) If $H$ is a CHA, then $H'$ has an commutative and associative algebra structure. 

(b) If $H$ is an object in $\mathbf{CHA}_{fd}$, then the algebra structure on $H'$ can be upgraded to a commutative Hopf algebra structure; 
the resulting assignment $H\mapsto H'$ is a contravariant functor $\mathbf{CHA}_{fd}\to \mathbf{HA}_{comm}$. 
\end{lem} 

\begin{proof}
(a) Let $\xi,\eta\in H'$. Let $n\geq 0$ be such that $\xi,\eta\in F^nH^\perp$, then $\xi\otimes\eta$ is a linear form on $(H/F^nH)^{\otimes 2}$, 
which can be pulled back by the map $H^{\hat\otimes2}\to(H/F^nH)^{\otimes 2}$ to define a linear form on $H^{\hat\otimes 2}$; it can be checked 
to be independent on the choice of $n$ and will be denoted $\ell_{\xi,\eta}$. Then $\ell_{\xi,\eta}$ vanishes on $F^nH\otimes H+H\otimes F^nH$. 
The assignment $\xi\cdot \eta : h\mapsto \ell_{\xi,\eta}\circ\Delta_H(h)$ is then a linear form on $H$, i.e. an element of $H^*$. Since 
$\Delta_H(F^{2n}H)\subset 
F^nH \otimes H+H \otimes F^nH$, one has $\xi\cdot\eta \in F^{2n}H^\perp$, therefore $\xi\cdot\eta \in H'$. One checks that this defines 
a commutative and associative algebra structure on $H'$. 

(b) If $n\geq 0$, then $H/F^nH$ is noncanonically isomorphic to $\oplus_{k=0}^n\mathrm{gr}^k(H)$, which is finite dimensional since it is 
generated by $\mathrm{gr}^1H$ and the latter space is finite dimensional. The coproduct $\Delta_H$ defines a coproduct $\Delta_{F^nH} : H/F^nH\to (H/F^nH)^{\otimes2}$ for any $n\geq 0$. Since $H/F^nH$ is finite dimensional, this coproduct can be dualized and defines a product 
$m_{(F^nH)^\perp} : ((F^nH)^\perp)^{\otimes2}\to(F^nH)^\perp$. As in the proof of Lem. \ref{lem:4:8:jeu}, one checks that the collection of these 
products is compatible, which gives rise to a product  $m_{H'} : H'^{\otimes 2}\to H'$, which can be shown to define a Hopf structure. 
\end{proof}

The following lemma is an topological analogue of Lem. \ref{IDFILT}. 
\begin{lem}\label{IDFILT:CHA}
If $H$ is a CHA, then for any $n\geq0$ one has $F_n (H')=(F^{n+1}H)^\perp$ (equality of subspaces of $H^*$). 
\end{lem}

\begin{proof}
Let us show that if $H$ is a CHA and $k,n\geq 1$, then the image of the product map $(F^1H/F^{n+k}H)^{\otimes n}\to H/F^{n+k}H$ is 
$F^nH/F^{n+k}H$. We proceed by induction on $k$. For $k=1$, the statement follows from the surjectivity of the product map 
$(F^1H/F^2H)^{\otimes n}\to F^nH/F^{n+1}H$. Assume the statement at order $k$ and let us show it at order $k+1$. 
The product map $(F^1H/F^2H)^{\otimes n+k}\to F^{n+k}H/F^{n+k+1}H$ is surjective, therefore $F^{n+k}H/F^{n+k+1}H$ is equal to the image of 
a subspace of $(F^1H/F^{n+k+1}H)^{\otimes n}$, namely 
$$
(F^1H/F^{n+k+1}H)^{\otimes n-1}\otimes \mathrm{im}((F^1H/F^{n+k+1}H)^{\otimes k+1}\to F^1H/F^{n+k+1}H). 
$$
Combining this with the statement at order $k$, one obtains the statement at order $k+1$ and therefore the induction. 
The proof of the lemma is then similar to that of Lem.~\ref{IDFILT}. 
\end{proof}

\begin{lem}\label{lem:5:12:1927}
The functors $\mathbf{HA}_{coco,fd}\to\mathbf{CHA}_{fd}\to\mathbf{HA}_{comm}$ (see Lem. \ref{lem:1954:1107}(b) and Lem. \ref{2237:1107}(b)) 
and $\mathbf{HA}_{coco,fd}\to\mathbf{HA}_{comm}$ (see Lem. \ref{lem:4:8:jeu}) are naturally equivalent. 
\end{lem}

\begin{proof}
For each $n\geq 0$, the algebra morphism $H\to H^\wedge$ induces an algebra isomorphism $H/F^nH\stackrel{\sim}{\to}H^\wedge/F^nH^\wedge$, 
which fits in a commutative diagram 
$$
\xymatrix{H\ar[r]\ar[d]&H^\wedge\ar[d]\\H/F^nH\ar[r]&H^\wedge/F^nH^\wedge}
$$
Dualizing and using the equalities $(H/F^nH)^*=F^nH^\perp$, $(H^\wedge/F^nH^\wedge)^*=(F^nH^\wedge)^\perp$, one obtains a commutative diagram
\begin{equation}\label{*1207}
\xymatrix{(F^nH^\wedge)^\perp\ar^\sim[r]\ar@{^{(}->}[d]&(F^nH)^\perp\ar@{^{(}->}[d]\\
(H^\wedge)^*\ar[r]&H^*}
\end{equation}
If $V,W$ are filtered vector spaces and $f : V\to W$ is a linear map such that for any $n\geq 0$, $f$ induces a linear isomorphism 
$F_nV\to F_nW$, then $f$ induces an isomorphism $F_\infty V\to F_\infty W$. This and \eqref{*1207} imply that the map $(H^\wedge)^*\to H^*$  
restricts to a linear isomorphism $(H^\wedge)'\to H'$. One then checks this isomorphism to be compatible with the Hopf algebra structures. 
\end{proof}

\subsection{Isomorphism $\mathbb C\Gamma'\simeq\mathcal O(\Gamma_{unip})$}\label{sect:D:2:0711}

Let $\Gamma$ be a group. A pro-unipotent completion of a group $\Gamma$ is the pair $(U,c)$ 
of a prounipotent $\mathbb C$-group scheme $U$ and a morphism $\Gamma\to U(\mathbb C)$ 
satisfying a universal property (see \cite{BGF}, Def. 3.217). 

If $\Gamma$ is an object in $\mathbf{Gp}_{fg}$, then $\mathbb C\Gamma$ is an object in $\mathbf{HA}_{fd}$ by Lem. \ref{lem:1954:1107}(a), 
therefore $(\mathbb C\Gamma)'$ is a commutative Hopf algebra (see Lem. \ref{lem:4:8:jeu}). Since 
$\Gamma^{\mathrm{ab}}\otimes\mathbb C$ is finite dimensional, one may use 
\cite{BGF}, Thm. 3.224, to obtain that a pro-unipotent completion exists and is unique up to isomorphism, and can be constructed
as the pair $(\Gamma_{unip},c_{unip})$, where $\Gamma_{unip}$ is the spectrum of the commutative 
Hopf algebra $\mathcal O(\Gamma_{unip}):=((\mathbb C\Gamma)^\wedge)'$, which is isomorphic to $(\mathbb C\Gamma)'$ by Lem. 
\ref{lem:5:12:1927},  and $c_{unip}$ is induced by the commutative 
algebra morphism $\mathcal O(\Gamma_{unip})=(\mathbb C\Gamma)'\hookrightarrow (\mathbb C\Gamma)^*=\mathbb C^\Gamma$, where 
$\mathbb C^\Gamma$ is the algebra of all functions $\Gamma\to\mathbb C$.  

\subsection{Relation between the filtered formalities of $\Gamma$ and $(\mathbb C\Gamma)'$}\label{sect:D3:0112}

\begin{defn}[see \cite{Q}, \S A.2] (a) A Malcev Lie algebra (MLA) is a Lie algebra $\mathfrak g$, equipped with a decreasing 
Lie algebra filtration $\mathfrak g=F^1\mathfrak g\supset F^2\mathfrak g\supset\cdots$ (i.e., $[F^{i}\mathfrak g,F^{j}\mathfrak g]
\subset F^{i+j}\mathfrak g$ for $i,j\geq 1$) for which it is complete and Hausdorff, and such that the associated graded Lie algebra
$\mathrm{gr}\mathfrak g$ is generated by $\mathrm{gr}^1\mathfrak g$. 

(b) A morphism between two MLAs is a Lie algebra morphism which is compatible with the filtrations. We denote 
by $\mathbf{MLA}$ the category of MLAs. 
\end{defn}

\begin{lem}[see \cite{Q}, Thm. 3.3] \label{lem:Q:2909}
The functors $\hat U : \mathbf{MLA}\leftrightarrow \mathbf{CHA} : \mathcal P$, where $\hat U(\mathfrak g):=\varprojlim_i U(\mathfrak g/F^i\mathfrak g)$ and $\mathcal P$ is the 
``primitive-elements functor'', taking $(A,\Delta_A)$ to $\mathcal P(A):=\{a\in A\,|\,\Delta_A(a)=a\otimes1+1\otimes a\}$, equipped with the
filtration given by $F_n\mathcal P(A) := \mathcal P(A) \cap  F_nA$ for $n \geq 0$, are quasi-inverse to one another. 
\end{lem}

\begin{defn} \label{def:Lie:2909}
The composed functor $\mathbf{Gp}\to\mathbf{CHA}\stackrel{\mathcal P}{\to}\mathbf{MLA}$, where the first functor is $\Gamma\mapsto (\mathbb C\Gamma)^\wedge$ (see Lem. \ref{lem:HA:CHA}), is denoted $\Gamma\mapsto\mathrm{Lie}(\Gamma)$. 
\end{defn}

\begin{lemdef}
The assignment $\mathfrak g\mapsto (\hat{\mathrm{gr}}(\mathfrak g),F^\bullet\hat{\mathrm{gr}}(\mathfrak g))$, where 
$\hat{\mathrm{gr}}(\mathfrak g):=\prod_{i\geq1}\mathrm{gr}^i(\mathfrak g)$ and for $n\geq1$, 
$F^n\hat{\mathrm{gr}}(\mathfrak g):=\prod_{i\geq n}\mathrm{gr}^i(\mathfrak g)$, is an endofunctor of $\mathbf{MLA}$. 
For $\mathfrak g$ an object in $\mathbf{MLA}$, one has $\mathrm{gr}(\hat{\mathrm{gr}}(\mathfrak g))\simeq\mathrm{gr}(\mathfrak g)$.   
\end{lemdef}

\begin{proof}
Immediate. 
\end{proof}

\begin{defn}[see \cite{SW}] A group $\Gamma$ is called filtered formal if there exists an isomorphism 
$\mathrm{Lie}(\Gamma)\to \hat{\mathrm{gr}}\mathrm{Lie}(\Gamma)$ in $\mathbf{MLA}$, whose associated graded is the identity. 
\end{defn}

One checks that $A\mapsto\hat{\mathrm{gr}}(A):=\prod_{i\geq 0}F^iA/F^{i+1}A$ is an endofunctor of $\mathbf{CHA}$.
\begin{lem}\label{*:2509}
The category equivalence $\mathbf{CHA}\leftrightarrow\mathbf{MLA}$
quasi-intertwines the endofunctors $\hat{\mathrm{gr}}$ on both sides.
\end{lem}

\begin{proof}
Follows from the natural isomorphism of graded Lie algebras
\begin{equation}\label{eq:quillen:2509}
\mathrm{gr}(\mathcal PA) \simeq \mathcal P(\mathrm{gr}(A))
\end{equation}
for $A$ on object of $\mathbf{CHA}$ from \cite{Q}, A1, Thm. 2.14.   
\end{proof}

\begin{defn}
$\mathbf{MLA}_{fd}$ is the full subcategory of $\mathbf{MLA}$ of MLAs $\mathfrak g$ such that $\mathrm{gr}^1\mathfrak g$ is finite dimensional.
\end{defn}

Recall the full subcategory $\mathbf{CHA}_{fd}$ of $\mathbf{CHA}$ (see Def. \ref{def:c2:2509}). 

\begin{lem}\label{lem:comm:gr:2909}
(a) The equivalence $\mathbf{CHA}\leftrightarrow\mathbf{MLA}$ induces an equivalence $\mathbf{CHA}_{fd}\leftrightarrow
\mathbf{MLA}_{fd}$.

(b) The endofunctors $\hat{\mathrm{gr}}$ of $\mathbf{CHA}$ and $\mathbf{MLA}$ induce endofunctors (still denoted 
$\hat{\mathrm{gr}}$) of $\mathbf{CHA}_{fd}$ and $\mathbf{MLA}_{fd}$. 

(c) The category equivalence $\mathbf{CHA}_{fd}\leftrightarrow\mathbf{MLA}_{fd}$
quasi-intertwines the endofunctors $\hat{\mathrm{gr}}$ on both sides.
\end{lem}

\begin{proof}
(a) For $A$ an object in $\mathbf{CHA}$, \eqref{eq:quillen:2509} implies the vector space isomorphism $\mathrm{gr}^1(\mathcal PA)\simeq\mathcal 
P(\mathrm{gr}(A))\cap\mathrm{gr}^1A$; moreover, $\mathrm{gr}(A)$ is a graded connected Hopf algebra, 
therefore $\mathrm{gr}^1(A)\subset \mathcal P(\mathrm{gr}(A))$; therefore $\mathrm{gr}^1(\mathcal PA) \simeq \mathrm{gr}^1A$.  

(b) follows from $\mathrm{gr}^1(\hat{\mathrm{gr}}(A))=\mathrm{gr}^1(A)$ for $A$ a CHA and 
$\mathrm{gr}^1(\hat{\mathrm{gr}}(\mathfrak g))=\mathrm{gr}^1(\mathfrak g)$ for $\mathfrak g$ an MLA. 

(c) follows from Lemma \ref{*:2509}. 
\end{proof}

\begin{lem}\label{lem:1010:2909}
Let $H$ be an object in $\mathbf{CHA}_{fd}$.

(a) There is a natural vector space isomorphism $(\hat{\mathrm{gr}}H)'\simeq\oplus_{n\geq0}\mathrm{gr}^n(H)^*$; it induces a
natural isomorphism  $F_n ((\hat{\mathrm{gr}}H)')\simeq \oplus_{k\leq n} \mathrm{gr}^k(H)^*$ for any 
$n\geq 0$; 

(b) There is a natural isomorphism  
\begin{equation}\label{***:2509}
(\hat{\mathrm{gr}}H)'\simeq \mathrm{gr} (H')    
\end{equation}
in $\mathbf{HA}$.  

(c) The composed isomorphism  
$\oplus_{k\leq n}\mathrm{gr}^k(H)^* 
\simeq F_n ((\hat{\mathrm{gr}}H)') 
\simeq F_n (\mathrm{gr} (H')) 
\simeq \mathrm{gr}_{\leq n} (H')$, where the first isomorphism arises from (a), the 
second from the image by $F_n $ of \eqref{***:2509} and the third isomorphism arises from 
Lem. \ref{lem:10:1:a:2909}(a), is the direct sum over $k\leq n$ of the isomorphisms $\mathrm{gr}^k(H)^*\simeq
\mathrm{gr}_{\leq n} (H')$ arising from Lem. \ref{IDFILT:CHA}. 
\end{lem}

\begin{proof} (a) $\hat{\mathrm{gr}}(H)$ is $\prod_{n\geq 0}\mathrm{gr}^nH$, equipped with the filtration $(\prod_{i\geq n}
\mathrm{gr}^iH)_{n\geq 0}$. It follows that $\hat{\mathrm{gr}}(H)'=\oplus_{n\geq 0}\mathrm{gr}^n(H)^*$.
One then has $F_n (\hat{\mathrm{gr}}(H)')=(F^{n+1}\hat{\mathrm{gr}}(H))^\perp\simeq \oplus_{k\leq n} \mathrm{gr}^k(H)^*$ for any 
$n\geq 0$, where the first equality follows from Lem. \ref{IDFILT:CHA}.  

(b) follows from the sequence of equalities 
\begin{align*}
&\mathrm{gr} (H')=\oplus_{n\geq0}F_{n+1}  H'/F_n  
H'=\oplus_{n\geq0}(F^{n+1}H)^\perp/(F^nH)^\perp=\oplus_{n\geq0}(F^nH/F^{n+1}H)^*
\\ & =\oplus_{n\geq0}\mathrm{gr}^n(H)^*=(\hat{\mathrm{gr}}H)', 
\end{align*}
where the second equality follows from $F_\bullet  H'=(F^{\bullet+1}H)^\perp$ (see Lem. \ref{IDFILT:CHA})
and the last equality follows from (a).

(c) follows from the identification of the said isomorphism with the degree $\leq k$ part of the above sequence of equalities. 
\end{proof}

\begin{lem}\label{lem:SW:1512}
Let $\Gamma$ be group. If $\Gamma$ is filtered formal, then there exists an isomorphism 
\begin{equation}\label{toto:2909}
iso_\Gamma :     (\mathbb C\Gamma)^\wedge \simeq \hat{\mathrm{gr}}((\mathbb C\Gamma)^\wedge)
\end{equation} 
in $\mathbf{CHA}$, such that $\mathrm{gr}(iso_\Gamma)=id$. 
\end{lem}

\begin{proof} Let $\Gamma$ be a filtered formal group. There is an isomorphism  
\begin{equation}\label{iso:FF:2909}
\mathrm{Lie}(\Gamma) \simeq \hat{\mathrm{gr}}\mathrm{Lie}(\Gamma)
\end{equation}
in $\mathbf{MLA}$ with associated graded the identity. There is a sequence of isomorphisms in $\mathbf{CHA}$
given by  
$$
(\mathbb C\Gamma)^\wedge \simeq \hat U(\mathrm{Lie}(\Gamma)) \simeq \hat U(\hat{\mathrm{gr}}\mathrm{Lie}(\Gamma)) 
\simeq \hat{\mathrm{gr}}(\hat U(\mathrm{Lie}(\Gamma)))=\hat{\mathrm{gr}}((\mathbb C\Gamma)^\wedge)
$$
where the first and last isomorphisms come from Lem. \ref{lem:Q:2909} and Def. \ref{def:Lie:2909}, 
the second isomorphism arises from applying the 
functor $\hat U$ to the isomorphism \eqref{iso:FF:2909}, and the third isomorphism arises from Lem. \ref{lem:comm:gr:2909}(c). 
This results in an isomorphism with the claimed properties. 
\end{proof}

\begin{prop}\label{prop:C3:3110}
If $\Gamma$ is a filtered formal finitely generated group, then the Hopf algebra
$(\mathbb C\Gamma)'$ is filtered formal (in the sense of Def. \ref{def:FF:HA:3110}). 
\end{prop}

\begin{proof}
Let $\Gamma$ be a filtered formal finitely generated group. Then \eqref{toto:2909}
is an isomorphism in $\mathbf{CHA}_{fd}$. There is a sequence of isomorphisms in $\mathbf{HA}$ 
\begin{equation}\label{SEQUENCE:2909}
(\mathbb C\Gamma)' \simeq ((\mathbb C\Gamma)^\wedge)' \simeq (\hat{\mathrm{gr}}((\mathbb C\Gamma)^\wedge)))' \simeq 
\mathrm{gr} (((\mathbb C\Gamma)^\wedge)') \simeq \mathrm{gr} ((\mathbb C\Gamma)')
\end{equation}
where the first and last isomorphisms follow from Lem. \ref{lem:5:12:1927}, the second isomorphism  arises from applying the functor 
$\mathbf{CHA}_{fd}\to\mathbf{HA}$, $H\mapsto H'$ (see Def. \ref{def:dual:CHAfd}) to \eqref{toto:2909}, and the third isomorphism comes from 
\eqref{***:2509}. 
It follows from $\mathrm{gr}(iso_\Gamma)=id$ and Lem. \ref{IDFILT:CHA} that the second isomorphism in \eqref{SEQUENCE:2909} is compatible with
the filtration $F _\bullet$ and that the associated graded is the identity. The third isomorphism in \eqref{SEQUENCE:2909} 
induces an isomorphism between the $F _n$ of both sides for any $n\geq0$. By Lem. \ref{lem:1010:2909}(c), the associated 
graded morphism coincides with the composition of natural isomorphisms 
$\mathrm{gr}_n ((\hat{\mathrm{gr}}((\mathbb C\Gamma)^\wedge)))')
\simeq \mathrm{gr}^n(\mathbb C\Gamma)^* \simeq \mathrm{gr}_n (\mathrm{gr} (((\mathbb C\Gamma)^\wedge)'))$.  
This implies that the image by $\mathrm{gr}_n $ of the isomorphism $(\mathbb C\Gamma)' \simeq\mathrm{gr} ((\mathbb C\Gamma)')$
induced by \eqref{SEQUENCE:2909} is the identification $\mathrm{gr}_n (\mathbb C\Gamma)' \simeq
\mathrm{gr}_n \mathrm{gr} ((\mathbb C\Gamma)')$ induced by Lem. \ref{lem:10:1:a:2909}(a). 
\end{proof}

\begin{rem}
Lem. \ref{lem:SW:1512} relates as follows to \cite{SW2}: combining Cors. 6.2 and 2.7 of {\it loc. cit.,} one obtains the equivalence
of the filtered formality of $\Gamma$ with the conclusion of Lem. \ref{lem:SW:1512} for any finitely generated group $\Gamma$.    
\end{rem}

\medskip\noindent{\bf Acknowledgements.} The research of B.E. has been partially funded by ANR grant “Project HighAGT ANR20-CE40-0016”. 
The research 
of F.Z. has been funded by the European Union’s Horizon 2020 research and innovation programme under the Marie Sklodowska-Curie grant agreement 
No 843960 for the project “HIPSAM”, and by the Royal Society,
under the grant URF\textbackslash R1\textbackslash201473. Both authors express their thanks to F. Brown, J. Burgos Gil, J. Fres\'an, R. Hain, N. 
Matthes and E. Panzer for fruitful exchanges on various aspects of this work.


\begin{thebibliography}{MMPPW}

\bibitem[Br]{Br:these}  F. Brown, Multiple zeta values and periods of moduli spaces $\overline{\mathfrak M}_{0,n}$, 
 Ann. Sci. Éc. Norm. Supér. (4) 42 (2009), no. 3, 371--489.

\bibitem[BGF]{BGF} J. Burgos Gil, J. Fresan, Multiple zeta values: from numbers to motives, {\it Clay Mathematics Proceedings,}
to appear.

\bibitem[BDDT1]{BDDT1} J. Broedel, C.  Duhr, F. Dulat, L. Tancredi, Elliptic polylogarithms and iterated integrals on
elliptic curves I: general formalism, J. High Energy Phys. 2018, no. 5, 093.  

\bibitem[BDDT2]{BDDT2}  J. Broedel, C. Duhr, F. Dulat, L. Tancredi, Elliptic polylogarithms and iterated integrals on elliptic curves 
II. An application to the sunrise integral, Physical Review D, Vol. 97, no.11, p. 116009 (2018)

\bibitem[Ch]{Chen} K.-T. Chen, Extension of $C^\infty$ functions by integrals and Malcev completion of $\pi_1$, 
Adv. in Math. 23 (1977), no. 2, 181--210.  

\bibitem[DDMS]{DDMS} M. Deneufchâtel, G. H. E. Duchamp, V. H. N. Minh, and A. I. Solomon, 
Independence of hyperlogarithms over function fields via algebraic combinatorics, pp. 127–139 in
Algebraic informatics (Linz, Austria, 2011), edited by F. Winkler, Lecture Notes in Comput. Sci.
6742, Springer, 2011. 

\bibitem[DGP]{DGP} E. D’Hoker, M.B. Green and B. Pioline, Asymptotics of the $D^8R^4$ genus-two string invariant, Commun. Num. Theor.
Phys. 13, 351--462 (2019). 

\bibitem[DHS]{DHS} E. D’Hoker, M. Hidding and O. Schlotterer,
Constructing polylogarithms on higher-genus Riemann surfaces, arXiv:2306.08644

\bibitem[EZ]{EZ} B. Enriquez, F. Zerbini, Elliptic hyperlogarithms, arXiv:2307.01833. 

\bibitem[Fr]{Fr} B. Fresse, Homotopy of operads and Grothendieck-Teichmüller groups. Part 1.
The algebraic theory and its topological background, Mathematical Surveys and Monographs, 217. 
American Mathematical Society, Providence, RI, 2017.

\bibitem[H]{H} R. Hain, The de Rham homotopy theory of complex algebraic varieties I,  
K-theory 1 (1987), 271--324. 

\bibitem[L]{L} S. Lang, Complex analysis, fourth edition, Graduate Texts in Mathematics, 103. Springer-Verlag, New York, 1999.

\bibitem[LD]{LD} J.A. Lappo-Danilevsky, Mémoires sur la théorie des systèmes des équations différentielles linéaires, Chelsea 
Publishing Co., New York, N. Y., 1953. 

\bibitem[MMPPW]{MMPPW} R. Marzucca, A. McLeod, B. Page, S. P\"ogel and S. Weinzierl, 
Genus Drop in Hyperelliptic Feynman Integrals, arXiv:2307.11497

\bibitem[Ma]{Ma} N. Matthes, On the algebraic structure of iterated integrals of quasimodular forms,  
Algebra Number Theory 11 (2017), no. 9, 2113–2130.

\bibitem[Pa]{Pa:thesis} E. Panzer, Feynman integrals and hyperlogarithms, PhD Thesis, arXiv:1506.07243v1, Humboldt-Universität 
zu Berlin, Mathematisch-Naturwissenschaftliche Fakultät, 2015 

\bibitem[Ph]{Ph} F. Pham, Singularities of integrals. Homology, hyperfunctions and microlocal analysis, Universitext. Springer, 
London; EDP Sciences, Les Ulis, 2011

\bibitem[Po]{Po} H. Poincaré, Sur les groupes des équations linéaires, Acta Math. 4 (1884), no. 1, 201--312.

\bibitem[Q1]{Q} D. Quillen, Rational homotopy theory, Ann. of Math. (2) 90 (1969), 205--295.

\bibitem[Q2]{Q2} D. Quillen, On the associated graded ring of a group ring, J. Algebra 10 (1968), 411--418.

\bibitem[SW1]{SW} A. Suciu, H. Wang, Formality properties of finitely generated groups and Lie algebras, Forum Math. 31 (2019), no. 4, 
867--905.

\bibitem[SW2]{SW2} A. Suciu, H. Wang, Taylor expansions of groups and filtered-formality, European Journal of Mathematics, 6 (2020), 
1073--1096.  
\end{thebibliography}
\end{document}